\documentclass{amsart}
\usepackage{amssymb}
\usepackage{eurosym}
\usepackage{amsfonts}
\usepackage{amsmath}
\usepackage{amstext}
\usepackage{graphicx}

\newtheorem{theorem}{Theorem}

\newtheorem{corollary}[theorem]{Corollary}

\newtheorem{definition}[theorem]{Definition}
\newtheorem{example}[theorem]{Example}

\newtheorem{lemma}[theorem]{Lemma}

\newtheorem{proposition}[theorem]{Proposition}
\newtheorem{remark}[theorem]{Remark}

\input{tcilatex}

\begin{document}
\title[Statistical properties of dynamics Functional analytic approach ]{%
Statistical properties of dynamics\\
Introduction to the functional analytic approach.}
\author{Stefano Galatolo}
\email{galatolo@dm.unipi.it}
\maketitle

\begin{abstract}
These are lecture notes for a simple minicourse approaching the satistical
properties of a dynamical system by the study of the associated transfer
operator (considered on a suitable functions or measures spaces).

The following questions will be addressed:

\begin{itemize}
\item existence of a regular invariant measure;

\item Lasota Yorke inequalities and spectral gap;

\item decay of correlations and some limit theorem;

\item stability under perturbations of the system

\item linear response

\item random systems

\item hyperbolic systems
\end{itemize}

The point of view taken is to present the general construction and ideas
needed to obtain these results in the simplest way. For this, some theorem
is proved in a form which is weaker than usually known, but with an
elementary and simple proof.\newline
\end{abstract}

\tableofcontents
\address{Dipartimento di Matematica, Universita di Pisa, Via \ Buonarroti
1,Pisa - Italy}

\section{Introduction}

The term \emph{statistical properties} of a dynamical system refers to the
long time behavior of a trajectory $x_{1},...,x_{n}$, or a set of
trajectories of the system: their distribution in the phase space, the
average of a given observable along the trajectory (consider a function $f$
with values on $\mathbb{R}$ or $\mathbb{C}$ and the time average $\frac{%
f(x_{1})+...+f(x_{n})}{n}$), the speed of convergence to those averages, the
frequency of a deviation from the average behavior and so on. As we will see
in the following, this relates to the properties of the evolution of large
sets of trajectories, measures or even distributions by the action of the
dynamics.

In chaotic systems the statistical properties of the dynamics are often a
better object to be studied than the pointwise behavior of trajectories.
Indeed, due to the initial condition sensitivity, the future behavior of
initial data can be unstable and unpredictable, but statistical properties
are often stable and their description simpler. This is a classical approach
to dynamics that has been implemented in the so called Ergodic Theory (the
reader may find in any library or even in the references to these notes very
good books on this general theory and about its \ many applications).

In Ergodic Theory often it is supposed that the dynamics preserves a given
measure, and some other properties (ergodicity, mixing) from this, deep
consequences on the statistical behavior of the system are deduced. But,
given a dynamical system, how to show that there are interesting invariant
measures, which are their properties and how to show that the resulting
measure preserving system is ergodic or mixing and how to estimate the
mixing speed? Furthermore, are these findings stable by perturbations of the
system? Can, all of this be computed numerically in a reliable way?

These notes focuses around some questions of this kind, between dynamics and
ergodic theory. We consider dynamical systems having certain geometrical
properties and show that they have "good" invariant measures. Sometime we
can prove that they satisfy other finer properties (mixing, fast decay of
correlations, spectral gap), so that we can apply many results from Ergodic
Theory and Probability to deduce other consequences. We also consider the
important problem of the stability of these invariant measures and mixing
properties, this allows to have information about whether the statistical
properties are stable under small changes in the system or not, or even
about the direction of change of these properties when the system changes.
(this has of course important applications for the understanding of the
behavior of many systems)

As said before, the properties we mean to investigate are related to the
evolution of measures by the action of the dynamics. We will see that given
a dynamical system it is possible to associate to the system a transfer
operator describing the action of the dynamics on suitable functional spaces
of measures (or distributions sometime). Many important results can be
obtained studying the properties of this transfer operator\footnote{%
From an historical point of view, this approach brought important
developments and applications starting around the beginning of 21 century
(see \cite{S},\cite{L2},\cite{S}) and it is still fastly developing.}. This
is the main subject of the following sections. We will start defining the
transfer operator and its basic properties. We will see how it is possible
to deduce the existence of a regular invariant measure and the speed of
convergence to this measure by the iteration of the dynamics. We will see
that under additional assumptions we may have a precise description of the
action of the transfer operator on suitable spaces of measures (spectral
gap, Section \ref{spg}), and some of its statistical consequences. We will
then consider the problem of stability of all these concepts under
perturbation (Section \ref{stab}).

The general theory and tools shown in the notes are applied to some of the
simplest kind of dynamical systems. We enter in details showing how all
these concepts can be applied to\emph{\ expanding maps}, but we also show
how the same approach can be applied to \emph{\ piecewise expanding maps}
(Section \ref{PW}) and some class of \emph{hyperbolic maps} (Section \ref{yp}%
). Here the technicalities needed are much more complicated, but we try to
give a sketch of the ideas needed to extend the transfer operator approach
to this case, in which there is both contraction and expansion in the map
generating the dynamics. We also see how the transfer operator approach can
be applied to\emph{\ random dynamical systems} (Section \ref{rnd}). \newline

This version of the notes was prepared as a support for the \ Hokkaido
University summer institute school 2017, 2021\textbf{\ }and the course "An
introduction to random dynamical systems and their
perturbations\textquotedblright\ organized by Centro de Giorgi and Scuola
Superiore S. Anna (Fondazione Conservatorio Santa Chiara, 2019).

\section{Physical measures}

In these section we introduce the notion of physical measure for
deterministic, discrete time dynamical systems. Let $X$ be a metric space, $%
T:X\mapsto X$ a Borel measurable map. Given some initial condition $x_{0}\in
X$ we define the orbit of $x_{0}$ by the dynamics $T$ as $%
x_{1}=T(x_{0}),x_{2}=T(x_{1}),...,x_{i}=T^{i}(x_{0})$. Now let $M(X)$ be the
set of Borel positive measures on $X$. Lus see how we can apply the dynamics
to measures instead of points, by considering the so called pushforward map $%
T_{\ast }:M(X)\rightarrow M(X).$ Given a Borel probability measure $\nu $ on 
$X,$ we define $T^{\ast }(\nu )$ as%
\begin{equation*}
T_{\ast }[\nu ](A)=\nu (T^{-1}(A)).
\end{equation*}%
where $A$ is any measurable subset of $X.$ \footnote{%
To help in the understanding of why this is a reasonable definition of the
action of the dynamics $T$ on measures, we invite the reader to verify that
by the above definition, if $\delta _{x_{0}}$ is the Dirac measure placed on 
$x_{0}$ then $T^{\ast }(\delta _{x_{0}})=\delta _{T(x_{0})}.$}

We say that a Borel probability measure $\mu $ is $T$\textbf{-invariant} if $%
T_{\ast }[\mu ]=\mu $, or equivalently, if for each measurable set $A$ it
holds $\mu (A)=\mu (T^{-1}(A))$. \ In this framework, when $T:X\rightarrow X$
and $\mu $ is invariant, we call the triple $(X,T,\mu )$ \ \emph{measure
preserving transformation.}

\begin{example}
For a rotation $x\rightarrow x+\alpha ~(\func{mod}1)$ the Lebesgue measure
(the 1-d volume) is invariant.
\end{example}

\begin{example}
For the doubling map $x\rightarrow 2x~(\func{mod}1)$ the Lebesgue measure
(the 1-d volume) is invariant, but there are many other invariant measures.
\end{example}

Invariant measures represent \emph{equilibrium} states, in the sense that
probabilities of events do not change in time. A given a map $T$, may have
many of these invariant measures\footnote{%
For example, to every periodic orbit it correspond one invariant measure for
the system, but as we will see in the following there can be are many more
in the system.}, but some of them is particularly important to describe the
statistical properties of the dynamics associated to $T$.

In this section we will define the notion of Physical measure, which is a
particularly important kind of invariant measure. In the following, we will
see that under suitable assumptions some physical measure measure is an
attractor of many other regular measures by the dynamics induced on $M(X)$
by $T_{\ast }$, and the speed of convergence to this equilibrium state has
important consequences for the statistical properties of the dynamics.

Given a measure preseving transformation $(X,T,\mu )$, a set $A\subseteq X$
is called $T$-invariant if $T^{-1}(A)=A$ up to zero measure sets. The system 
$(X,T,\mu )$ is said to be ergodic if each $T$-invariant set has total or
null measure. An ergodic system is then a system whose dynamics is
indecomposable into different invariant sets (up to zero measure).

The celebrated\textbf{\ Birkhoff pointwise ergodic theorem} (see any book
about ergodic theory) says that in this case, time averages computed along $%
\mu $-typical orbits coincides with space average with respect to $\mu .$
More precisely, in ergodic systems, for any $f\in L^{1}(X,\mu )$ it holds 
\begin{equation}
\underset{n\rightarrow \infty }{\lim }\frac{S_{n}^{f}(x)}{n}=\int \!{f}\,%
\mathrm{d}{\mu },  \label{Birkhoff}
\end{equation}%
for $\mu $ almost each $x$, where $S_{n}^{f}=f+f\circ T+\ldots +f\circ
T^{n-1}.$

Note that the equality in (\ref{Birkhoff}) is up to negligible sets
according to $\mu $. A given map $T:X\mapsto X$ may hence have many
invariant measures corresponding to many possible statistical limit
behaviors for its orbits.

It is important to select the physically relevant ones; the ones which come
from the time averages of a large set of points. Large according to the
natural measure we can consider a priori on our phase space; when $X$ is a
manifold this could be the Riemann volume or Lebesgue measure.

\begin{definition}
Let $X$ be a manifold with boundary. We say that a point $x$ belongs to the
basin of an invariant measure $\mu $ if (\ref{Birkhoff}) holds at $x$ for
each bounded continuous $f$. A \textbf{physical measure} is an invariant
measure whose basin has positive Lebesgue measure.
\end{definition}

Often these physical measures also have other interesting features such as:

\begin{itemize}
\item they are as regular as possible among the invariant ones;

\item they have a certain stability under perturbations of the system;

\item they are in some sense limits of iterates of the Lebesgue measure (by
iterating the pushforward map or taking suitable Cesaro averages of such
iterates).
\end{itemize}

These measures hence encodes important information about the statistical
behavior of the system (see \cite{LSY} for a general survey). In the
following we will see some method to select those measures, prove their
existence and some of its main statistical properties.

\begin{remark}
If $\mu $ is ergodic and physical, then the correspondence between time and
space averages expressed by \ref{Birkhoff} also hold for a set which is also
relevant for the Lebesgue measure.
\end{remark}

\begin{remark}
In the following we will show techniques to find the physical invariant
measures of a given system and to know many intresting properties of them.
Altough in many systems of interest in the mathematics and in the
applications one can show that such invariant measures exist, it is worth to
remark that there are deterministic systems having no invariant meaures at
all. A remarkable general statement about existence of invariant measures
(Krylov-Bogoliubov theorem) says that "if $X$ is a compact metric space $%
T:X\rightarrow X$ is continuous then there is at least one invariant
measure", we now show a system which is discontinuous and does not have such
measures. Consider $T:[0,1]\rightarrow \lbrack 0,1]$ defined as 
\begin{equation*}
T(x)=\left\{ 
\begin{array}{c}
\frac{1}{2}x+\frac{1}{4}~if~x\neq \frac{1}{2} \\ 
0~~if~x=\frac{1}{2}%
\end{array}%
\right. ,
\end{equation*}%
whose graph is represented in the following figure.

\begin{equation*}
\FRAME{itbpF}{1.0836in}{1.0793in}{0in}{}{}{Plot}{\special{language
"Scientific Word";type "MAPLEPLOT";width 1.0836in;height 1.0793in;depth
0in;display "USEDEF";plot_snapshots TRUE;mustRecompute FALSE;lastEngine
"MuPAD";xmin "0.0";xmax "1";xviewmin "0.0";xviewmax "1";yviewmin
"0";yviewmax "1";viewset"XY";rangeset"X";plottype 4;axesFont "Times New
Roman,12,0000000000,useDefault,normal";numpoints 100;plotstyle
"patch";axesstyle "normal";axestips FALSE;xis \TEXUX{x};var1name
\TEXUX{$x$};function \TEXUX{$\frac{1}{2}x+\frac{1}{4}$};linecolor
"black";linestyle 1;pointstyle "point";linethickness 2;lineAttributes
"Solid";var1range "0.0,1";num-x-gridlines 100;curveColor
"[flat::RGB:0000000000]";curveStyle "Line";function \TEXUX{$x$};linecolor
"green";linestyle 1;pointstyle "point";linethickness 1;lineAttributes
"Solid";var1range "0.0,1";num-x-gridlines 100;curveColor
"[flat::RGB:0x00008000]";curveStyle "Line";function \TEXUX{$0$};linecolor
"black";linestyle 1;pointplot TRUE;pointstyle "point";linethickness
2;lineAttributes "Solid";var1range "0.5,0.5";num-x-gridlines 100;curveColor
"[flat::RGB:0000000000]";curveStyle "Point";rangeset"X";function
\TEXUX{$\frac{1}{2}$};linecolor "black";linestyle 1;pointplot
TRUE;pointstyle "circle";linethickness 2;lineAttributes "Solid";var1range
"0.5,0.5";num-x-gridlines 100;curveColor "[flat::RGB:0000000000]";curveStyle
"Point";rangeset"X";VCamFile 'R8QG5I00.xvz';valid_file "T";tempfilename
'R8QG5I00.wmf';tempfile-properties "XPR";}}
\end{equation*}

This map has no invariant measures, let us briefly explain why. Let $A=[0,%
\frac{1}{2})\cup (\frac{1}{2},1]$, \ it holds that $\forall \epsilon
~\exists n~s.t.T^{n}(A)\subseteq (\frac{1}{2}-\epsilon ,\frac{1}{2}+\epsilon
)-\{\frac{1}{2}\}:=S_{\epsilon }$.

Suppose $\mu $ is an invariant measure, since the map is $1-1$ on $A$ we get 
$\mu (T^{n}A)=\mu (A)=\mu ([0,1]-\{\frac{1}{2}\}).$ It must also hold that $%
\mu (S_{\epsilon })\rightarrow 0$, hence $\mu (A)=0.$ However since $T^{-1}(%
\frac{1}{2})=\emptyset $ we also have $\mu (0)=0$ and thus the whole space
has zero measure.

\begin{example}
We show two examples of maps of the unit interval with a plot of the density
of their associated absolutely continuous and hence physical invariant
measures, computed with the methods described in Section \ref{Ul}.

The first example is a piecewise expanding map (see Section \ref{PW}) $%
T:[0,1]\rightarrow \lbrack 0,1]$ defined by%
\begin{equation*}
T(x)=\left\{ 
\begin{array}{cc}
\frac{109}{64}(1/2-x)^{51/64} & x<1/2 \\ 
1-\frac{109}{64}(x-1/2)^{51/64} & x>1/2.%
\end{array}%
\right.
\end{equation*}%
The graph of the map and the associated invariant density is shown in figure %
\ref{o1}.
\end{example}
\end{remark}

\FRAME{ftbpFU}{2.9813in}{1.5384in}{0pt}{\Qcb{A piecewise expanding map and
its invariant density.}}{\Qlb{o1}}{lorenz2.bmp}{\special{language
"Scientific Word";type "GRAPHIC";maintain-aspect-ratio TRUE;display
"USEDEF";valid_file "F";width 2.9813in;height 1.5384in;depth
0pt;original-width 5.3004in;original-height 2.7198in;cropleft "0";croptop
"1";cropright "1";cropbottom "0";filename
'../../Congressi-scuole-workshop/Hokkaido-2021/Slides/Mon/lorenz2.bmp';file-properties "XNPEU";}%
}

The second example is another map on the interval which has a generally
expanding behavior, but also a fixed point where the derivative is $1$. This
is a slowly repelling fixed point, and this fact forces the invariant
measure to concentrate around that point, givig an invariant density which
diverges around the point. Again we consider a map $T:[0,1]\rightarrow
\lbrack 0,1]$ defined by%
\begin{equation}
T(x)=x+x^{1+\frac{1}{8}}\quad \text{mod $1$,}  \label{eq:mann}
\end{equation}

whose graph is shown in Figure\ 2.%
\begin{equation*}
\FRAME{itbpFUX}{1.0585in}{1.0286in}{0in}{\Qcb{Figure 2.}}{\Qlb{f2}}{Plot}{%
\special{language "Scientific Word";type "MAPLEPLOT";width 1.0585in;height
1.0286in;depth 0in;display "USEDEF";plot_snapshots TRUE;mustRecompute
FALSE;lastEngine "MuPAD";xmin "0";xmax "1";xviewmin "0";xviewmax
"1";yviewmin "0";yviewmax "1";viewset"XY";rangeset"X";plottype 4;axesFont
"Times New Roman,12,0000000000,useDefault,normal";numpoints 100;plotstyle
"patch";axesstyle "normal";axestips FALSE;xis \TEXUX{x};var1name
\TEXUX{$x$};function \TEXUX{$x+x^{2}-\left\lfloor x+x^{2}\right\rfloor
$};linecolor "black";linestyle 1;pointstyle "point";linethickness
1;lineAttributes "Solid";var1range "0,1";num-x-gridlines 100;curveColor
"[flat::RGB:0000000000]";curveStyle "Line";function \TEXUX{$x$};linecolor
"green";linestyle 1;pointstyle "point";linethickness 1;lineAttributes
"Solid";var1range "0,1";num-x-gridlines 100;curveColor
"[flat::RGB:0x00008000]";curveStyle "Line";VCamFile
'R8WOG80N.xvz';valid_file "T";tempfilename
'R8WOUG06.wmf';tempfile-properties "XPR";}}
\end{equation*}
\ and a plot of the invariant density is shown in Figure 3.\FRAME{dtbpFU}{%
1.3342in}{1.3342in}{0pt}{\Qcb{Figure 3.}}{\Qlb{f3}}{%
mann_0125_dens-eps-converted-to.bmp}{\special{language "Scientific
Word";type "GRAPHIC";maintain-aspect-ratio TRUE;display "USEDEF";valid_file
"F";width 1.3342in;height 1.3342in;depth 0pt;original-width
2.2657in;original-height 2.2657in;cropleft "0";croptop "1";cropright
"1";cropbottom "0";filename
'mann_0125_dens-eps-converted-to.bmp';file-properties "XNPEU";}}

\section{The transfer operator}

We have seen how the dynamics $T$ can be applied to measures by the
pusforward map. Let us consider the space $SM(X)$ of Borel measures with
sign on $X$ (equivalently complex valued measures can be considered) this is
a vector space and the pushforward map $T_{\ast }$ can be extended to $SM(X)$
with the same definition to a \emph{linear map} $L:SM(X)\rightarrow SM(X)$.
Because of its importance and to emphatize the fact that this is a linear
function we will denote it by $L$ or $L_{T}$ and call it the \textbf{%
transfer operator} associated to $T$. Recalling the definition, we remark
that if $\nu \in SM(X)$ then $L[\nu ]\in SM(X)$ is such that%
\begin{equation*}
L[\nu ](A)=\nu (T^{-1}(A)).
\end{equation*}

The main theme of these lectures is to see how the understanding of the
properties of this operator, applied on suitable normed vector spaces of
measures allow to understand many important statistical properties of the
dynamics of $T$.

We will see this by dealing first with some of the simplest examples, the
expanding maps for which we will prove the existence of physical measures
having a smooth density, the convergence to equilibrium and statistical
stability properties. To do this we will hence restrict our interest on
absolutely continuous measures, and until Section \ref{sec:contr} we will
only consider this kind of measures and the related densities, also on other
more complicated classes of of deterministic and random systems. We then now
focus our attention on the understanding of the properties of the transfer
operator applied to absolutely continuous measures.

Remark that if the measure we consider is absolutely continuous: $d\nu =f~dm$
(here we are considering the Lebesgue measure $m$ as a reference measure,
note that other measures can be considered) and if \thinspace $T$ is
nonsingular\footnote{%
A map is nonsingular (with respect to the Lebesgue measure) when $%
m(T^{-1}(A))=0$ $\iff m(A)=0.$} the operator $L$ induces another operator $%
\tilde{L}:L^{1}(m)\rightarrow L^{1}(m)$ acting on the measure densities
defined by 
\begin{equation*}
\tilde{L}f=\frac{d(L(f~m))}{dm}.
\end{equation*}
By a small abuse of notation we will still indicate by $L$ this operator%
\footnote{%
In many applications the transfer operator will be considered as acting on
some suitable space of regular measures with sign hence other subspaces of $%
L^{1}(m)$ or $SM(X)$.}. We remark that in this point of view, one can see $%
L^{1}(m)$ equivalently \ as a space of integrable funtions or a space of
measures which are absolutely continuous with respect to $m,$ hence measures
having some regularity.

Considering hence $L:L^{1}\rightarrow L^{1}$, it is easy to verify that this
is a \emph{positive operator\footnote{%
A positive operator $L$ is an operator for which $f\geq 0\implies Lf\geq 0$.
We remark that by linearity one also has $f\geq g\implies Lf\geq Lg.$} }and
preserves the integral

\begin{equation*}
\int_{X}Lf~dm=\int_{X}f~dm.
\end{equation*}%
We recall that positive, integral preserving operators are also called \emph{%
Markov Operators. }Let us see some other important basic properties

\begin{proposition}
\label{weakcontr} $L:L^{1}\rightarrow L^{1}$ is a \emph{weak contraction}
for the $L^{1}$ norm. If $f$ is a $L^{1}$ density, then 
\begin{equation*}
||Lf||_{1}\leq ||f||_{1}.
\end{equation*}
\end{proposition}

\begin{proof}
Since $L$ preserves the integral%
\begin{eqnarray*}
||Lf||_{1} &=&\int |Lf|~dm\leq \int |L(f^{+}-f^{-})|dm\leq \int
|L(f^{+})|+|L(f^{-})|dm \\
&\leq &\int |(f^{+})|+|(f^{-})|dm=\int |f|~dm=||f||_{1}.
\end{eqnarray*}
\end{proof}

\begin{proposition}
\label{duality}Consider $f$ $\in L^{1}(m)$, and $g\in L^{\infty }(m)$, then:%
\begin{equation*}
\int g~L(f)~dm=\int g\circ T~f~dm.
\end{equation*}
\end{proposition}

\begin{proof}
Let us first prove it for simple functions if $g=1_{B}$ then%
\begin{eqnarray*}
\int g\circ T~f~dm &=&\int 1_{B}\circ T~f~dm= \\
\int 1_{T^{-1}B}~~fdm &=&\int_{T^{-1}B}f~dm= \\
\int_{B}Lf~ &=&\int 1_{B}~Lf~dm.
\end{eqnarray*}%
If $g\in L^{\infty }$ we can approximate it by a combination of simple
functions $\hat{g}=\sum_{i}a_{i}1_{A_{i}}$ in a way that $||g-\hat{g}%
||_{\infty }\leq \epsilon $ and%
\begin{equation*}
\int \hat{g}\circ T~f~dm=\int \hat{g}~L(f)~dm.
\end{equation*}%
Then%
\begin{eqnarray*}
\int g\circ T~f~dm &=&\int [g-\hat{g}+\hat{g}]\circ T~f~dm \\
&=&\int [g-\hat{g}]\circ T~f~dm+\int \hat{g}\circ T~f~dm.
\end{eqnarray*}%
Moreover%
\begin{eqnarray*}
\int g~L(f)~dm &=&\int [g-\hat{g}+\hat{g}]~L(f)~dm \\
&=&\int [g-\hat{g}]~L(f)~dm+\int \hat{g}~L(f)~dm
\end{eqnarray*}%
and since $T$ is nonsingular 
\begin{eqnarray*}
|\int [g-\hat{g}]\circ T~f~dm| &\leq &||[g-\hat{g}]\circ T||_{\infty
}||f||_{1} \\
&\leq &||[g-\hat{g}]||_{\infty }||f||_{1}\leq \epsilon ,
\end{eqnarray*}%
moreover%
\begin{equation*}
|\int [g-\hat{g}]~L(f)~dm|\leq ||[g-\hat{g}]||_{\infty }||Lf||_{1}\leq
\epsilon
\end{equation*}%
directly leading to the statement.
\end{proof}

\begin{remark}
A similar statement 
\begin{equation*}
\int g~d(L\mu )=\int g\circ T~d\mu .
\end{equation*}
applies in the more general case $\mu $ is any Borel measure with sign. The
proof is almost the same as above.
\end{remark}

Measures which are invariant for $T$ are fixed points of $L$. Since physical
measures usually have some "as good as possible" regularity property we will
find such invariant measures in some space of "regular" measures. A first
example which will be explained in more details below is the one of
expanding maps, where we are going to find physical measures in the space of
invariant measures having an absolutely continuous density.

\section{Expanding maps: regularizing action of the transfer operator and
existence of a regular invariant measure\label{maps}}


\begin{figure}[!ht]
\centering
\includegraphics[width=30mm]{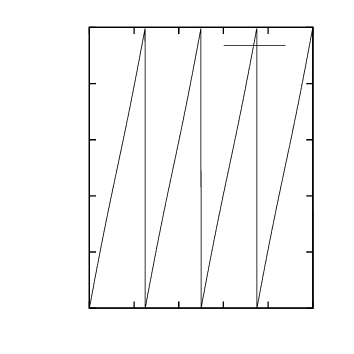}
\hspace{3mm}
\includegraphics[width=30mm]{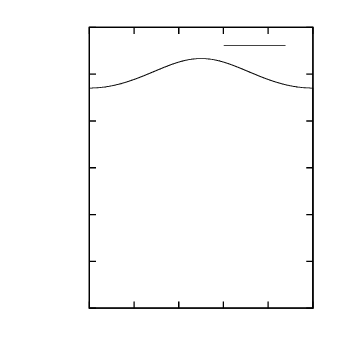}
\caption{The expanding map $x\rightarrow
4x+0.01\sin (8\protect\pi x)$ $\func{mod}1$ and a plot of its
invariant density.}
\end{figure}

In this section we illustrate one approach which allows to prove the
existence of regular invariant measures. The approach is quite general, but
we will show it on a class of one dimensional maps where the construction is
technically simple. An important step is to find a suitable function space
on which the transfer operator has good properties.

Let us consider a map $T$ which is expanding on the circle. i.e.

\begin{itemize}
\item $T:S^{1}\rightarrow S^{1},$

\item $T\in C^{2},$

\item $|T^{\prime }(x)|>1$ $\forall x$.
\end{itemize}

Let us consider the Banach space $W^{1,1}$ of absolutely continuous density
functions\footnote{%
For which $~f(x)=f(0)+\int_{0}^{x}f^{^{\prime }}(t)~dt$ for some $f^{\prime
}\in L^{1}$.} with the norm 
\begin{equation*}
||f||_{W^{1,1}}=||f||_{1}+||f^{\prime }||_{1}.
\end{equation*}

We will show that the transfer operator is \emph{regularizing} for the $%
||~||_{W^{1,1}}$ norm. This implies that iterates of a starting measure have
bounded $||~||_{W^{1,1}}$ norm, allowing to find a suitable invariant
measure (and much more information on the statistical behavior of the
system, as it will be described in the following sections).

\subsection{Lasota-Yorke inequalities\label{sec:LY}}

A main tool to implement this idea is the so called Lasota Yorke inequality (%
\cite{LY}) \footnote{%
In the probabilistic context, this kind of estimations are often called as
Doeblin Fortet inequalities.}, let us see what it is about: we consider the
operator $L$ restricted to some normed vector space \ of signed Borel
measures $(B_{s},||~||_{s})$ (often a Banach space) and we consider another
space $B_{w}\supset B_{s}$ equipped with a weaker norm $||~||_{w}$ such that 
$||L^{n}||_{B_{w}\rightarrow B_{w}}\leq M$ is uniformly bounded (as it is
for the $L^{1}$ norm, see Remark \ref{weakcontr}). In this context, if the
two spaces are well chosen it is possible in many interesting cases to prove
that there are $A\geq 0,0\leq \lambda <1$ such that for each $n$%
\begin{equation}
||L^{n}g||_{s}\leq A\lambda ^{n}||g||_{s}+B||g||_{w}.  \label{lytyp}
\end{equation}%
This means that the iterates $L^{n}g$ have bounded strong norm $||~||_{s}$
and then by suitable compactness arguments which will be shown in the next
paragraphs this inequality may show the existence of an invariant measure in 
$B_{s}$. Similar inequalities can be proved in many systems, and they are a
main tool for the study of statistical properties of dynamical systems.

Now let us see how the above inequality can be obtained in the case of
expading maps, were $W^{1,1}$ and $L^{1}$ are considered as strong and weak
space. Let us consider a nonsingular transformation and see its action on
densities. In the case of expanding maps we are considering we have an
explicit formula for the transfer operator\footnote{%
To prove the formula consider a small neighborhood $B(x,\epsilon )$ of $x$
and check the amount of measure which is sent there by $L$. The value of the
density $Lf(x)$ will a.e. be the value of the limit $\lim_{\epsilon
\rightarrow 0}\frac{\mu (T^{-1}(B(x,\epsilon )))}{2\epsilon }$ (or see \cite%
{gora} pag. 85 ).}:%
\begin{equation}
\lbrack Lf](x)=\sum_{y\in T^{-1}(x)}\frac{f(y)}{|T^{\prime }(y)|}.
\label{pf}
\end{equation}%
Taking the derivative of (\ref{pf}) (remember that $T^{\prime }(y)=T^{\prime
}(T^{(-1)}(x))$) we get%
\begin{equation*}
(Lf)^{^{\prime }}=\sum_{y\in T^{-1}(x)}\frac{1}{(T^{\prime }(y))^{2}}%
f^{\prime }(y)-\frac{T^{\prime \prime }(y)}{(T^{\prime }(y))^{3}}f(y).
\end{equation*}

Note that%
\begin{equation}
(Lf)^{^{\prime }}=L(\frac{1}{T^{\prime }}f^{\prime })-L(\frac{T^{\prime
\prime }}{(T^{^{\prime }})^{2}}f)  \label{preLY}
\end{equation}%
and%
\begin{eqnarray*}
||(Lf)^{^{\prime }}||_{1} &\leq &||\frac{1}{T^{\prime }}f^{\prime }||_{1}+||%
\frac{T^{\prime \prime }}{(T^{^{\prime }})^{2}}f||_{1} \\
&\leq &\alpha ||f^{\prime }||_{1}+||\frac{T^{\prime \prime }}{(T^{^{\prime
}})^{2}}||_{\infty }||f||_{1}
\end{eqnarray*}

where $\alpha =\max (\frac{1}{T^{\prime }})$.

Hence%
\begin{equation*}
||(Lf)^{^{\prime }}||_{1}+||Lf||_{1}\leq \alpha ||f^{\prime }||_{1}+\alpha
||f||_{1}+(||\frac{T^{\prime \prime }}{(T^{^{\prime }})^{2}}||_{\infty
}+1)||f||_{1}
\end{equation*}%
and%
\begin{equation*}
||(Lf)||_{W^{1,1}}\leq \alpha ||f||_{W^{1,1}}+(||\frac{T^{\prime \prime }}{%
(T^{^{\prime }})^{2}}||_{\infty }+1)||f||_{1}.
\end{equation*}

Iterating the inequality\footnote{$||L^{2}f||\leq \alpha ||Lf||+B||Lf||_{1}$%
\par
$\leq \alpha ^{2}||f||+\alpha B||f||_{1}+B||f||_{1},...$}%
\begin{equation}
||(L^{n}f)||_{W^{1,1}}\leq \alpha ^{n}||f||_{W^{1,1}}+\frac{(||\frac{%
T^{\prime \prime }}{(T^{^{\prime }})^{2}}||_{\infty }+1)}{1-\alpha }%
||f||_{1}.  \label{LY}
\end{equation}

Hence if we start with $f\in W^{1,1}$ all the elements of the sequence $%
L^{n}f^{^{\prime }}$ of iterates of $f$ are in $W^{1,1}$ and their strong
norms are uniformly bounded.

\subsubsection{A Lipschitz Lasota Yorke inequality}

By Equation \ref{LY}, $||(L^{n}f)||_{W^{1,1}}$ is uniformly bounded, then
also $||(L^{n}f)||_{\infty }$ is. \ Let us remark that since the transfer
operator is positive 
\begin{equation*}
M:=\sup_{n,||f||_{\infty }\leq 1}||(L^{n}f)||_{\infty
}=\sup_{n}||(L^{n}1)||_{\infty }.
\end{equation*}%
Then there is $n_{1}$ such that $\alpha ^{n_{1}}M<1$.

Now consider a new map, $T_{2}=T^{n_{1}}$. This map still has the same
regularity properties as before and is uniformly expanding on the circle.
Let $L_{2}$ be its transfer operator. From (\ref{preLY}) \ and rhe
positivity of $L_{2}$ we have%
\begin{eqnarray}
||(L_{2}f)^{^{\prime }}||_{\infty } &\leq &||L_{2}(\frac{1}{T_{2}^{\prime }}%
f^{\prime })||_{\infty }+||L_{2}(\frac{T_{2}^{\prime \prime }}{%
(T_{2}^{^{\prime }})^{2}}f)||_{\infty } \\
&\leq &\alpha ^{n_{1}}M||f^{\prime }||_{\infty }+M||\frac{T_{2}^{\prime
\prime }}{(T_{2}^{^{\prime }})^{2}}||_{\infty }||f||_{\infty }.
\end{eqnarray}

Then $L_{2}$ satisfies a Lasota Yorke inequality, with the norms $%
||~||_{Lip} $ defined as $||f||_{Lip}=||f^{\prime }||_{\infty
}+||f||_{\infty }$ and $||~||_{\infty }$, that is%
\begin{equation*}
||L^{nn_{1}}f||_{Lip}\leq \lambda ^{n}||f||_{Lip}+B||f||_{\infty }
\end{equation*}%
with $\lambda =\alpha ^{n_{1}}M<1$ and $B\geq 1$. By this%
\begin{eqnarray}
||L^{nn_{1}+q}f||_{Lip} &\leq &\lambda
^{n}||L^{q}f||_{Lip}+B||L^{q}f||_{\infty }  \label{LYlip} \\
&\leq &\lambda ^{n}M||f||_{Lip}+BM||f||_{\infty }  \notag
\end{eqnarray}

and then also $L$ satisfies a Lasota Yorke inequality of the type $(\ref%
{lytyp})$ with these norms.

\begin{remark}
In this section we have shown two examples of regularization estimations on
different function spaces. This kind of estimations are possible over many
kind of systems having some uniform contracting/expanding behavior. In these
cases the choice of the good measure spaces involved is crucial. When the
system is expanding, even on higher dimension and with low regularity,
spaces of bounded variation functions and absolutely continuous measures are
usually considered. If the system has contracting directions, the physical
measure usually has fractal support and it is often included in some
suitable space on which a Lasota Yorke estimation can be proved (see. e.g. 
\cite{L2},\cite{B},\cite{gora},\cite{BT}, \cite{GL}), in these cases it is
often useful to include the space of measures in a suitable distribution
spaces. In Section \ref{yp} we will see an example of spaces adapted to a
class of hyperbolic systems.
\end{remark}

\subsection{Existence of a regular invariant measure\label{ext}}

In this section we prove the existence of an absolutely continuous invariant
measure for the expanding maps. The idea is to iterate the transfer operator
on the space $W^{1,1}$, the Lasota Yorke inequality will ensure that the
iterates of some initial smooth density remain of bounded $W^{1,1}$ norm.
The following theorem provides a compactness argument to prove that from
this sequence of iterates there is a subsequence converging in $L^{1}$. (see 
\cite{RC} for more details and generalizations)

\begin{proposition}[Rellich-Kondrachov]
\label{RK}$W^{1,1}$ is compactly immersed in $L^{1}$. If $B\subset W^{1,1}$
is a strongly bounded set: $B\subseteq B_{W^{1,1}}(0,K)\footnote{%
The strong ball centered in $0$ with radius $K$.}$ then for each $\epsilon ,$
$B$ has a finite $\epsilon $-net for the $L^{1}$ topology.

In particular any bounded subsequence $f_{n}\in W^{1,1}$ has a weakly
converging subsequence. There is $f_{n_{k}}$ and $f\in L^{1}$ such that%
\begin{equation*}
f_{n_{k}}\rightarrow f
\end{equation*}%
in $L^{1}$.
\end{proposition}

\begin{proof}
(sketch)\ Let us consider a subdivision $x_{1},...,x_{m}$ of $I$ with step $%
\epsilon .$ Let $f\in B_{W^{1,1}}(0,K),$let us consider $\pi _{n}f$ to be
the piecewise linear approximation of $f$ such that%
\begin{equation*}
\pi _{n}f(x_{i})=f(x_{i}).
\end{equation*}

Now, if $x_{i}\leq \overline{x}\leq x_{i+1}$ then 
\begin{equation*}
|f_{n}(\overline{x})-\pi _{n}f(\overline{x})|\leq
\int_{x_{i}}^{x_{i+1}}|f_{n}^{^{\prime }}(t)|~dt:=h_{i}.
\end{equation*}%
Remark that $\sum h_{i}\leq K.$ Hence%
\begin{equation}
||f_{n}-\pi _{n}f||_{1}\leq \epsilon \sum h_{i}\leq \epsilon K.  \label{w1}
\end{equation}

Since $\pi _{n}$ has finite rank it is then standard to construct a finite $%
2\epsilon K$-net for $\pi _{n}(B_{W^{1,1}}(0,K))$ which by $(\ref{w1})$ will
also cover $B_{W^{1,1}}(0,K).$
\end{proof}

Now let us prove the existence of an absolutely continuous invariant
measure\ (with density in $W^{1,1}$) for the expanding maps. As the reader
may notice, the procedure can be generalized to many other cases where a
regularization inequality and a compactness statement like Proposition \ref%
{RK} are available.

\begin{proposition}
The transfer operator $L$ associate to an expanding map has an invariant
density \ $h\in L^{1}$.
\end{proposition}

\begin{proof}
Let us consider the sequence 
\begin{equation}
g_{n}=\frac{1}{n}\sum_{i=0}^{n-1}L^{n}1  \label{gn}
\end{equation}%
where $1$ is the density of the normalized Lebesgue measure.\ By the Lasota
Yorke inequality $(\ref{LY})$\ the sequence has uniformly bounded $W^{1,1}$
norm and by Proposition \ref{RK} has a subsequence $g_{n_{k}}$ converging in 
$L^{1}$ to a limit $h$.

Now recall that $L$ is continuous in the $L^{1}$ norm. By this\footnote{%
Remark that by the definition $g_{n}=\frac{1}{n}(L^{0}1+...+L^{n}1)$ giving $%
||Lg_{n_{K}}-g_{n_{k}}||_{1}\leq \frac{2}{n}$.} 
\begin{equation*}
Lh=L(\lim_{k\rightarrow \infty }g_{n_{k}})=\lim_{k\rightarrow \infty
}Lg_{n_{k}}=h.
\end{equation*}%
Then $h$ is an invariant density.
\end{proof}

\begin{proposition}
\label{cauchy}The density $h$ found above has the following properties:

\begin{itemize}
\item $h\in W^{1,1}$ and

\item $||h||\leq \frac{(||\frac{T^{\prime \prime }}{(T^{^{\prime }})^{2}}%
||_{\infty }+1)}{1-\alpha }.$
\end{itemize}
\end{proposition}

\begin{proof}
Consider $g_{n}$ as defined at $(\ref{gn})$ and\ $g_{n,m}=L^{m}(g_{n}).$
Remark that, $g_{n,m}\in W^{1,1}$ and the norms are uniformly bounded. By
the Lasota Yorke inequality 
\begin{equation}
||g_{n_{1},a+m}-g_{n_{2},b+m}||_{W^{1,1}}\leq \alpha
^{m}||g_{n_{1},a}-g_{n_{2},b}||_{W^{1,1}}+B||g_{n_{1},a}-g_{n_{2},b}||_{1}
\label{cazzo}
\end{equation}%
also remark that if $||g_{n_{k},0}-h||_{1}\leq \epsilon $ then $%
||g_{n_{k},j}-h||_{1}\leq \epsilon $ for all $j\geq 0$. Then the sequence $%
g_{n_{k},k}\rightarrow h$ in $L^{1},$ and by $(\ref{cazzo})$ is a Cauchy
sequence in the $W^{1,1}$norm, indeed, suppose $k_{1}\leq k_{2}$%
\begin{equation}
||g_{n_{k_{1}},k_{1}}-g_{n_{k_{2}},k_{2}}||_{W^{1,1}}\leq \alpha
^{k_{1}}||g_{n_{k_{1}},0}-g_{n_{k_{2},k_{2}-k_{1}}}||_{W^{1,1}}+B||g_{n_{k_{1}},0}-g_{n_{k_{2},k_{2}-k_{1}}}||_{1}.
\end{equation}

Since $W^{1,1}$ is complete, this implies that it converges in $W^{1,1}$ to
some limit which is forced to be $h.$ Hence $h\in W^{1,1}.$

By the Lasota Yorke inequality, since $h$ is invariant then $%
||h||=||Lh||\leq \frac{(||\frac{T^{\prime \prime }}{(T^{^{\prime }})^{2}}%
||_{\infty }+1)}{1-\alpha }$.
\end{proof}

\begin{remark}
\label{c3} Using the procedure explained in this section, using the Lasota
Yorke inequality \ref{LYlip} (remark that starting with the constant density 
$1$ we are constructing at each step a $C^{1}$ function and we are
controlling its Lipschitz norm, which is equivalent to the $C^{1}$ norm on $%
C^{1}$). It is possible to prove that a $C^{2}$ expanding map has a $C^{1}$
invariant density. In a similar way it is possible to prove that a $C^{3}$ \
expanding map of the circle has a $C^{2}$ invariant density.
\end{remark}

\begin{remark}
The measure $hm$, where $h$ is the smooth invariant density found above is a
physical measure for the expanding map $T$.

In the next section we will also see that such $h$ is unique.
\end{remark}

\section{Convergence to equilibrium and mixing\label{CEM}}

In this section we see the concept of convergence to equilibrous, and how to
ise it to prove that the absolutely continuous invariant measure found for
an expanding map in the previous section is unique. We will also see that an
expanding map considered with its alsolutely continuous invariant measure is
a mixing measure preserving transformation.

Let us consider a positive, integral preserving (Markov) operator $L$ acting
on a strong and a weak normed vector spaces $B_{s}\subseteq B_{w}\subseteq
L^{1}$ as we have seen before.

\noindent \textbf{Assumptions A.} Let us suppose that:

\begin{itemize}
\item $B_{w}$ contains the indicator functions of measurable sets

\item $B_{s}$ is closed under product and there is $C\geq 1$ such that for
each $f,g\in B_{s}$, $||fg||_{s}\leq C||f||_{s}||g||_{s}$.

\item $B_{s}$ is dense in $B_{w}$ for the $||~||_{w}$ topology

\item for each $f\in B_{w}$ $||f||_{w}\geq \int |f|~dm$ and $%
||~||_{B_{s}}\geq ||~||_{L^{\infty }}.$
\end{itemize}

\begin{remark}
\label{sarig}We remark that by a rescaling the assumption $||fg||_{s}\leq
C||f||_{s}||g||_{s}$ can be put in the form $||fg||_{s}\leq
||f||_{s}~||g||_{s}$. Considering the new rescaled norm $%
||~||_{K}=K||~||_{s} $ it holds%
\begin{equation*}
||fg||_{K}=K||fg||_{s}\leq K^{2}||f||_{s}~||g||_{s}\leq ||f||_{K}||g||_{K}.
\end{equation*}
\end{remark}

Let us consider the strong and weak space of zero average densities

\begin{equation*}
V_{s}=\{g\in B_{s}~s.t.~\int g~dm=0\}
\end{equation*}%
and%
\begin{equation*}
V_{w}=\{g\in B_{w}~s.t.~\int g~dm=0\}.
\end{equation*}

If the dynamics has some mixing properties we expect that iterating a zero
average density its positive part gets annihilated with the negative part
and the iterates will eventually converge to zero in some sense.

This can be expressed more in general for Markov operators and the general
definition will be useful to treat deterministic and random dynamical
systems with the same tools and general results.

\begin{definition}[Convergence to equilibrium]
A Markov operator $L:B_{s}\rightarrow B_{s}$ is said to have convergence to
equilibrium if for each $g\in V_{s}$%
\begin{equation*}
\lim_{n\rightarrow \infty }||L^{n}g||_{w}=0.
\end{equation*}
\end{definition}

We recall the classical definition of mixing

\begin{definition}[Mixing]
A measure preserving transformation $(X,T,\mu )$ is said to be mixing if for
each measurable $E,F\subseteq X$%
\begin{equation*}
\lim_{n\rightarrow \infty }\mu (E\cap T^{-n}F)=\mu (E)~\mu (F).
\end{equation*}
\end{definition}

We now see how a convergence to equilibrium result implies mixing.

\begin{proposition}
\label{mix}Consider $B_{s}$ and $B_{w}$ satisfying the properties listed at
the beginning of the section and the transfer operator $L:B_{w}\rightarrow
B_{w}$ associated to a map $T$ having $\mu =hm$ with $h\in B_{s}$ as an
invariant probability measure. Suppose that $h$ is bounded, and that for
each $f\in V_{s}$%
\begin{equation}
\lim_{n\rightarrow \infty }||L^{n}f||_{w}\rightarrow 0  \label{conv121}
\end{equation}%
then the system $(T,\mu )$ is mixing.
\end{proposition}

\begin{proof}
By the assumptions, given any $f\in B_{s}$ with $\int f~dm=1$ the sequence $%
L^{n}f-h\int f~dm\rightarrow 0$ in $B_{w}.$

Now let us consider two measurable sets $E$ and $F$ and the indicator
functions $1_{E}$, $1_{F}\in B_{w}$. We have that%
\begin{eqnarray*}
\mu (E\cap T^{-n}F) &=&[hm](E\cap T^{-n}F) \\
&=&\int 1_{E}~(1_{F}\circ T^{n})~hdm.
\end{eqnarray*}

By the density of $B_{s}$ in $B_{w}$, let us consider $\epsilon >0$ and $%
g_{\epsilon }\in B_{s}$ such that $||1_{E}-g_{\epsilon }||_{w}\leq \epsilon $
thus for each $n\geq 0$%
\begin{equation*}
|\mu (E\cap T^{-n}F)-\int g_{\epsilon }~(1_{F}\circ T^{n})~hdm|\leq
||1_{E}-g_{\epsilon }||_{w}
\end{equation*}%
and%
\begin{equation*}
|\mu (E\cap T^{-n}F)-\int 1_{F}~L^{n}(hg_{\epsilon })~dm|\leq \epsilon .
\end{equation*}

By $(\ref{conv121})$ we have $hg_{\epsilon }\rightarrow h[\int (hg_{\epsilon
})~dm]$ in $B_{w}$, then 
\begin{equation*}
\int 1_{F}~L^{n}(hg_{\epsilon })~dm\underset{n\rightarrow \infty }{%
\rightarrow }\int 1_{F}~[h\int (hg_{\epsilon })~dm]~dm=\int
(h1_{F})~dm~\cdot ~\int (hg_{\epsilon })~dm
\end{equation*}%
and since 
\begin{equation*}
|\int (hg_{\epsilon })~dm-\int (h1_{E})~dm|\leq \epsilon
\end{equation*}%
we have that for each $\epsilon ,$ eventually as $n\rightarrow \infty $ 
\begin{equation*}
|\mu (E\cap T^{-n}F)-\mu (E)~\mu (F)|\leq 3\epsilon
\end{equation*}%
proving the statement.
\end{proof}

Now we prove that expanding maps have convergence to equilibrium and are
mixing. Later we will see that the speed of convergence of this equilibrium
is exponential. We consider $W^{1,1}$ and $L^{1}$ as a strong and weak
spaces.

\begin{proposition}
\label{propora}For each $g\in V_{s}$, it holds 
\begin{equation*}
\lim_{n\rightarrow \infty }||L^{n}g||_{1}=0.
\end{equation*}
\end{proposition}

\begin{proof}
First let us suppose that $||g||_{Lip}<\infty $. By $(\ref{LYlip})$ we know
that all the iterates of $g$ have uniformly bounded $l$ norm. There is $%
\overline{M}\geq 0$ such that for each $n\geq 0$ 
\begin{equation*}
||L^{n}g||_{Lip}\leq \overline{M}.
\end{equation*}

Let us denote by $g^{+},g^{-}$ the positive and negative parts of $g$.
Remark that since $g\in V_{s}$, $||g||_{1}=2\int g^{+}~dm$. There is a point 
$\overline{x}$ such that $g^{+}(\overline{x})\geq \frac{1}{2}||g||_{1}$.
Around this point consider a neighborhood $N=B(\overline{x},\frac{1}{4}%
||g||_{1}\overline{M}^{-1}).$ For each point $x\in N,$ $g^{+}(x)\geq \frac{1%
}{4}||g||_{1}$.

Now let $d=\min |T^{^{\prime }}|,$ $D=\max |T^{\prime }|$. If $n_{1}$ is the
smallest integer such that $d^{n_{1}}\frac{1}{2}||g||_{1}\overline{M}^{-1}>1$
\begin{equation*}
(i.e.~\frac{\log (2||g||_{1}^{-1}\overline{M})}{\log d}+1>n_{1}>\frac{\log
(2||g||_{1}^{-1}\overline{M})}{\log d})
\end{equation*}
then $T^{n_{1}}(N)=S^{1}$ and $L^{n_{1}}g^{+}$ has density%
\begin{eqnarray*}
L^{n_{1}}g^{+} &\geq &\frac{||g||_{1}}{4D^{n_{1}}} \\
&\geq &\frac{||g||_{1}}{4De^{\log (2||g||_{1}^{-1}\overline{M})\frac{\log D}{%
\log d}}} \\
&=&\frac{||g||_{1}}{4D}(\frac{1}{2}||g||_{1}\overline{M}^{-1})^{\frac{\log D%
}{\log d}}
\end{eqnarray*}%
on $S^{1}$. The same is true for $g^{-}$ and then, after iterating $n_{1}$
times this positive constant part of the density and the corresponding
negative one annihilates, and setting $C=\frac{(\frac{1}{2}\overline{M}%
^{-1})^{\frac{\log D}{\log d}}}{4D},$ it holds%
\begin{equation*}
||L^{n_{1}}g||_{1}\leq ||g||_{1}-C||g||_{1}{}^{\frac{\log D}{\log d}+1}.
\end{equation*}%
Let us denote $g_{1}=L^{n_{1}}g$. We can repeat the above construction and
obtain $n_{2}$ such that%
\begin{equation*}
||g_{2}||_{1}:=||L^{n_{2}}g_{1}||_{1}\leq ||g_{1}||_{1}-C||g_{1}||_{1}{}^{%
\frac{\log D}{\log d}+1}
\end{equation*}%
and so on. Continuing, we have a sequence $g_{n}$ such that 
\begin{equation*}
||g_{n+1}||_{1}\leq ||g_{n}||_{1}-C||g_{n}||_{1}{}^{\frac{\log D}{\log d}+1}
\end{equation*}%
and then $||g_{n}||_{1}\rightarrow 0.$

If now more generally, $g\in W^{1,1}$ we can approximate $g$ with a $\tilde{g%
}$ such that $||\tilde{g}||_{Lip}<\infty $ in a way that $||g-\tilde{g}%
||_{1}\leq \epsilon $. Since $||L||_{L^{1}\rightarrow L^{1}}\leq 1$, $%
\lim_{n\rightarrow \infty }||L^{n}(g-\tilde{g})||_{1}\leq \epsilon $\ . And
then the statement follows.
\end{proof}

Using Proposition \ref{mix} we get

\begin{corollary}
Expanding maps of the circle, considered with its $W^{1,1}$invariant measure
are mixing.
\end{corollary}

\begin{corollary}
For expanding maps of the circle, there is only one invariant measure in $%
W^{1,1}.$
\end{corollary}

\begin{proof}
If there are two invariant probability densities $h_{1},h_{2}$ then $%
h_{1}-h_{2}\in V_{s}$ and invariant, impossible.
\end{proof}

\begin{corollary}
The whole sequence $g_{n}=\frac{1}{n}\sum_{i=0}^{n-1}L^{n}1$ converges to $%
h. $
\end{corollary}

\begin{remark}
By Proposition \ref{propora} it also follows that, if $g\in W^{1,1}$is a
probability density (and then $g-h\in V_{s}$) then $L^{n}g\rightarrow h$ in
the $L^{1}$ norm. By the Lasota Yorke inequality we also get the convergence
in $W^{1,1}$. Indeed (see the proof of Proposition \ref{cauchy}) $%
||L^{n+m}(g-h)||_{w^{1,1}}\leq \lambda
^{n}||L^{m}(g-h)||_{W^{1,1}}+B||L^{m}(g-h)||_{1}$ and letting first $%
m\rightarrow \infty $ \ and then $n\rightarrow \infty $ \ we get the
statement.
\end{remark}

\subsection{Speed of convergence to equilibrium and decay of correlations 
\label{eq}}

For several applications, it is important to quantify the speed of mixing or
convergence to equilibrium.

Let us see how to quantify: consider two vector subspaces of the space of
signed (complex) Borel measures on $X$ 
\begin{equation*}
B_{s}\mathcal{\subseteq }B_{w}
\end{equation*}%
endowed with two norms, the strong norm $||~||_{s}$ on $B_{s}$ and the weak
norm $||~||_{w}$ on $B_{w}$, such that $||~||_{s}\geq ||~||_{w}$ on $B_{w}$

In the case of operators acting on spaces of measures which are not
necessarily absolutely continuous we say that an operator is \emph{Markov }%
if it preserves positive measures and for each signed measure $\mu $ it
holds $\mu (X)=[L\mu ](X)$.

\begin{definition}
\label{conve}Let us consider $\Phi :\mathbb{N\rightarrow R}$ with $%
\lim_{n\rightarrow \infty }\Phi (n)=0.$ We say that a Markov operator $L$
has convergence to equilibrium with speed $\Phi $ with respect to these
norms if for any $f\in V_{s}$.%
\begin{equation}
||L^{n}f||_{w}\leq \Phi (n)~||f||_{s}.  \label{wwe}
\end{equation}
\end{definition}

We remark that in this case if $\nu $ is a starting probability measure in $%
B_{s}$ and $\mu $ is the invariant measure, still in $B_{s}$, then $\nu -\mu
\in V_{s}$ and then 
\begin{equation*}
||L^{n}\nu -\mu ||_{w}\leq \Phi (n)~||\nu -\mu ||_{s}.
\end{equation*}%
and then $L^{n}\nu $ converges to $\mu $ with a speed $\Phi (n)$. Depending
on the strong norm, in certain cases one may prove that for each probability
measure $\nu \in B_{s}$, $||\nu -\mu ||_{s}\leq K||\nu ||_{s}$ where $K$
does not depend on $\nu $, obtaining 
\begin{equation*}
||L^{n}\nu -\mu ||_{w}\leq K\Phi (n)~||\nu ||_{s}
\end{equation*}

\subsection{Decay of correlations\label{corrr}}

Convergence to equilibrium is often estimated or applied in the form of
correlation integrals. In this subsection we show an example of how to
relate these colleration estimates with the notions of convergence to
equilibrium we defined above.

Let us suppose that $X$ is a manifold, let us denote by $m$ the normalized
Lebesgue measure, consider non singular transformations, and let us apply
the transfer operators to absolutely continuous invariant measures. We
consider the transfer operator acting on a strong and a weak space $%
B_{s}\subseteq B_{w}\subseteq L^{1}$ with norms satisfying the Assumptions $%
A $ listed at beginning of Section \ref{CEM}.

The reader can easily verify that these assumptions hold for example when $%
X=S^{1}$ and $B_{s}=W^{1,1},$ $B_{w}=L^{1}$ and apply the results to
expanding maps. Similar arguments applies to other kinds of spaces.

\subsubsection{Estimating $\protect\int \protect\psi \circ T^{n}~g~dm-%
\protect\int g~dm~\protect\int \protect\psi ~d\protect\mu $}

In this subsection we consider a measure preserving transformation $(T,\mu )$
with $\mu =hm$ and we see how the notion of convergence to equilibrium, as
in Definition \ref{conve} allow the estimate of the following integral%
\begin{equation*}
\int \psi \circ T^{n}~g~dm-\int g~dm~\int \psi ~d\mu .
\end{equation*}

\begin{lemma}
Let us consider normed vector spaces $B_{s}\subseteq B_{w}$, as described
above. Consider the transfer operator $L$ associated to a map having $\mu
=hm $ with $h\in B_{s}$ as an invariant measure. Suppose that there is $\Phi
(n)\rightarrow 0$ such that for each $f\in V_{s}$ and $n\geq 0$%
\begin{equation*}
||L^{n}f||_{w}\leq \Phi (n)||f||_{B_{s}}
\end{equation*}%
then for each $\forall \psi \in L^{\infty },$ $g\in B_{s},$ $n\geq 0$ we
have 
\begin{equation*}
|\int \psi \circ T^{n}~g~dm-\int g~dm~\int \psi ~d\mu |\leq 2C\Phi
(n)||h||_{B_{s}}||\psi ||_{\infty }||g||_{B_{s}}.
\end{equation*}
\end{lemma}

\begin{proof}
Since $\mu $ is invariant \ and $\psi \in L^{\infty }$%
\begin{eqnarray*}
|\int \psi \circ T^{n}~g~dm-\int g~dm~\int \psi ~d\mu | &\leq & \\
&\leq &|\int \psi \circ T^{n}~g~dm-\int g~dm\int \psi \circ T^{n}~h~dm| \\
&\leq &|\int \psi \circ T^{n}~[g-h\int g~dm]~dm| \\
&=&|\int \psi ~L^{n}[g-h\int g~dm]~dm|
\end{eqnarray*}%
since $g-h\int g~dm$ is a zero average density in $B_{s}$%
\begin{eqnarray}
|\int \psi \circ T^{n}~g~dm-\int g~dm~\int \psi ~d\mu | &\leq &||\psi
||_{\infty }||L^{n}[g-h\int g~dm]||_{w}  \label{conv} \\
&\leq &\Phi (n)||\psi ||_{\infty }||g-h\int g~dm||_{B_{s}}
\end{eqnarray}%
and 
\begin{eqnarray*}
||(g-\int g~dm)h||_{B_{s}} &\leq &||gh||_{B_{s}}+||\int g~dmh||_{B_{s}} \\
&\leq &2C||h||_{B_{s}}||g||_{B_{s}}.
\end{eqnarray*}
\end{proof}

\subsubsection{Estimating $\protect\int \protect\psi \circ T^{n}~g~d\protect%
\mu -\protect\int g~d\protect\mu ~\protect\int \protect\psi ~d\protect\mu $
(a decay of correlation estimate)}

Suppose we have a system having invariant measure $\mu $ and convergence to
equilibrium with respect to norms $||~||_{w},||~||_{B_{s}}$ and speed $\Phi $
as above. For many applications it is useful to estimate the speed of
decreasing of the following correlation integral 
\begin{equation}
|\int g\cdot (\psi \circ F^{n})~d\mu -\int \psi ~d\mu \int g~d\mu |.
\end{equation}

\begin{lemma}
Consider the transfer operator $L$ associated to a map having $\mu =hm$ with 
$h\in B_{s}$ as an invariant measure as above. Suppose that for each $f\in
B_{s}$ such that $\int f~dm=0$ and $n\geq 0$%
\begin{equation*}
||L^{n}f||_{w}\leq \Phi (n)||f||_{B_{s}}
\end{equation*}%
then for each $\forall \psi \in L^{\infty },g\in B_{s}$ and $n\geq 0$ it
holds%
\begin{equation*}
|\int g\cdot (\psi \circ F^{n})~d\mu -\int \psi ~d\mu \int g~d\mu |\leq
2C[||h||_{B_{s}}+||h||_{B_{s}}^{2}]\Phi (n)||\psi ||_{\infty }||g||_{B_{s}}.
\end{equation*}
\end{lemma}

\begin{proof}
First we remark that adding a constant $K$ to $g$ does not change the
correlation integral:%
\begin{eqnarray}
\int (g+K)\cdot (\psi \circ F^{n})~d\mu -\int \psi ~d\mu \int (g+K)~d\mu
&=&\int g\cdot (\psi \circ F^{n})~d\mu +K\int (\psi \circ F^{n})~d\mu \\
&&-\int \psi ~d\mu \int gd\mu -K\int \psi ~d\mu
\end{eqnarray}%
but since $\mu $ in invariant 
\begin{equation*}
K\int (\psi \circ F^{n})~d\mu -K\int \psi ~d\mu =0
\end{equation*}%
and then 
\begin{equation*}
|\int (g+K)\cdot (\psi \circ F^{n})~d\mu -\int \psi ~d\mu \int (g+K)~d\mu
|=|\int g\cdot (\psi \circ F^{n})~d\mu -\int \psi ~d\mu \int g~d\mu |.
\end{equation*}%
We then choose $K=-\int g~d\mu ,$ we have that $\int g-[\int g~d\mu ]~d\mu
=0 $ and%
\begin{eqnarray*}
|\int g\cdot (\psi \circ F^{n})~d\mu -\int \psi ~d\mu \int g~d\mu | &=&|\int
(g-\int g~d\mu )\cdot (\psi \circ F^{n})~d\mu | \\
&=&|\int \psi ~L^{n}[(g-\int g~d\mu )h]~dm| \\
&\leq &||\psi ||_{\infty }||L^{n}[(g-\int g~d\mu )h]||_{w}
\end{eqnarray*}%
and since $\int [g-\int g~d\mu ]h~dm=0$ the convergence to equilibrium
implies 
\begin{equation*}
|\int g\cdot (\psi \circ F^{n})~d\mu -\int \psi ~d\mu \int g~d\mu |\leq
||\psi ||_{\infty }\Phi (n)||(g-\int g~d\mu )h||_{B_{s}}
\end{equation*}%
and 
\begin{eqnarray*}
||(g-\int g~d\mu )h||_{B_{s}} &\leq &||gh||_{B_{s}}+||\int g~d\mu
~h||_{B_{s}} \\
&\leq &C||h||_{B_{s}}||g||_{B_{s}}+|\int g~d\mu |~||h||_{B_{s}} \\
&\leq &C[||h||_{B_{s}}+||h||_{B_{s}}^{2}]||g||_{B_{s}}
\end{eqnarray*}
\end{proof}

In the next section we will see how the Lasota Yorke inequality and the
properties of the spaces we have chosen allows to prove exponential speed of
convergence for our circle expanding maps.

\section{Spectral gap and consequences\label{spg}}

Now we see a general result that easily implies that the convergence to
equilibrium of certain systems is exponentially fast.

We recall some basic concepts on the spectrum of operators. Let $%
L:B\rightarrow B$ be an operator acting on a complex Banach space $(B,||~||)$%
:

\begin{itemize}
\item the spectrum of an operator is defined as%
\begin{equation*}
spec(L)=\{\lambda \in \mathbb{C}:(\lambda I-L)~has~no~bounded~inverse\}
\end{equation*}

\item the spectral radius of $L$ is defined as 
\begin{equation*}
\rho (L)=\sup \{|z|:z\in spec(L)\}.
\end{equation*}
\end{itemize}

An important connection between the spectral properties of the operator and
the asymptotic behavior of its iterates is given by the following formula

\begin{proposition}[Spectral radius formula]
Under the above assumptions%
\begin{equation*}
\rho (L)=\lim_{n\rightarrow \infty }\sqrt[n]{||L^{n}||}=\inf_{n}\sqrt[n]{%
||L^{n}||}.
\end{equation*}
\end{proposition}

\begin{definition}[Spectral gap]
\label{defgap}The operator $L:B\rightarrow B$ is said to have spectral gap if%
\begin{equation*}
L=\lambda \func{P}+\func{N}
\end{equation*}%
where

\begin{itemize}
\item $\func{P}$ is a projection (i.e. $\func{P}^{2}=\func{P}$) and $\dim
(Im(\func{P}))=1$;

\item the spectral radius of $\func{N}$ satisfies $\rho (\func{N})<|\lambda
| $;

\item $\func{P}\func{N}=\func{N}\func{P}=0$.
\end{itemize}
\end{definition}

The following is an elementary tool to verify spectral gap of $L$ on $B_{s}$.

\begin{theorem}
\label{gap}Let us consider a Markov operator $L$ acting on two normed vector
spaces $(B_{s},||~||_{s}),~(B_{w},||~||_{w}),$ $B_{s}\subseteq
B_{w}\subseteq CM(X)$ with $||~||_{s}\geq ||~||_{w}$ (where $CM(X)$ stands
for the set of Borel complex valued measures on $X$). Suppose:

\begin{enumerate}
\item (Lasota Yorke inequality). There are $A,B\geq 0$ and $0\leq \lambda
_{1}\leq 1$ such that for each $g\in B_{s}$%
\begin{equation*}
||L^{n}g||_{s}\leq A\lambda _{1}^{n}||g||_{s}+B||g||_{w};
\end{equation*}

\item (Convergence to equilibrium) for each $g\in V_{s}$, it holds 
\begin{equation}
\lim_{n\rightarrow \infty }||L^{n}g||_{w}=0;  \label{CE}
\end{equation}

\item (Compact inclusion) the strong zero average space $V_{s}$ is compactly
immersed in the weak one $V_{w}$ (more precisely, the strong unit ball has a
finite $\epsilon $ net in the weak topology\ for each $\epsilon $);

\item (Weak boundedness) the weak norm of the operator restricted to $V_{s}$
satisfies%
\begin{equation*}
\sup_{n}||L^{n}|_{V_{s}}||_{w}<\infty .
\end{equation*}
\end{enumerate}

Under these assumptions there are $C_{2}>0,\rho _{2}<1$ such that for all $%
g\in V_{s}$%
\begin{equation}
||L^{n}g||_{s}\leq C_{2}\rho _{2}^{n}||g||_{s}.  \label{gap2}
\end{equation}
\end{theorem}

\begin{proof}
We first show that assumptions (2) and (3) and (4) imply that $L$ is
uniformly contracting from $V_{s}$ to $V_{w}$: there is $n_{1}>0$ such that $%
\forall g\in V_{s}$%
\begin{equation}
||L^{n_{1}}g||_{w}\leq \lambda _{2}||g||_{s}  \label{1}
\end{equation}

where$~\lambda _{2}B<1.$

Indeed, by (3), for any $\epsilon $ there is a finite set $\{g_{i}\}_{i\in
(1,...,k)}$ in the strong unit ball $\mathcal{B}$ of $V_{s}$ such that for
each $g$ in $\mathcal{B}$ there is a $g_{i}\in V_{s}$ such that $%
||g-g_{i}||_{w}\leq \epsilon .$

Hence%
\begin{equation*}
\sup_{g\in V_{s},||g||_{s}\leq 1}||L^{n}g||_{w}\leq \sup_{1\leq i\leq k,v\in
\{v\in V_{s}~s.t.~||v||_{w}\leq \epsilon \}}||L^{n}(g_{i}+v)||_{w}.
\end{equation*}%
Now, by (4) suppose that $\forall n~$\ $||L^{n}|_{V_{s}}||_{w}\leq M$ , then%
\begin{equation*}
\sup_{i}||L^{n}(g_{i}+v)||_{w}\leq \sup_{i}||L^{n}(g_{i})||_{w}+M\epsilon .
\end{equation*}%
Since $\epsilon $ can be chosen as small as wanted \ and by (2) for each $i$%
, $\lim_{n\rightarrow \infty }||L^{n}(g_{i})||_{w}=0$ and we have $(\ref{1})$
(first fix $\epsilon $ small enough and then choose $i$ big enough) .

Let us apply the Lasota Yorke inequality to strengthen $(\ref{1})$ to an
estimate for the strong norm. For each $f\in V_{s}$%
\begin{equation*}
||L^{n_{1}+m}f||_{s}\leq A\lambda
_{1}^{m}||L^{n_{1}}f||_{s}+B||L^{n_{1}}f||_{w}
\end{equation*}%
then%
\begin{eqnarray*}
||L^{n_{1}+m}f||_{s} &\leq &A\lambda _{1}^{m}||L^{n_{1}}f||_{s}+B\lambda
_{2}||f||_{s} \\
&\leq &A\lambda _{1}^{m}[A\lambda _{1}^{n_{1}}||f||_{s}+B||f||_{w}]+B\lambda
_{2}||f||_{s}.
\end{eqnarray*}

If $m$ is big enough%
\begin{equation*}
||L^{n_{1}+m}f||_{s}\leq \lambda _{3}||f||_{s}
\end{equation*}%
with $\lambda _{3}<1.$

This easily implies the statement. Indeed \ set $n_{2}=n_{1}+m,$ for each $%
k,q\in \mathbb{N},$ $q\leq n_{2}$, $g\in V_{s}$,%
\begin{eqnarray*}
||L^{kn_{2}+q}g||_{s} &\leq &\lambda _{3}^{k}||L^{q}g||_{s} \\
&\leq &\lambda _{3}^{k}(\lambda _{1}^{q}||g||_{s}+B||g||_{w}).
\end{eqnarray*}

Implying that for each $g\in V_{s}$ there are $C_{2}>0,\rho _{2}<1$ such that%
\begin{equation}
||L^{n}g||_{s}\leq C_{2}\rho _{2}^{n}||g||_{s}.  \label{2232}
\end{equation}
\end{proof}

In the case where $(B_{s},||~||_{s})$ is a complex Banach space, by this
theorem and the spectral \ radius formula, the spectral radius of $L$
restricted to $V_{s}$ is strictly smaller than $1$, and the spectral gap as
defined in Definition \ref{defgap} follows.

\begin{theorem}
\label{spg1}Under the assumptions of Theorem \ref{gap}, if $%
(B_{s},||~||_{s}) $ is a Banach space then $L$ has spectral gap.
\end{theorem}

Before the proof we need a preliminary Lemma

\begin{lemma}
Under the assumptions of Theorem \ref{gap}, if $(B_{s},||~||_{s})$ is a
Banach space then\ $L$ has a unique invariant probability measure in $B_{s}$.
\end{lemma}

\begin{proof}
The proof follows the same construction as in section \ref{ext} using the
compact immersion (assumption 3 of Theorem \ref{gap}) instead of Proposition %
\ref{RK}.
\end{proof}

\begin{proof}[Proof of Theorem \protect\ref{spg1} (sketch)]
Remark that by the Lasota Yorke inequality and the spectral radius formula,
the spectral radius of $L$ on $B_{s}$ is not greater than than $1$. Since
there is an invariant measure in $B_{s}$ then this radius is $1$. By (2)
there can be only one fixed probability measure of $L$ in $B_{s}$ which we
denote by $\mu $ (if there were two, consider the difference which is in $%
V_{s}$ and iterate...).

Now let us remark that every $g\in B_{s}$ can be written as follows:%
\begin{equation*}
g=[g-\mu ~g(X)]+[\mu ~g(X)].
\end{equation*}%
the\footnote{%
Where $g(X)$ \ stands for the $g-$measure of the whole space.} function $%
P:B_{s}\rightarrow B_{s}$ defined as 
\begin{equation*}
P(g)=\mu ~g(X)
\end{equation*}%
is a projection. The function \ $N:B_{s}\rightarrow B_{s}$ defined as%
\begin{equation*}
N(g)=L[g-\mu ~g(X)]
\end{equation*}%
is such that $N(B_{s})\subseteq V_{s}$ , $N|V_{s}=L|V_{s}$, and by $(\ref%
{gap2})$ it satisfies $\rho (N)<1$. It holds 
\begin{equation*}
L=P+N
\end{equation*}%
and $PN=NP=0$. Thus $L$ has spectral gap according to the Definition \ref%
{defgap}.
\end{proof}

We remark that in several texts the role of Theorem \ref{spg1} is played by
a general result referred to Hennion, Herv\'{e} or Ionescu-Tulcea and
Marinescu (see e.g. \cite{L2}, \cite{S}) whose proof is more complicated.

\begin{remark}
\label{conv1}Equation \ref{2232} obviously implies exponential convergence
to equilibrium with exponential speed.
\end{remark}

\begin{remark}[spectral gap for expanding maps of the circle]
\label{cnv2}By Proposition \ref{RK}\footnote{%
Which can be easily adapted to $V_{s}$, by considering an integral
preserving projection $\pi _{2}f=\pi f-\int \pi f$ \ .} , Proposition \ref%
{propora} and the Lasota Yorke inequality, the assumptions of Theorem \ref%
{gap} are verified on our expanding maps of the circle for the $W^{1,1}$
norm (with the $L^{1}$ norm as a weak norm). Then their transfer operator
have spectral gap.
\end{remark}

\subsection{Central limit}

We see an application of Theorem \ref{gap} to the estimation of the
fluctuations of an observable, obtaining a sort of central limit theorem. A
proof of the result can be found in \cite{S} (see also Remark \ref{sarig}).

\begin{theorem}
\label{CL}Let $(X,T,\mu )$ be a mixing nonsingular measure preserving
transformation. Consider its associater transfer operator $L$ acting on some
Banach spaces $B_{s}$ \ and $B_{w}=L^{1}$, suppose $B_{s}$ and $B_{w}$
satisfy Assumptions A at beginning of Section \ref{CEM} and that furthermore 
$B_{s}$ contains the constant fuctions.\ Suppose $L:B_{s}\rightarrow B_{s}$
has spectral gap.

Let $f\in V_{s}$. If there is no $\nu \in \mathcal{B}$ such that $f=\nu -\nu
\circ T$ a.e., then $\exists \sigma >0$ s.t. for all intervals $[a,b],$%
\begin{equation*}
\mu \left\{ x:\frac{1}{\sqrt{n}}\sum_{k=0}^{n-1}f\circ T^{k}\in \lbrack
a,b]\right\} \rightarrow \frac{1}{\sqrt{2\pi \sigma ^{2}}}%
\int_{a}^{b}e^{-t^{2}/2\sigma ^{2}}dt.
\end{equation*}
\end{theorem}

We remark, that by Theorem \ref{spg1} and the general properties we have
shown about the spaces $W^{1,1}$and $L^{1},$ Theorem \ref{CL} applies to
expanding maps on $S^{1}$.

\section{Stability and response to perturbation\label{stab}}

In this section we consider small perturbations of a given system and try to
study the dependence of the invariant measure on the perturbation. If the
measure varies continuously, we know that many of the statistical properties
of the system are stable under perturbation (see \cite{AS} and \cite{DZ} for
examples of results in this direction, in several classes of systems) and
the system will be said to have statistical stability.

On the other hand it is known that even in relatively simple families of
piecewise expanding maps, the physical invariant measure may change
discontinuously (see Section. \ref{disc}).

We will see that under certain general assumptions related to the
convergence to equilibrium of the system and the kind of perturbation, the
physical measure changes continuously, and we can estimate quantitatively
the modulus of continuity. If stronger, assumptions applies, the dependence
can be Lipschitz, or even differentiable.

We remark that with more work, other stability results can be proved for the
whole spectral picture of the system and not only for the physical measure
(see \cite{L2}).

Consider again two vector spaces of measures with sign on $X$ 
\begin{equation*}
B_{s}\mathcal{\subseteq }B_{w}\mathcal{\subseteq }SM(X),
\end{equation*}%
endowed with two norms, the strong norm $||~||_{s}$ on $B_{s}$ and the weak
norm $||~||_{w}$ on $B_{w}$, such that $||~||_{s}\geq ||~||_{w}$ as before.
Suppose $L_{\delta }(B_{s})\subseteq B_{s}$and $L_{\delta }(B_{w})\subseteq
B_{w}$. Denote as before by $V_{s},V_{w}$ the "zero average" spaces.

\textbf{A uniform family of operators. }Let us consider a one parameter
family of operators $L_{\delta }$ , $\delta \in \lbrack 0,1)$. Suppose that:

\begin{enumerate}
\item[\textbf{UF1}] (Uniform Lasota Yorke ineq.) There are constants $%
A,B,\lambda _{1}\geq 0$ with $\lambda _{1}<1$ such that $\forall f\in
B_{s},\forall n\geq 1,\forall \delta \in \lbrack 0,1)$ and each operator
satisfies a Lasota Yorke inequality. 
\begin{equation}
||L_{\delta }^{n}f||_{s}\leq A\lambda _{1}^{n}||f||_{s}+B||f||_{w}.
\end{equation}

\item[\textbf{UF2}] Suppose that $L_{\delta }$ approximates $L_{0}$ when $%
\delta $ is small in the following sense: there is $C\in \mathbb{R}$ such
that $\forall g\in B_{s}$:%
\begin{equation}
||(L_{\delta }-L_{0})g||_{w}\leq \delta C||g||_{s}.
\end{equation}

\item[\textbf{UF3}] Suppose that $L_{0}$ has exponential convergence to
equilibrium, with respect to the norms $||~||_{w}$ and $||~||_{s}$.

\item[\textbf{UF4}] (The weak norm is not expaned) There is $M$ such that $%
\forall \delta ,n,g\in B_{s}$ $\ ||L_{\delta }^{n}g||_{w}\leq M||g||_{w}.$
\end{enumerate}

We will see that under these assumptions we can ensure that the invariant
measure of the system varies continuously (in the weak norm) when $L_{0}$ is
perturbed to $L_{\delta }$ for small values of $\delta $. We will also
provide a quantitative estimate for the modulus of continuity.

\begin{remark}
\label{conter}We remark that UF3, UF4 and UF1 together implies that $L_{0}$
\ eventually contracts exponentially fast the zero average space $V_{s}$.
Indeed let $f\in $ $V_{s}$, using the inequality and then the convergence to
equilibrium%
\begin{eqnarray*}
||L_{0}^{n+m}f||_{s} &\leq &A\lambda
_{1}^{n}||L_{0}^{m}f||_{s}+B||L_{0}^{m}f||_{w} \\
&\leq &A\lambda _{1}^{n}||L_{0}^{m}f||_{s}+BE\lambda _{2}^{m}||f||_{s} \\
&\leq &A\lambda _{1}^{n}(B+A)||f||_{s}+BE\lambda _{2}^{m}||f||_{s}
\end{eqnarray*}%
by which there are $n,m$ \ big enough that $||L_{0}^{n+m}f||_{s}\leq \frac{1%
}{2}||f||_{s}$.
\end{remark}

\subsection{Stability of fixed points, a general statement}

We state a general result on the stability of fixed points of Markov
operators satisfying certain assumptions. This will be a flexible tool to
obtain the stability of the invariant measure under small perturbations.

Let us consider two operators $L_{0}$ and $L_{\delta }$ preserving spaces $%
B_{s}\subseteq B_{w}\mathcal{\subseteq }SM(X)$ with norms $||\
||_{s},||~||_{w}$. Let us suppose that $f_{0},$ $f_{\delta }\in B_{s}$ are
fixed probability measures, respectively of $L_{0}$ and $L_{\delta }$.

\begin{lemma}
\label{gen}Suppose that:

\begin{description}
\item[a)] $||L_{\delta }f_{\delta }-L_{0}f_{\delta }||_{w}<\infty $

\item[b)] $\exists \,C_{i}~s.t.~\forall g\in \mathcal{B},~||L_{0}^{i}g||_{w}%
\leq C_{i}||g||_{w}$ (compare with UF4)
\end{description}

Then for each $N$%
\begin{equation}
||f_{\delta }-f_{0}||_{w}\leq ||L_{0}^{N}(f_{\delta
}-f_{0})||_{w}+||L_{\delta }f_{\delta }-L_{0}f_{\delta }||_{w}\sum_{i\in
\lbrack 0,N-1]}C_{i}.  \label{mainres}
\end{equation}
\end{lemma}

\begin{proof}
The proof is a direct computation%
\begin{eqnarray*}
||f_{\delta }-f_{0}||_{w} &\leq &||L_{\delta }^{N}f_{\delta
}-L_{0}^{N}f_{0}||_{w} \\
&\leq &||L_{0}^{N}f_{0}-L_{0}^{N}f_{\delta }||_{w}+||L_{0}^{N}f_{\delta
}-L_{\delta }^{N}f_{\delta }||_{w}
\end{eqnarray*}%
Hence%
\begin{equation*}
||f_{0}-f_{\delta }||_{\mathcal{B}}\leq ||L_{0}^{N}(f_{0}-f_{\delta
})||_{w}+||L_{0}^{N}f_{\delta }-L_{\delta }^{N}f_{\delta }||_{w}
\end{equation*}%
but%
\begin{equation*}
L_{0}^{N}-L_{\delta }^{N}=\sum_{k=1}^{N}L_{0}^{N-k}(L_{0}-L_{\delta
})L_{\delta }^{k-1}
\end{equation*}%
and%
\begin{eqnarray*}
(L_{0}^{N}-L_{\delta }^{N})f &=&\sum_{k=1}^{N}L_{0}^{N-k}(L_{0}-L_{\delta
})L_{\delta }^{k-1}f_{\delta } \\
&=&\sum_{k=1}^{N}L_{0}^{N-k}(L_{0}-L_{\delta })f_{\delta }
\end{eqnarray*}%
by item b)%
\begin{eqnarray*}
||(L_{0}^{N}-L_{\delta }^{N})f_{\delta }||_{w} &\leq
&\sum_{k=1}^{N}C_{N-k}||(L_{0}-L_{\delta })f_{\delta }||_{w} \\
&\leq &||(L_{0}-L_{\delta })f_{\delta }||_{w}\sum_{i\in \lbrack 0,N-1]}C_{i}
\end{eqnarray*}%
then%
\begin{equation*}
||f_{\delta }-f_{0}||_{w}\leq ||L_{0}^{N}(f_{0}-f_{\delta
})||_{w}+||(L_{0}-L_{\delta })f_{\delta }||_{w}\sum_{i\in \lbrack
0,N-1]}C_{i}.
\end{equation*}
\end{proof}

Now, let us apply the statement to our family of operators satisfying
assumptions UF 1,...,4. \ On can fix $C_{i}=M$. We have the following

\begin{proposition}
\label{stabi}Suppose $L_{\delta }$ is a uniform family of operators
satisfying UF1,...,4. $f_{0}$ is the unique invariant probability measure of 
$L_{0}$, $f_{\delta }$ is an invariant probability measure of $L_{\delta }$.
Then%
\begin{equation*}
||f_{\delta }-f_{0}||_{w}=O(\delta \log \delta ).
\end{equation*}
\end{proposition}

\begin{proof}
We remark that by the uniform Lasota Yorke inequality $||f_{\delta
}||_{s}\leq M$ are uniformly bounded.

Hence 
\begin{equation*}
||L_{\delta }f_{\delta }-L_{0}f_{\delta }||_{w}\leq \delta CM
\end{equation*}%
(see item a) of Lemma \ref{gen}). Moreover by UF$4$, $C_{i}\leq M_{2}.$

Hence%
\begin{equation*}
||f_{\delta }-f_{0}||_{w}\leq \delta CMM_{2}N+||L_{0}^{N}(f_{0}-f_{\delta
})||_{w}.
\end{equation*}%
Now by the exponential convergence to equilibrium of $L_{0}$%
\begin{eqnarray*}
||L_{0}^{N}(f_{\delta }-f_{0})||_{w} &\leq &C_{2}\rho _{2}^{N}||(f_{\delta
}-f_{0})||_{s} \\
&\leq &C_{2}\rho _{2}^{N}M
\end{eqnarray*}%
hence%
\begin{equation*}
||f_{\delta }-f_{0}||_{w}\leq \delta CMM_{2}N+C_{2}\rho _{2}^{N}M
\end{equation*}%
choosing $N=\left\lfloor \frac{\log \delta }{\log \rho _{2}}\right\rfloor $%
\begin{eqnarray}
||f_{\delta }-f_{0}||_{w} &\leq &\delta CMM_{2}\left\lfloor \frac{\log
\delta }{\log \rho _{2}}\right\rfloor +C_{2}\rho _{2}^{\left\lfloor \frac{%
\log \delta }{\log \rho _{2}}\right\rfloor }M \\
&\leq &\delta \log \delta CM_{2}M\frac{1}{\log \rho _{2}}+C_{2}\delta M. 
\notag
\end{eqnarray}
\end{proof}

\begin{remark}
We remark that in this statement we did not really use the Lasota Yorke
inequality in its full strength. We used it only to get $||f_{\delta
}||_{s}\leq M.$ \ Moreover the statement could be generalized to slower than
exponential convergence to equilibrium (see \cite{G}) obtaining other kinds
of continuity relations. In the following sections we apply these statements
to some classes of maps, we remark that the modulus of continuity $\delta
\log \delta $ is sharp for Piecewise Expanding map (see Section \ref{PW}).
\end{remark}

\begin{remark}
We remark that in UF2 the size of the perturbation is measured in the \
strong-weak norm, i.e. as an operator$:B_{s}\rightarrow B_{w}$. This allows
general perturbations (allowing to move discontinuities, like when making
small perturbation in the Skorokhod distance, see Eq. \ref{sko}) furthermore
even in the differentiable case, small perturbations of the map in the $%
C^{k} $ norm induce small perturbations of the associated transfer operator
in the strong-weak norm (see also Proposition \ref{ssd}). Measuring the size
of the perturbation in the strong-strong norm, will lead to stronger
results, like Lipschitz or differentiable stability (see Sections \ref%
{lipsec}, \ref{lrsec}). However the typical perturbations one is interested
to put on a deterministic dynamical system (perturbing a $C^{k}$map slightly
in the $C^{k}$ norm e.g.) are not small in the strong-strong norm.
\end{remark}

\subsection{Application to expanding maps}

In the previous section we considered the stability of the invariant measure
under small perturbations of the transfer operator.

There are many kinds of interesting perturbations to be considered. Two main
classes are deterministic or stochastic ones.

In the deterministic ones the transfer operator is perturbed by small
changes on the underlying dynamics (the map).

The stochastic ones can be of several kinds. The simplest one is the adding
of some noise perturbing the result of the deterministic dynamics at each
iteration (see \cite{L2} for some example and related estimations).

We now consider small deterministic perturbations of our expanding maps on $%
S^{1}.$ Let us consider an expanding map $T_{0}$ and a one parameter family $%
T_{\delta },$ $\delta \in \lbrack 0,1]$ of expanding maps of the circle
satisfying the properties stated at beginning of Section \ref{maps} and

\begin{description}
\item[UFM] $||T_{\delta }-T_{0}||_{C^{2}}\leq K\delta $ for some $K\in 
\mathbb{R}$.
\end{description}

To each of these maps it is associated a transfer operator $L_{\delta }$
acting on $W^{1,1}.$ We now prove that the transfer operators of a uniform
family of expanding maps satistfies the general property \textbf{UF2} and
this will allow to apply our general quantitative stability results.

\begin{proposition}
\label{ssd}If $L_{0}$ and $L_{\delta }$ are transfer operators of expanding
maps $T_{0}$ and $T_{\delta },$ satisfying UFM, then there is a $C\in 
\mathbb{R}$\ such that $\forall g\in W^{1,1}$:%
\begin{equation}
||(L_{\delta }-L_{0})f||_{1}\leq \delta C||f||_{W^{1,1}}
\end{equation}%
and assumption \textbf{UF2} is satisfied.
\end{proposition}

\begin{proof}
We have that the transfer operator is defined by the formula%
\begin{equation}
\lbrack L_{\delta }f](x)=\sum_{y\in T_{\delta }^{-1}(x)}\frac{f(y)}{%
|T_{\delta }^{\prime }(y)|}.
\end{equation}%
\begin{eqnarray*}
|[L_{\delta }f](x)-[L_{0}f](x)| &=&|\sum_{y\in T_{\delta }^{-1}(x)}\frac{f(y)%
}{|T_{\delta }^{\prime }(y)|}-\sum_{y\in T_{0}^{-1}(x)}\frac{f(y)}{%
|T_{0}^{\prime }(y)|}| \\
&\leq &|\sum_{y\in T_{\delta }^{-1}(x)}\frac{f(y)}{|T_{\delta }^{\prime }(y)|%
}-\sum_{y\in T_{\delta }^{-1}(x)}\frac{f(y)}{|T_{0}^{\prime }(y)|}|+ \\
&&+|\sum_{y\in T_{\delta }^{-1}(x)}\frac{f(y)}{|T_{0}^{\prime }(y)|}%
-\sum_{y\in T_{0}^{-1}(x)}\frac{f(y)}{|T_{0}^{\prime }(y)|}|.
\end{eqnarray*}

The first summand can be estimated as follows%
\begin{eqnarray*}
|\sum_{y\in T_{\delta }^{-1}(x)}\frac{f(y)}{|T_{\delta }^{\prime }(y)|}%
-\sum_{y\in T_{\delta }^{-1}(x)}\frac{f(y)}{|T_{0}^{\prime }(y)|}| &\leq
&|\sum_{y\in T_{\delta }^{-1}(x)}\frac{f(y)}{|T_{\delta }^{\prime }(y)|}[1-%
\frac{|T_{\delta }^{\prime }(y)|}{|T_{0}^{\prime }(y)|}]| \\
&\leq &D_{1}(\delta )|\sum_{y\in T_{\delta }^{-1}(x)}\frac{f(y)}{|T_{\delta
}^{\prime }(y)|}| \\
&\leq &D_{1}(\delta )|L_{\delta }f(x)|
\end{eqnarray*}%
where $D_{1}(\delta )=\sup_{y\in S^{1}}|1-\frac{|T_{\delta }^{\prime }(y)|}{%
|T_{0}^{\prime }(y)|}|$ and remark that $D_{1}=O(\delta )$. For second
summand let us denote $T_{\delta }^{-1}(x)=\{y_{1},...,y_{n}\}$ , $%
T_{0}^{-1}(x)=\{y_{1}^{0},...,y_{n}^{0}\}$. Let $\Delta
_{y}=\sup_{x,i}(|y_{i}-y_{i}^{0}|)$, \ since $\Delta _{y}=O(\delta )$ (see
Lemma \ref{tec} and its proof\ or \cite{GP} Lemma 3.2 for details)%
\begin{eqnarray*}
|\sum_{y\in T_{\delta }^{-1}(x)}\frac{f(y)}{|T_{0}^{\prime }(y)|}-\sum_{y\in
T_{0}^{-1}(x)}\frac{f(y)}{|T_{0}^{\prime }(y)|}| &\leq &|\sum_{i=1}^{n}\frac{%
f(y_{i})-f(y_{i}^{0})}{|T_{0}^{\prime }(y_{i})|}|+|%
\sum_{i=1}^{n}f(y_{i}^{0})(\frac{1}{|T_{0}^{\prime }(y_{i})|}-\frac{1}{%
|T_{0}^{\prime }(y_{i}^{0})|})| \\
&\leq &|\sum_{i=1}^{n}\frac{f(y_{i})-f(y_{i}^{0})}{|T_{0}^{\prime }(y_{i})|}%
|+|\sum_{i=1}^{n}\frac{f(y_{i}^{0})}{|T_{0}^{\prime }(y_{i}^{0})|}(\frac{%
|T_{0}^{\prime }(y_{i}^{0})|}{|T_{0}^{\prime }(y_{i})|}-1)| \\
&\leq &|\sum_{i=1}^{n}\frac{f(y_{i})-f(y_{i}^{0})}{|T_{0}^{\prime }(y_{i})|}%
|+D_{2}(\delta )|\sum_{i=1}^{n}\frac{f(y_{i}^{0})}{|T_{0}^{\prime
}(y_{i}^{0})|}| \\
&\leq &|\sum_{i=1}^{n}\frac{\int_{y_{i}^{0}}^{y_{i}}f^{^{\prime }}(t)dt}{%
|T_{0}^{\prime }(y_{i})|}|+D_{2}(\delta )|L_{0}f(x)|
\end{eqnarray*}%
where $D_{2}(\delta ):=\sup_{x,i}|\frac{|T_{0}^{\prime }(y_{i}^{0})|}{%
|T_{0}^{\prime }(y_{i})|}-1|=O(\delta )$. Hence 
\begin{eqnarray*}
||L_{\delta }f-L_{0}f||_{1} &\leq &D_{1}(\delta )||L_{\delta
}f(x)||_{1}+||\sum_{i=1}^{n}\frac{\int_{y_{i}^{0}}^{y_{i}}f^{^{\prime }}(t)dt%
}{|T_{0}^{\prime }(y_{i})|}||_{1}+D_{2}(\delta )||L_{0}f(x)||_{1} \\
&\leq &(D_{1}(\delta )+D_{2}(\delta ))||f(x)||_{1}+||\sum_{i=1}^{n}\frac{%
\int_{y_{i}^{0}}^{y_{i}}f^{^{\prime }}(t)dt}{|T_{0}^{\prime }(y_{i})|}||_{1}
\\
&\leq &O(\delta )||f(x)||_{1}+||\sum_{i=1}^{n}\frac{\int_{y_{i}-\Delta
_{y}}^{y_{i}}|f^{^{\prime }}(t)|dt}{|T_{0}^{\prime }(y_{i})|}||_{1} \\
&\leq &O(\delta )||f(x)||_{1}+||\sum_{i=1}^{n}\frac{[1_{[0,\Delta _{y}]}\ast
|f^{^{\prime }}|](y_{i})}{|T_{0}^{\prime }(y_{i})|}||_{1} \\
&\leq &O(\delta )||f(x)||_{1}+||L_{\delta }[1_{[0,\Delta _{y}]}\ast
|f^{^{\prime }}|]||_{1} \\
&\leq &O(\delta )||f(x)||_{1}+||1_{[0,\Delta _{y}]}\ast |f^{^{\prime
}}|||_{1} \\
&\leq &O(\delta )||f(x)||_{1}+||1_{[0,\Delta _{y}]}||_{1}||f^{^{\prime
}}||_{1} \\
&\leq &O(\delta )||f||_{W^{1.1}.}
\end{eqnarray*}

where $[1_{[0,\Delta _{y}]}\ast |f^{^{\prime }}|]$ stands for the
convolution function between the characteristic of the interval $[0,\Delta
_{y}]$ $(\func{mod}$~$1)$ and $|f^{\prime }|$. And the statement is proved.
\end{proof}

By the above estimates one can get a quantitative statistical stability
estimate for expanding maps and small deterministic perturbations.

\begin{corollary}
\label{de2}Let $T_{0}$ \ a $C^{2}$ espanding map and $T_{\delta }$ be a
family of expanding maps satisffying the assumption UFM \ above. Let $%
h_{\delta }$ be the family of invariant measures in $L^{1}$ \ for the maps $%
T_{\delta }$. Then%
\begin{equation*}
||h_{0}-h_{\delta }||_{1}=O(\delta \log \delta ).
\end{equation*}
\end{corollary}

\begin{proof}
It is easy to verify that for $\delta $ small enough the map $T_{\delta }$
satisfy the assumptions at beginning of Section \ref{maps} and the
associated transfer operators satisfy UF1,UF3,UF4 uniformly. By Proposition %
\ref{ssd}, UFM implies UF2. This allow to apply Proposition \ref{stabi} and
obtain the result.
\end{proof}

\begin{remark}
We will see in the next sections that in the case of expanding maps one can
be able to prove more precise estimates on the stability of the physical
measure to deterministic perturbations of the system. A stability result
similar \ to Corollary \ref{de2} also applies to suitable deterministic
perturbations of piecewise expanding maps (See Section \ref{PW}).
\end{remark}

\subsubsection{Further small Perturbation estimates}

In this subsection we show that small perturbations of an expanding map
induces a small perturbation of the associated transfer operator when
considered as acting from stronger to weaker Sobolev spaces.

\begin{proposition}
\label{prop14} If $L_{0}$ and $L_{\delta }$ are transfer operators of $C^{4}$
expanding maps $T_{0}$ and $T_{\delta },$ such that for some $K\in \mathbb{R}
$ 
\begin{equation*}
||T_{\delta }-T_{0}||_{C^{2}}\leq K\delta
\end{equation*}%
then there is a $C\in \mathbb{R}$\ such that $\forall f\in W^{1,1}$:%
\begin{equation}
||(L_{\delta }-L_{0})f||_{W^{1,1}}\leq \delta C||f||_{W^{2,1}}.  \label{2l}
\end{equation}

If furthermore%
\begin{equation*}
||T_{\delta }-T_{0}||_{C^{3}}\leq K\delta
\end{equation*}%
then%
\begin{equation}
||(L_{\delta }-L_{0})f||_{W^{2,1}}\leq \delta C||f||_{W^{3,1}}.  \label{3l}
\end{equation}
\end{proposition}

\begin{proof}
In Proposotion \ref{ssd} it is shown that if $L_{0}$ and $L_{\delta }$ are
transfer operators of expanding maps $T_{0}$ and $T_{\delta },$ such that
for some $K\in \mathbb{R}$ 
\begin{equation*}
||T_{\delta }-T_{0}||_{C^{2}}\leq K\delta
\end{equation*}%
then there is a $C\in \mathbb{R}$\ such that $\forall g\in W^{1,1}$:%
\begin{equation}
||(L_{\delta }-L_{0})f||_{1}\leq \delta C||f||_{W^{1,1}}
\end{equation}%
and the first line of \ref{2l} \ is established. From this we can also
recover the second line, indeed in our case we have an explicit formula for
the transfer operator:%
\begin{equation}
\lbrack L_{0}f](x)=\sum_{y\in T^{-1}(x)}\frac{f(y)}{|T_{0}^{\prime }(y)|}.
\end{equation}%
Considering that $T_{0}^{\prime }(y)=T_{0}^{\prime }(T_{0}^{(-1)}(x))$ we
can compute the derivative of $(\ref{pf})$ 
\begin{equation*}
(L_{0}f)^{^{\prime }}=\sum_{y\in T_{0}^{-1}(x)}\frac{1}{(T_{0}^{\prime
}(y))^{2}}f^{\prime }(y)-\frac{T_{0}^{\prime \prime }(y)}{(T_{0}^{\prime
}(y))^{3}}f(y).
\end{equation*}%
And similarly for $L_{\delta }.$ Note that%
\begin{equation}
(L_{0}f)^{^{\prime }}=L_{0}(\frac{1}{T_{0}^{\prime }}f^{\prime })-L_{0}(%
\frac{T_{0}^{\prime \prime }}{(T_{0}^{^{\prime }})^{2}}f).
\end{equation}

Hence 
\begin{eqnarray*}
||(L_{\delta }-L_{0})f||_{W^{1,1}} &\leq &||(L_{\delta
}-L_{0})f||_{L^{1}}+||((L_{\delta }-L_{0})f)^{\prime }||_{L^{1}} \\
&\leq &\delta C||f||_{W^{1,1}}+||(L_{\delta }-L_{0})(\frac{1}{T_{0}^{\prime }%
}f^{\prime })-(L_{\delta }-L_{0})(\frac{T_{0}^{\prime \prime }}{%
(T_{0}^{^{\prime }})^{2}}f)||_{L^{1}} \\
&&+||(\frac{1}{T_{\delta }^{\prime }}f^{\prime })-(\frac{1}{T_{0}^{\prime }}%
f^{\prime })+(\frac{T_{\delta }^{\prime \prime }}{(T_{\delta }^{^{\prime
}})^{2}}f)-(\frac{T_{0}^{\prime \prime }}{(T_{0}^{^{\prime }})^{2}}%
f)||_{L^{1}} \\
&\leq &\delta C(||f||_{W^{1,1}}+||\frac{1}{T^{\prime }}f^{\prime
}||_{W^{1,1}}+||\frac{T^{\prime \prime }}{(T^{^{\prime }})^{2}}f||_{W^{1,1}})
\\
&&+||T_{\delta }-T_{0}||_{C^{2}}||f||_{W^{1,1}} \\
&\leq &\delta C_{2}||f||_{W^{2,1}}
\end{eqnarray*}%
for some $C_{2}\geq 0$ depending on $T_{0}$ but not on $f.$ This proves $(%
\ref{2l}).$

To prove $(\ref{3l})$ one has to take a further derivative, applying $(\ref%
{preLY})$ again, and involving further derivatives of $T_{0}$ and $f$, but
leading to a similar computation and final result.
\end{proof}

\subsection{Uniform family of operators and uniform $V_{s}$ contraction\label%
{sunifcontr}}

Now we show how a suitable uniform family of nearby operators, not only has
a certain stability on the invariant measure as seen above, but also a
uniform rate of contraction of the space $V_{s}$ and hence a uniform
convergence to equilibrium and spectral gap (we remark that stability
results on the whole spectral picture are known, see \cite{L2}, \cite{B}
e.g.).

Consider two vector subspaces of the space of signed measures on $X$%
\begin{equation*}
B_{s}\mathcal{\subseteq }B_{w}\mathcal{\subseteq }SM(X),
\end{equation*}%
endowed with two norms, the strong norm $||~||_{s}$ on $B_{s}$ and the weak
norm $||~||_{w}$ on $B_{w}$, such that $||~||_{s}\geq ||~||_{w}$. Denote as
before by $V_{s},V_{w}$ the "zero average" strong and weak spaces.\textbf{\ }

\begin{proposition}[Uniform $V_{s}$ contraction for the uniform family of
operators]
\label{unifcont}Let us consider a one parameter family of operators $%
L_{\delta }$, $\delta \in \lbrack 0,1)$. Suppose that they satisfy
UF1,...UF4, then there are $\lambda _{4}<1$ and $A_{2},\delta _{0}\geq 0$
such that for each $\delta \leq \delta _{0}$ and $f\in V_{s}$%
\begin{equation}
||L_{\delta }^{k}f||_{s}\leq A_{2}\lambda _{4}^{k}||f||_{s}.
\label{contract3}
\end{equation}
\end{proposition}

We remark that the contraction rate of the zero average space for the
operator $L_{0}$ can be obtained simply by applying directly Remark \ref%
{conter}. Before the proof of Proposition \ref{unifcont} we need the
following

\begin{lemma}
\label{lemmapre}Suppose that $L_{0}$ satisfies a Lasota Yorke inequality 
\begin{equation*}
||L_{0}^{k}g||_{s}\leq A\lambda _{1}^{k}||g||_{s}+B||g||_{w}
\end{equation*}%
and the following holds

\begin{itemize}
\item $\forall g\in B_{s}$ $\ \ ||(L_{\delta }-L_{0})g||_{w}\leq C\delta
||g||_{s};$

\item there is $M\geq 0$ such that $\forall \delta ,n,g\in B_{s},$ $\
||L_{\delta }^{n}g||_{w}\leq M||g||_{w}$ $;$
\end{itemize}

then $L_{\delta }^{n}$ approximates $L_{0}^{n}$ in the following sense:
there are constants $E,D\geq 0$ such that $\forall g\in B_{s},\forall n\geq
0 $%
\begin{equation}
||(L_{\delta }^{n}-L_{0}^{n})g||_{w}\leq \delta (E||g||_{s}+nD||g||_{w}).
\end{equation}
\end{lemma}

\begin{proof}
Developing $L_{\delta }^{n}-L_{0}^{n}$ as a telescopic sum as done before%
\begin{eqnarray*}
||(L_{\delta }^{n}-L_{0}^{n})g||_{w} &\leq &\sum_{k=1}^{n}||L_{\delta
}^{n-k}(L_{\delta }-L_{0})L_{0}^{k-1}g||_{w}\leq M\sum_{k=1}^{n}||(L_{\delta
}-L_{0})L_{0}^{k-1}g||_{w} \\
&\leq &M\sum_{k=1}^{n}\delta C||L_{0}^{k-1}g||_{s} \\
&\leq &\delta MC\sum_{k=1}^{n}(A\lambda _{1}^{k-1}||g||_{s}+B||g||_{w}) \\
&\leq &\delta MC(\frac{A}{1-\lambda _{1}}||g||_{s}+Bn||g||_{w}).
\end{eqnarray*}
\end{proof}

\begin{proof}[Proof of Proposition \protect\ref{unifcont}]
Let us apply the Lasota Yorke inequality%
\begin{equation*}
||L_{\delta }^{n+m}f||_{s}\leq A\lambda _{1}^{n}||L_{\delta
}^{m}f||_{s}+B||L_{\delta }^{m}f||_{w}
\end{equation*}%
by the assumption UF3, and Lemma \ref{lemmapre}%
\begin{eqnarray*}
||L_{\delta }^{n+m}f||_{s} &\leq &A\lambda _{1}^{n}||L_{\delta
}^{m}f||_{s}+B[F\lambda _{2}^{m}||f||_{s}+\delta (E||f||_{s}+mD||f||_{w})] \\
&\leq &A\lambda _{1}^{n}[A\lambda _{1}^{m}||f||_{s}+B||f||_{w}]+B[F\lambda
_{2}^{m}||f||_{s}+\delta (E||f||_{s}+mD||f||_{w})].
\end{eqnarray*}%
If $n,m$ are big enough suitably chosen and $\delta $ small enough, then we
have that there is a $\lambda _{3}<1$ such that for each $f\in V_{s}$%
\begin{equation*}
||L_{\delta }^{n+m}f||_{s}\leq \lambda _{3}||f||_{s}
\end{equation*}%
thus there are $\lambda _{4}<1,$ $A_{2}\in \mathbb{R}$ such that for each $%
f\in V_{s}$ and $\delta $ small enough%
\begin{equation}
||L_{\delta }^{k}f||_{s}\leq A_{2}\lambda _{4}^{k}||f||_{s}.
\end{equation}
\end{proof}

The following Lemma establishes the continuity of the resolvent when the
operator is perturbed, and will be used in the following.

\begin{lemma}
\label{restab}Let us suppose that $B_{w}$ and $B_{s}$ are Banach spaces. Let
us consider a one parameter family of operators $L_{\delta }$, $\delta \in
\lbrack 0,1)$. Suppose that they satisfy UF1,...UF4. Consider the resolvent
operator $(Id-L_{\delta })^{-1}:V_{s}\rightarrow V_{w}$ 
\begin{equation*}
(Id-L_{\delta })^{-1}:=\sum_{k=0}^{\infty }L_{\delta }^{k}
\end{equation*}%
which is well defined and continuous thanks to $(\ref{contract3})$ and the
completeness of $B_{w}$. Under these assumptions we have 
\begin{equation*}
\lim_{\delta \rightarrow 0}||(Id-L_{\delta
})^{-1}-(Id-L_{0})^{-1}||_{V_{s}\rightarrow V_{w}}=0.
\end{equation*}
\end{lemma}

\begin{proof}
Let us fix $\epsilon >0$ \ and let us prove that for $\delta $ small enough $%
||(Id-L_{\delta })^{-1}-(Id-L_{0})^{-1}||_{V_{s}\rightarrow V_{w}}\leq
\epsilon $. By $(\ref{contract3})$ there is there is $\overline{\delta }>0$
and $n\geq 0$ such that for each $0\leq \delta \leq \overline{\delta }$ $%
\sum_{k=n}^{\infty }||L_{\delta }^{k}||_{V_{s}\rightarrow V_{w}}\leq \frac{%
\epsilon }{3}.$ By Lemma \ref{lemmapre} \ there is $\overline{\delta }>0$
such for each $k\leq n$, $||(L_{\delta }^{k}-L_{0}^{k})||_{B_{s}\rightarrow
B_{w}}\leq \frac{\epsilon }{3n}$. Then considering $g\in V_{s}$ with $%
||g||_{s}\leq 1$ we have for each $0\leq \delta \leq \overline{\delta }$%
\begin{eqnarray*}
||(Id-L_{\delta })^{-1}g-(Id-L_{0})^{-1}g||_{w} &\leq &||\sum_{k=n}^{\infty
}L_{\delta }^{k}g||_{w}+||\sum_{k=n}^{\infty
}L_{0}^{k}g||_{w}+||\sum_{k=0}^{n}(L_{\delta }^{k}-L_{0}^{k})g||_{w} \\
&\leq &\epsilon .
\end{eqnarray*}
\end{proof}

\subsection{Lipschitz continuity\label{lipsec}.}

Let us suppose $B_{w}$ and $B_{s}$ are a Banach spaces as above. Now we see
that exploiting the uniform contraction rate of $V_{s}$ and some further
assumptions we can prove Lipschitz dependence of the relevant invariant
measure under system perturbations (see \cite{L2} for similar reasoings on
expanding maps). Further work also lead to differentiable dependence (see
next section).

\begin{proposition}
\label{lip}Let us consider a uniform family $L_{\delta }$, $\delta \in
\lbrack 0,1)$ of operators satisfying UF1,...,UF4. Suppose that each
operator $L_{\delta }$ has a unique invariant probability measure $h_{\delta
}$ in $B_{s}$. Suppose furthermore that there is $C_{h_{0}}$ such that for $%
\delta $ small enough%
\begin{equation}
||(L_{\delta }-L_{0})h_{0}||_{s}\leq \delta C_{h_{0}}  \label{lop}
\end{equation}%
then the map $\delta \rightarrow $ $h_{\delta }$ is Lipschitz (\underline{%
with respect to the strong norm})%
\begin{equation*}
||h_{0}-h_{\delta }||_{s}\leq O(\delta ).
\end{equation*}
\end{proposition}

\begin{proof}
Denote $\Delta h=h_{\delta }-h_{0}$:%
\begin{eqnarray*}
(I-L_{\delta })\Delta h &=&(I-L_{\delta })(h_{\delta }-h_{0}) \\
&=&h_{\delta }-L_{\delta }h_{\delta }-h_{0}+L_{\delta }h_{0} \\
&=&(L_{\delta }-L_{0})h_{0}.
\end{eqnarray*}%
By the uniform contraction $(\ref{contract3})$ we have that \ $(I-L_{\delta
})$ is invertible on $V_{s},$ and $(I-L_{\delta })^{-1}=\sum_{0}^{\infty
}L_{\delta }^{i}$ is uniformly bounded \ and there is $M_{2}\geq 0$ such
that for $\delta $ small ehough $||(I-L_{\delta })^{-1}||_{V_{s}\rightarrow
V_{s}}\leq M_{2}$.

Since $(L_{\delta }-L_{0})h_{0}\in V_{s}$, then%
\begin{equation*}
\Delta h=(I-L_{\delta })^{-1}(L_{\delta }-L_{0})h_{0}.
\end{equation*}%
and since $||(L_{\delta }-L_{0})h_{0}||_{s}\leq \delta C_{h_{0}}$%
\begin{equation}
||\Delta h||_{s}\leq \delta M_{2}C_{h_{0}}.
\end{equation}%
hence we have the statement.
\end{proof}

\begin{remark}
\label{rmk122}We remark that $(\ref{lop})$ contains a "small pertutbation"
estimate similar to UF2, but on the strong ropology. The result can be
easily applied to a suitable family of \ expanding maps satisfying
UF1,...,UF4 and $(\ref{lop})$\ obtaining Lipschitz statistical stability on
the strong norm for this family of maps. In Section \ref{lrex} we show a set
of easy to be verified conditions on the family implying \ $(\ref{lop})$.
\end{remark}

\section{Some general Linear Response statements\label{lrsec}}

In this section we prove that when a systems has fast enough convergence to
equilibrium and \emph{it is perturbed smoothly in with respect to the strong
norm }then its invariant measure (and then its statistical properties)
changes in a smooth way. This is called Linear Response. We refer to \cite%
{Baicm} for a general introduction to this kind of problems and a survey of
recent results. In the following we show a general and simple result
(Theorems \ref{wlr}, \ref{LR}) allowing to prove linear response for a quite
large set of systems and perturbations. A different approach to prove a
Linear Response result (even for higher derivatives) was also provided in 
\cite{Sed}. Other general results can be found in \cite{BGN} .

Let $X$ be a compact metric space. Let us consider some complete normed
vector subpaces \ $(B_{ss},\Vert ~\Vert _{ss})\subseteq (B_{s},\Vert ~\Vert
_{s})\subseteq (B_{w},\Vert ~\Vert _{w})\subseteq SM(X)$ of the space of
signed Borel measures on $X$, $SM(X)$, with with norms satisfying%
\begin{equation*}
\Vert ~\Vert _{w}\leq \Vert ~\Vert _{s}\leq \Vert ~\Vert _{ss}.
\end{equation*}%
We will assume that the linear form $\mu \rightarrow \mu (X)$ is continuous
on $B_{i}$, for $i\in \{ss,s,w\}$. We will consider Markov operators acting
on these spaces, the following (closed) spaces $V_{ss}\subseteq
V_{s}\subseteq V_{w}$ of \ zero average measures defined as:%
\begin{equation*}
V_{i}:=\{\mu \in B_{i}|\mu (X)=0\}
\end{equation*}%
where $i\in \{ss,s,w\}$, will play an important role. If $A,B$ are two
normed vector spaces and $L:A\rightarrow B$ we denote the mixed norm $\Vert
L\Vert _{A\rightarrow B}$ as 
\begin{equation*}
\Vert L\Vert _{A\rightarrow B}:=\sup_{f\in A,\Vert f\Vert _{A}\leq 1}\Vert
Lf\Vert _{B}.
\end{equation*}%
Let us consider a system having a transfer operator $L_{0},$ some $\delta
_{0}>0$ \ and a family of "nearby" system $L_{\delta }$ with $\delta \in
\lbrack 0,\overline{\delta }]$ and suppose the operators $L_{\delta }$
preserve the spaces: $L_{\delta }(B_{ss})\subset B_{ss}$, $L_{\delta
}(B_{s})\subset B_{s}$ and $L_{\delta }(B_{w})\subset B_{w}$.

We now now enter more in details about what we mean by linear response and
the motivation behind the search of results establishing the differentiable
behavior of the invariant measure with respect to a small perturbation. Let
us consider a family of dynamical systems $(X,T_{\delta })$ and suppose that 
$\mu _{\delta }$ is a physical measure for $T_{\delta }$. Suppose that the
invariant measure varies in a smooth way, and after a small perturbation $p$
of size $\delta $ we know that in some sense%
\begin{equation}
\frac{\mu _{\delta }-\mu _{0}}{\delta }\rightarrow \dot{\mu}.  \label{1m}
\end{equation}%
Consider an observable $f$ \ and its time average $\underset{n\rightarrow
\infty }{\lim }\frac{S_{n}^{f}(x)}{n}$ (see $(\ref{Birkhoff})$). If the
topology in which $(\ref{1m})$ converges is strong enough, on the basin of $%
f $ we get%
\begin{equation*}
\frac{d(\underset{n\rightarrow \infty }{\lim }\frac{S_{n}^{f}(x)}{n})}{%
d\delta }=\frac{\int f~d\mu _{\delta }-\int f~d\mu _{0}}{\delta }\rightarrow
\int f~d\dot{\mu}.
\end{equation*}

This shows how the Linear Response controls the behavior of long time
averages of observables and the statistical behavior of the system under
perturbations.

Let us see \QTR{frametitle}{an heuristic argument to compute a Linear
Response formula, giving \ }$\dot{\mu}$ \ as a function of the initial
system and of the perturbation applied\QTR{frametitle}{.}

By using that $\mu _{0}$ and $\mu _{\delta }$ are fixed points of their
respective operators \ $L_{0},L_{\delta }$ we obtain that%
\begin{equation}
(Id-L_{0})\frac{\mu _{\delta }-\mu _{0}}{\delta }=\frac{1}{\delta }%
(L_{\delta }-L_{0})f_{\delta }.  \label{5}
\end{equation}%
If the convergence to equilibrium is fast enough, like done in the proof of
Proposition \ref{lip} one can prove that the resolvent $(Id-L_{0})^{-1}$ is
well defined and continuous. By applying the resolvent to $(\ref{5})$ \ one
gets 
\begin{eqnarray*}
(Id-L_{0})^{-1}(Id-L_{0})\frac{\mu _{\delta }-\mu _{0}}{\delta }
&=&(Id-L_{0})^{-1}\frac{L_{\delta }-L_{0}}{\delta }\mu _{\delta } \\
&=&(Id-L_{0})^{-1}\frac{L_{\delta }-L_{0}}{\delta }\mu _{0} \\
&&+(Id-L_{0})^{-1}\frac{L_{\delta }-L_{0}}{\delta }(\mu _{\delta }-\mu _{0})
\end{eqnarray*}

and then 
\begin{equation*}
\frac{\mu _{\delta }-\mu _{0}}{\delta }=(Id-L_{0})^{-1}\frac{L_{\delta
}-L_{0}}{\delta }\mu _{0}+(Id-L_{0})^{-1}\frac{L_{\delta }-L_{0}}{\delta }%
(\mu _{\delta }-\mu _{0}).
\end{equation*}

if we choose the right topologies:

\begin{itemize}
\item $\frac{\mu _{\delta }-\mu _{0}}{\delta }$ tends to the Linear response 
$\dot{\mu}$.

\item $(Id-L_{0})^{-1}\frac{L_{\delta }-L_{0}}{\delta }f_{0}$ tends to $%
(Id-L_{0})^{-1}\dot{L}f_{0}$

\item $(Id-L_{0})^{-1}\frac{L_{\delta }-L_{0}}{\delta }(f_{\delta }-f_{0})$
tends to zero
\end{itemize}

\QTR{frametitle}{We now show how to make this argument rigorous proving a
general theorem and showing examples of application. }

We remark that the first result we are going to see (Theorem \ref{wlr})
applies to systems having less than exponential convergence to equilibrium.
Examples of application of this statement in this case are outside the scope
of these lectures. In next section we apply the statement to expanding maps.

We recall that the speed of convergence to equilibrium of the system\ (see
Definition $(\ref{conve})$) is measured by the speed of contraction to $0$
of the zero average spaces $V_{s}$ and $V_{w}$ we suppose that the topology
of $B_{s}$ is strong enough so that $V_{s}$ is a closed subspace of $B_{s}$.

\begin{theorem}[Linear Response, summable decay]
\label{wlr}Let $L_{\delta }$ a family of Markov operators preserving $%
B_{ss}, $ $B_{s}$ and $B_{w}$ as above. Suppose that for each $\delta \in
\lbrack 0,\overline{\delta })$ there is $f_{\delta }\in B_{ss}$ \ which is a
fixed probability measure of $L_{\delta }.$ Suppose the system satisfy the
following:

\begin{enumerate}
\item (summable convergence) there is $\phi $ such that $\sum \phi
(n)<\infty $, such that $L_{0}$ has \emph{convergence to equilibrium }with
respect $B_{s}$, $B_{w}$ and speed $\phi $. (remark that by this $f_{0}$ is
the unique fixed point of $L_{0}$).

\item (strong statistical stability) $\lim_{\delta \rightarrow 0}||f_{\delta
}-f_{0}||_{ss}=0;$

\item (derivative operator) suppose there is $\dot{L}:B_{ss}\rightarrow
V_{s} $ continuous such that for each $f\in B_{ss}$%
\begin{equation*}
\underset{\delta \rightarrow 0}{\lim }||\frac{(L_{\delta }-L_{0})}{\delta }f-%
\dot{L}f||_{s}=0
\end{equation*}%
then 
\begin{equation*}
\lim_{\delta \rightarrow 0}||\frac{f_{\delta }-f_{0}}{\delta }-(1-L_{0})^{-1}%
\dot{L}f_{0}||_{w}=0
\end{equation*}%
where $(1-L_{0})^{-1}:=\sum_{0}^{\infty }L_{0}^{i}$ is a continuous operator:%
$V_{s}\rightarrow V_{w}$.
\end{enumerate}
\end{theorem}

\begin{proof}
Let $f\in B_{s}$. Since $\sum_{0}^{\infty }\phi (n)<\infty ,$ then $%
(1-L_{0})^{-1}f:=\sum_{0}^{\infty }L_{0}^{i}f$ converges in $B_{w}$ and
defines a continuous operator $V_{s}\rightarrow \ V_{w}$. It also holds $%
||(1-L_{0})^{-1}||_{B_{s}\rightarrow B_{w}}\leq \sum_{0}^{\infty }\phi (n)$.

Denote $\Delta f=f_{\delta }-f_{0}$%
\begin{eqnarray*}
(I-L_{0})\frac{\Delta f}{\delta } &=&(I-L_{0})\frac{f_{\delta }-f_{0}}{%
\delta } \\
&=&\frac{1}{\delta }(f_{\delta }-L_{0}f_{\delta }-f_{0}+L_{0}f_{0}) \\
&=&\frac{1}{\delta }(L_{\delta }-L_{0})f_{\delta }. \\
(1+L_{0}+...+L_{0}^{n})(I-L_{0})\frac{\Delta f}{\delta }
&=&(1+L_{0}+...+L_{0}^{n})\frac{L_{\delta }-L_{0}}{\delta }f_{\delta } \\
\frac{\Delta f}{\delta }-L_{0}^{n+1}\frac{\Delta f}{\delta }
&=&(1+L_{0}+...+L_{0}^{n})\frac{L_{\delta }-L_{0}}{\delta }f_{\delta }
\end{eqnarray*}%
\begin{eqnarray*}
\frac{\Delta f}{\delta }-L_{0}^{n+1}\frac{\Delta f}{\delta }
&=&(1+L_{0}+...+L_{0}^{n})\frac{L_{\delta }-L_{0}}{\delta }(f_{\delta
}+f_{0}-f_{0}) \\
(1-L_{0}^{n+1})\frac{\Delta f}{\delta } &=&(1+L_{0}+...+L_{0}^{n})\frac{%
L_{\delta }-L_{0}}{\delta }f_{0}+ \\
&&(1+L_{0}+...+L_{0}^{n})\frac{L_{\delta }-L_{0}}{\delta }(f_{\delta
}-f_{0}).
\end{eqnarray*}%
Letting $n\rightarrow \infty ,$ since $\Delta f\in V_{s}$, by convergence to
equilibrium, it holds that $L_{0}^{n+1}\frac{\Delta f}{\delta }\rightarrow 0$
in the weak norm. Thus 
\begin{equation*}
\frac{\Delta f}{\delta }=(1-L_{0})^{-1}\frac{L_{\delta }-L_{0}}{\delta }%
f_{0}+(1-L_{0})^{-1}\frac{L_{\delta }-L_{0}}{\delta }(f_{\delta }-f_{0})
\end{equation*}%
as elements of $B_{w}.$ Now by the strong statistical stability%
\begin{equation*}
||(1-L_{0})^{-1}\frac{L_{\delta }-L_{0}}{\delta }(f_{\delta
}-f_{0})||_{w}\leq (\sum_{i}||L_{0}^{i}||_{B_{s}\rightarrow B_{w}})~||\dot{L}%
||_{B_{ss}\rightarrow B_{s}}||f_{\delta }-f_{0}||_{ss}\rightarrow 0
\end{equation*}%
then in the weak norm, as $\delta \rightarrow 0$%
\begin{equation*}
\dot{\mu}=\lim_{\delta \rightarrow 0}\frac{\Delta f}{\delta }=(1-L_{0})^{-1}%
\dot{L}f_{0}.
\end{equation*}
\end{proof}

\begin{remark}
\label{lr1}Theorem \ref{wlr} is quite abstract and it is stated for families
of operators. In particular it may be adapted both to stochastic or
deterministic perturbations of (deterministic of stochastic) systems. One
key point is the existence of the derivative operator (assumption 3) ). The
form of this operator is strictly related to the kind of perturbation
considered.\newline
In the following section we will compute this operator for smooth
perturbations of expanding maps, see \cite{BGN}, \cite{GS} or \cite{GG} for
the derivative operator in a stochastic case.
\end{remark}

\begin{remark}
\label{lr3}If $||L_{0}^{n}(g)||_{s}\leq \phi (n)||g||_{s}$ with $\phi (n)$
summable (prove that this is equivalent to exponential contraction of the
zero average space). Then $(1-L_{0})^{-1}$ is defined $B_{s}\rightarrow \
B_{s}$ and the conclusion of the theorem is reinforced: \ $\lim_{\delta
\rightarrow 0}||\frac{f_{\delta }-f_{0}}{\delta }-(1-L_{0})^{-1}\dot{L}%
f_{0}||_{s}=0.$
\end{remark}

Theorem \ref{wlr} requires a weak assumption on the decay of correlation,
which is only assumed to be summable and only checked at the unperturbed
transfer operator $L_{0}$, on the other hand it requires the strong
statistical stability of the system. $\lim_{\delta \rightarrow 0}||f_{\delta
}-f_{0}||_{ss}=0.$

Sometimes is easy to verify some uniform convergence to equilibrium for the
family of perturbed systems, like in Section \ref{sunifcontr}. We then show
a linear response result exploting this uniform estimate in the place of the
strong statistical stability.

\begin{theorem}
\label{LR}Suppose that for $\delta \in \lbrack 0,\bar{\delta})$ there is a
probability measure $\ v_{\delta }\in B_{s}$ such that 
\begin{equation*}
L_{\delta }v_{\delta }=v_{\delta }
\end{equation*}%
and that there is $\dot{L}v_{0}\in B_{s}$ such that%
\begin{equation*}
\lim_{\delta \rightarrow 0}||\frac{L_{\delta }-L_{0}}{\delta }v_{0}-\dot{L}%
v_{0}||_{s}=0.
\end{equation*}%
\ Suppose the resolvent operator is defined and bounded from $V_{s}$ to $%
V_{w},$ $||(Id-L_{\delta })^{-1}||_{V_{s}\rightarrow V_{w}}=M<+\infty $ and 
\begin{equation*}
\lim_{\delta \rightarrow 0}||(Id-L_{\delta
})^{-1}-(Id-L_{0})^{-1}||_{V_{s}\rightarrow V_{w}}=0.
\end{equation*}%
Then 
\begin{equation*}
\lim_{\delta \rightarrow 0}||\frac{v_{\delta }-v_{0}}{\delta }%
-(Id-L_{0})^{-1}\dot{L}v_{0}||_{V_{w}}=0.
\end{equation*}
\end{theorem}

\begin{proof}
We have that for each $\delta \in \lbrack 0,\bar{\delta})$, $v_{\delta }$ is
a fixed point of $L_{\delta }.$ Using this we get%
\begin{eqnarray*}
(Id-L_{\delta })\frac{v_{\delta }-v_{0}}{\delta } &=&\frac{v_{\delta }-v_{0}%
}{\delta }-\frac{L_{\delta }v_{\delta }-L_{\delta }v_{0}}{\delta } \\
&=&\frac{-v_{0}+L_{\delta }v_{0}}{\delta } \\
&=&\frac{1}{\delta }(L_{\delta }-L_{0})v_{0}.
\end{eqnarray*}%
We remark that for each $\delta ,$ $L_{\delta }$ preserves $V_{s}$. Since $%
\forall \delta >0$, $\frac{L_{\delta }-L_{0}}{\delta }v_{0}\in V_{s}$ \ and $%
(Id-L_{\delta })^{-1}:V_{s}\rightarrow V_{w}$ is a bounded operator, we can
apply the resolvent both sides and get%
\begin{eqnarray}
\frac{v_{\delta }-v_{0}}{\delta } &=&(Id-L_{\delta })^{-1}\frac{L_{\delta
}-L_{0}}{\delta }v_{0}  \label{xx1} \\
&=&(Id-L_{\delta })^{-1}\frac{L_{\delta }-L_{0}}{\delta }%
v_{0}-(Id-L_{0})^{-1}\frac{L_{\delta }-L_{0}}{\delta }v_{0}  \notag \\
&&+(Id-L_{0})^{-1}\frac{L_{\delta }-L_{0}}{\delta }v_{0}.
\end{eqnarray}%
Since $||(Id-L_{\delta })^{-1}-(Id-L_{0})^{-1}||_{V_{s}\rightarrow
V_{w}}\rightarrow 0$ \ we have%
\begin{eqnarray*}
||[(Id-L_{\delta })^{-1}-(Id-L_{0})^{-1}]\frac{L_{\delta }-L_{0}}{\delta }%
v_{0}||_{w} &\leq &||(Id-L_{\delta
})^{-1}-(Id-L_{0})^{-1}||_{V_{s}\rightarrow V_{w}}||\frac{L_{\delta }-L_{0}}{%
\delta }v_{0}||_{s} \\
&\rightarrow &0.
\end{eqnarray*}%
Since $\lim_{\delta \rightarrow 0}\frac{L_{\delta }-L_{0}}{\delta }v_{0}$
converges in $V_{s},$ then $\ (\ref{xx1})$ \ implies that in the $B_{w}$
topology%
\begin{eqnarray*}
\lim_{\delta \rightarrow 0}\frac{v_{\delta }-v_{0}}{\delta } &=&\lim_{\delta
\rightarrow 0}~(Id-L_{0})^{-1}\frac{L_{\delta }-L_{0}}{\delta }v_{0} \\
&=&(Id-L_{0})^{-1}\lim_{\delta \rightarrow 0}\frac{L_{\delta }-L_{0}}{\delta 
}v_{0} \\
&=&(Id-L_{0})^{-1}[\dot{L}v_{0}].
\end{eqnarray*}
\end{proof}

\subsection{Applying the general theorems to expanding maps}

In this subsection we show how to get a linear response result for small
dterministic perturbations of expanding maps, applying Theorem \ref{LR}. Let 
$T_{\delta }:S^{1}\rightarrow S^{1}$ \ be a family of $C^{3}$ expanding
orientation preserving maps of the circle $X$ where $\delta \in (0,\overline{%
\delta })$. Let us suppose that the dependence of the family on $\delta $ is
differentiable at $0$, hence can be written 
\begin{equation}
T_{\delta }(x)=T_{0}(x)+\delta \dot{T}(x)+o_{C^{3}}(\delta )~for~x\in X
\label{fam}
\end{equation}%
where $\dot{T}\in C^{3}(X,\mathbb{R})$, and $o_{C^{3}}(\delta )$ denotes a
term whose $C^{3}$ norm tends to zero faster than $\delta $, as $\delta
\rightarrow 0$.\footnote{%
More precisely we say that $T_{\delta }$ is a differentiable family of $%
C^{3} $ expanding maps if there exists $\epsilon \in C^{3}(X,\mathbb{R})$
such that $\Vert (T_{\delta }-T_{0})/\delta -\epsilon \Vert
_{C^{3}}\rightarrow 0$ as $\delta \rightarrow 0$, where 
\begin{equation*}
\Vert f(x)\Vert _{C^{3}}=\sup_{x\in X}|f(x)|+\sup_{x\in X}|f^{\prime
}(x)|+\sup_{x\in X}|f^{\prime \prime }(x)|+\sup_{x\in X}|f^{\prime \prime
\prime }(x)|
\end{equation*}%
is the usual norm on $C^{3}$ functions.}We will see that if $\overline{%
\delta }$ is small enough, UF1,...,UF4 \ are satisfied by the associated
transfer operators, when applied to suitable Sobolev spaces $W^{k,1}$ and
then, provided we establish the existence of the derivative operator,
Theorem \ref{LR} can be applied.

\begin{definition}
\label{ufm}A set $A_{M,L}$ of expanding maps is called a \emph{uniform C}$%
^{k}$\emph{\ family} with parameters $M\geq 0$ and $L>1$ if it satisfies
uniformly the expansiveness and regularity condition: $\forall T\in A_{M,L}$%
\begin{equation*}
||T||_{C^{k}}\leq M,~\inf_{x\in S^{1}}|T^{\prime }(x)|\geq L.
\end{equation*}
\end{definition}

We already proved in Section \ref{LY} \ Lasota Yorke inequalities for the
transfer operators associated to such maps, acting on Sobolev spaces, with a
similar proof on can obtain a general result (see \cite{GS}, Lemma 29 \ and
its proof).

\begin{lemma}
\label{Lemsu} Let $A_{M,L}$ be a uniform $C^{k}$ family of expanding maps,
the transfer operators $L_{T}$ associated to a $T\in A_{M,L}$ satisfy a
uniform Lasota-Yorke inequality on $W^{i,1}(\mathbb{S}^{1})$: for each $%
1\leq i\leq k-1$ there are $\alpha <1$, $A_{i},~B_{i}\geq 0$ such that for
each $n\geq 0,$ $T\in A_{M,L}$%
\begin{eqnarray}
||L_{T}^{n}f\Vert _{W^{i-1,1}} &\leq &A_{i}||f\Vert _{W^{i-1,1}} \\
||L_{T}^{n}f\Vert _{W^{i,1}} &\leq &\alpha ^{in}\Vert f\Vert
_{W^{i,1}}+B_{i}\Vert f\Vert _{W^{i-1,1}}.
\end{eqnarray}
\end{lemma}

From this last result, and the compact immmersion of \ $W^{k,1}$ in $%
W^{k-1,1}$ it is classically deduced that the transfer operator $L_{T}$ of a 
$C^{k}$ expanding map $T$ has spectral gap on each $W^{i,1}(\mathbb{S}^{1}),$
with $1\leq i\leq k-1$.

Since a family of maps $T_{\delta }$ as in $(\ref{fam})$ for $\delta $ small
enough is a uniform family in the sense of Definition \ref{ufm}, then Lemma %
\ref{Lemsu} applies to their transfer operators, and then UF1 and UF4 are
verified considering $~B_{s}=W^{1,1}$and $B_{w}=L^{1}$. UF2 \ is also
verified by Proposition \ref{ssd} and UF3 is proved in Section \ref{CEM}.
Since Proposition \ref{restab} allow to establish the stability of the
resolvent for these perturbations, to apply Theorem \ref{LR} \ we only need
to verify the existence of the derivative transfer operator.

\subsubsection{The derivative operator and linear response for circle
expanding maps\label{lrex}}

In the following we show how to obtain the existence of the derivative
operator for a smooth family of expanding maps. Let \ us consider $T_{\delta
}:S^{1}\rightarrow S^{1}$ \ be a family of $C^{3}$ expanding maps as in $(%
\ref{fam})$.

We hence have a family of $C^{3}$ expanding maps, and each one of them has a
invariant density $f_{\delta }$ in $C^{2}$ Remark \ref{c3}. The following
proposition present a detailed description of the structure of the operator $%
\dot{L}:C^{2}\rightarrow {W^{1,1}}$ in our case.

\begin{proposition}
\label{mainprop} Let $w\in C^{2}(S^{1},\mathbb{R})$. For each $x\in S^{1}$
we can write 
\begin{equation}
\dot{L}w(x)=\lim_{\delta \rightarrow 0}\left( \frac{L_{\delta }w(x)-L_{0}w(x)%
}{\delta }\right) =-L_{0}\left( \frac{w\dot{T}^{\prime }}{T_{0}^{\prime }}%
\right) (x)-L_{0}\left( \frac{\dot{T}w^{\prime }}{T_{0}^{\prime }}\right)
(x)+L_{0}\left( \frac{\dot{T}T_{0}^{\prime \prime }}{T_{0}^{\prime 2}}%
w\right) (x)
\end{equation}%
and the convergence is also in the $C^{1}$ topology.
\end{proposition}

Before presenting the proof of Proposition \ref{mainprop} we state a
technical lemma.

\begin{lemma}
\label{tec} Let $T_{\delta }:S^{1}\rightarrow S^{1}$, where $\delta \in (0,%
\hat{\delta})$ be a family of $C^{2}$ expanding maps. Let us suppose that
the dependence of the family on $\delta $ is differentiable at $0$ in the
following sense%
\begin{equation}
T_{\delta }(x)=T_{0}(x)+\delta \dot{T}(x)+o_{C^{2}}(\delta )  \label{zzz}
\end{equation}%
where $\dot{T}\in C^{2}(S^{1},\mathbb{R})$, and $o_{C^{k}}(\delta )$ denotes
a function $f(x)$ $\in C^{k}(S^{1})$ satisfying $\lim_{\delta \rightarrow 0}%
\frac{||f||_{C^{k}}}{\delta }=0$.

Under these assumptions, if $y_{i}^{\delta }\in T_{\delta }^{-1}(x)$ then
when $\delta \rightarrow 0$ we can expand 
\begin{equation*}
y_{i}^{\delta }=y_{i}^{0}+\delta \left( -\frac{\dot{T}(y_{i}^{0})}{%
T_{0}^{\prime }(y_{i}^{0})}\right) +o_{C^{1}}(\delta ).
\end{equation*}
\end{lemma}

\begin{proof}[Proof of Lemma \protect\ref{tec}]
Let us fix $x\in S^{1}$ and write 
\begin{equation}
y_{i}^{\delta }(x)=y_{i}^{0}(x)+\delta \epsilon _{i}(x)+F_{i}(\delta ,x)
\label{sss}
\end{equation}%
where for each $x$, $F_{i}(\delta ,x)=o(\delta )$ as $\delta \rightarrow 0$.
We will show that $\epsilon _{i}(x)=-\frac{\dot{T}(y_{i}^{0}(x))}{%
T_{0}^{\prime }(y_{i}^{0}(x))}$ for each $x,$ then we will show that $%
F_{i}(\delta ,x)=o_{C^{1}}(\delta )$.

For the first claim, let us fix $x\in S^{1}.$ Substituting $(\ref{zzz})$
into the identity $T_{\delta }(y_{i}^{\delta }(x))=x$ we can expand%
\begin{eqnarray}
x &=&T_{\delta }(y_{i}^{\delta }(x)) \\
&=&T_{0}(y_{i}^{\delta }(x))+\delta \dot{T}(y_{i}^{\delta }(x))+E(\delta ,x)
\end{eqnarray}%
where $E(\delta ,x)=o_{C^{2}}(\delta )$ and then by $(\ref{sss})$ 
\begin{eqnarray}
x &=&T_{0}(y_{i}^{0}(x)+\delta \epsilon _{i}(x)+F_{i}(\delta ,x))
\label{long} \\
&&\ \ \ +\delta \dot{T}(y_{i}^{0}(x)+\delta \epsilon _{i}(x)+F_{i}(\delta
,x))+E(\delta ,x).
\end{eqnarray}%
Since $T_{0}\in C^{2}$ we can write the first term in the right hand side of 
$(\ref{long})$ as%
\begin{eqnarray}
T_{0}(y_{i}^{0}(x)+\delta \epsilon _{i}(x)+F_{i}(\delta ,x))
&=&T_{0}(y_{i}^{0}(x))+T_{0}^{\prime }(y_{i}^{0}(x))(\delta \epsilon
_{i}(x)+F_{i}(\delta ,x)) \\
&+&o(\delta \epsilon _{i}(x)+F_{i}(\delta ,x)).  \label{aaa}
\end{eqnarray}

Since $\dot{T}\in C^{2}$ we can write the second term of $(\ref{long})$ as%
\begin{eqnarray}
\delta \dot{T}(y_{i}^{0}(x)+\delta \epsilon _{i}(x)+F_{i}(\delta ,x))
&=&\delta \dot{T}(y_{i}^{0}(x)) \\
&&+\delta \dot{T}^{\prime }(y_{i}^{0}(x))(\delta \epsilon
_{i}(x)+F_{i}(\delta ,x))+\delta o((\delta \epsilon _{i}(x)+F_{i}(\delta
,x))).
\end{eqnarray}%
and use that $T_{0}(y_{i}^{0}(x))=x$ to cancel terms on either side of $(\ref%
{long})$ to get that 
\begin{eqnarray*}
0 &=&T_{0}^{\prime }(y_{i}^{0}(x))(\delta \epsilon _{i}(x)+F_{i}(\delta
,x))+\delta \dot{T}(y_{i}^{0}(x))+\delta \dot{T}^{\prime
}(y_{i}^{0}(x))(\delta \epsilon _{i}(x)+F_{i}(\delta ,x)) \\
&&+o(\delta \epsilon _{i}(x)+F_{i}(\delta ,x)).
\end{eqnarray*}

For each fixed $x,$ as $\delta \rightarrow 0$ we can then identify the
relation among the first order terms (dividing by $\delta $ and letting $%
\delta \rightarrow 0$) as 
\begin{equation*}
\delta T_{0}^{\prime }(y_{i}^{0}(x))\epsilon _{i}(x)+\delta \dot{T}%
(y_{i}^{0}(x))=0
\end{equation*}%
giving $\epsilon _{i}(x)=-\frac{\dot{T}(y_{i}^{0})}{T_{0}^{\prime
}(y_{i}^{0})}$. \ Now we have $F_{i}(\delta ,x)=y_{i}^{\delta
}(x)-y_{i}^{0}(x)+\delta \frac{\dot{T}(y_{i}^{0}(x))}{T_{0}^{\prime
}(y_{i}^{0}(x))}.$ We remark since each $T_{\delta }$ is uniformly $C^{2}$
and $T^{^{\prime }}>1$ we have that $\frac{\partial F_{i}(\delta ,x)}{%
\partial x}$ and $\frac{\partial ^{2}F_{i}(\delta ,x)}{\partial x^{2}}$ are
uniformly bounded for each $\delta ,$ $x$ and $i$. \ \ Thus $||\frac{%
F_{i}(\delta ,x)}{\delta }||_{C^{1}}\rightarrow 0$ as $\delta \rightarrow 0$
and $F_{i}(\delta ,x)=o_{C^{1}}(\delta ).$
\end{proof}

We now return to the proof of Proposition \ref{mainprop}.

\begin{proof}[Proof of Proposition \protect\ref{mainprop}]
Let us again denote by $\{y_{i}^{\delta }\}_{i=1}^{d}:=T_{\delta }^{-1}(x)$
and $\{y_{i}^{0}\}_{i=1}^{d}:=T_{0}^{-1}(x)$ the $d$ preimages under $%
T_{\delta }$ and $T_{0}$, respectively, of a point $x\in X$. Furthermore, we
assume that the indexing is chosen so that $y_{i}^{\delta }$ is a small
perturbation of $y_{i}^{0}$, for $1\leq i\leq d$. We can write 
\begin{eqnarray*}
\frac{L_{\delta }w(x)-L_{0}w(x)}{\delta } &=&\frac{1}{\delta }\left(
\sum_{i=1}^{d}\frac{w(y_{i}^{\delta })}{T_{\delta }^{\prime }(y_{i}^{\delta
})}-\sum_{i=1}^{d}\frac{w(y_{i}^{0})}{T_{0}^{\prime }(y_{i}^{0})}\right) \\
&=&\underbrace{\frac{1}{\delta }\left( \sum_{i=1}^{d}w(y_{i}^{\delta
})\left( \frac{1}{T_{\delta }^{\prime }(y_{i}^{\delta })}-\frac{1}{%
T_{0}^{\prime }(y_{i}^{\delta })}\right) \right) }_{=:(I)}+\underbrace{\frac{%
1}{\delta }\left( \sum_{i=1}^{d}\frac{w(y_{i}^{\delta })-w(y_{i}^{0})}{%
T_{0}^{\prime }(y_{i}^{\delta })}\right) }_{=:(II)} \\
&&+\underbrace{\frac{1}{\delta }\left( \sum_{i=1}^{d}w(y_{i}^{0})\left( 
\frac{1}{T_{0}^{\prime }(y_{i}^{\delta })}-\frac{1}{T_{0}^{\prime
}(y_{i}^{0})}\right) \right) }_{=:(III)}.
\end{eqnarray*}%
\label{eq1} To develop the the first term we differentiate the expansion $%
T_{\delta }(x)=T_{0}(x)+\delta \dot{T}(x)+o_{C^{3}}(\delta )$ in $x$ to get: 
\begin{equation*}
T_{\delta }^{\prime }(x)=T_{0}^{\prime }(x)+\delta \dot{T}^{\prime
}(x)+o_{C^{2}}(\delta ).
\end{equation*}%
We can then write 
\begin{eqnarray*}
(I) &=&\frac{1}{\delta }\left( \sum_{i=1}^{d}w(y_{i}^{\delta })\left( \frac{1%
}{T_{\delta }^{\prime }(y_{i}^{\delta })}-\frac{1}{T_{0}^{\prime
}(y_{i}^{\delta })}\right) \right) \\
&=&\frac{1}{\delta }\left( \sum_{i=1}^{d}\frac{w(y_{i}^{\delta })}{T_{\delta
}^{\prime }(y_{i}^{\delta })}\left( 1-\frac{T_{\delta }^{\prime
}(y_{i}^{\delta })}{T_{0}^{\prime }(y_{i}^{\delta })}\right) \right) \\
&=&\frac{1}{\delta }\left( \sum_{i=1}^{d}\frac{w(y_{i}^{\delta })}{T_{\delta
}^{\prime }(y_{i}^{\delta })}\left( 1-\left( \frac{T_{0}^{\prime
}(y_{i}^{\delta })+\delta \dot{T}^{\prime }(y_{i}^{\delta
})+o_{C^{2}}(\delta )}{T_{0}^{\prime }(y_{i}^{\delta })}\right) \right)
\right) \\
&=&\left( -\sum_{i=1}^{d}\frac{w(y_{i}^{\delta })\dot{T}^{\prime
}(y_{i}^{\delta })}{T_{\delta }^{\prime }(y_{i}^{\delta })T_{0}^{\prime
}(y_{i}^{\delta })}\right) +o_{C^{2}}(1).
\end{eqnarray*}%
%
%
%
%
%
%
%
%
%
%
%
%
%
%
%
%
%
%
%
%
%
%
%
%
%
%
%
%
%
%
%
%
%
%
%
%
%
%
%
%
%
%
%
%
%
%
%
%
%
%
%
%
%
%
%
%
%
%
%
%
%
%
%
%
%
%
%
%
%
%
%
%
%
%
%
%
%
Thus we have that 
\begin{eqnarray*}
\lim_{\delta \rightarrow 0}\frac{1}{\delta }\left(
\sum_{i=1}^{d}w(y_{i}^{\delta })\left( \frac{1}{T_{\delta }^{\prime
}(y_{i}^{\delta })}-\frac{1}{T_{0}^{\prime }(y_{i}^{\delta })}\right)
\right) &=&\lim_{\delta \rightarrow 0}\left( -\sum_{i=1}^{d}\frac{%
w(y_{i}^{\delta })\dot{T}^{\prime }(y_{i}^{\delta })}{T_{\delta }^{\prime
}(y_{i}^{\delta })T_{0}^{\prime }(y_{i}^{\delta })}\right) \\
&=&-L_{0}\left( \frac{w\dot{T}^{\prime }}{T_{0}^{\prime }}\right)
\end{eqnarray*}%
%
%
%
%
%
%
%
%
%
%
%
%
%
%
%
%
%
%
%
%
%
%
%
%
%
%
%
%
%
%
%
%
%
%
%
%
%
%
%
%
%
%
%
%
%
%
%
%
%
%
%
%
%
%
%
%
%
%
%
%
%
%
%
%
%
%
%
%
%
%
%
%
%
%
%
%
%
and by Lemma \ref{tec} \ the limit also converges in $C^{1}$.

For the second term of (\ref{eq1}) we remark that by Lagrange theorem, for
any small $h$ there is $\xi $ such that $|\xi |\leq |h|$ and%
\begin{equation*}
w(y_{i}^{0}+h)=w(y_{i}^{0})+hw^{\prime }(y_{i}^{0})+\frac{1}{2}%
h^{2}w^{\prime \prime }(\xi )
\end{equation*}%
considering Lemma \ref{tec} and setting $h=y_{i}^{\delta }-y_{i}^{0}=\delta
\left( -\frac{\dot{T}(y_{i}^{0})}{T_{0}^{\prime }(y_{i}^{0})}\right)
+o_{C^{1}}(\delta )$ we get

\begin{eqnarray}
w(y_{i}^{\delta }) &=&w(y_{i}^{0}+\delta \left( -\frac{\dot{T}(y_{i}^{0})}{%
T_{0}^{\prime }(y_{i}^{0})}\right) +o_{C^{1}}(\delta ))  \label{OX} \\
&=&w(y_{i}^{0})+w^{\prime }(y_{i}^{0})(\delta \left( -\frac{\dot{T}%
(y_{i}^{0})}{T_{0}^{\prime }(y_{i}^{0})}\right) +o_{C^{1}}(\delta )) \\
&+&\frac{1}{2}(\delta \left( -\frac{\dot{T}(y_{i}^{0})}{T_{0}^{\prime
}(y_{i}^{0})}\right) +o_{C^{1}}(\delta ))^{2}w^{\prime \prime }(\xi )
\end{eqnarray}%
Since $w^{\prime \prime }$ is uniformly bounded, then%
\begin{equation*}
w(y_{i}^{\delta })=w(y_{i}^{0})+w^{\prime }(y_{i}^{0})\delta \left( -\frac{%
\dot{T}(y_{i}^{0})}{T_{0}^{\prime }(y_{i}^{0})}\right) +o_{C^{1}}(\delta ).
\end{equation*}

Thus 
\begin{eqnarray*}
(II)=\frac{1}{\delta }\sum_{i=1}^{d}\frac{w(y_{i}^{\delta })-w(y_{i}^{0})}{%
T_{0}^{\prime }(y_{i}^{\delta })} &=&\sum_{i=1}^{d}\frac{w^{\prime
}(y_{i}^{0})}{T_{0}^{\prime }(y_{i}^{\delta })}\left( -\frac{\dot{T}%
(y_{i}^{0})}{T_{0}^{\prime }(y_{i}^{0})}\right) +o_{C^{1}}(1) \\
&=&-\sum_{i=1}^{d}\frac{\dot{T}(y_{i}^{0})w^{\prime }(y_{i}^{0})}{%
T_{0}^{\prime }(y_{i}^{0})T_{0}^{\prime }(y_{i}^{\delta })}+o_{C^{1}}(1)
\end{eqnarray*}%
%
%
%
%
%
%
%
%
%
%
%
%
%
%
%
%
%
%
%
%
%
%
%
%
%
%
%
%
%
%
%
%
%
%
%
%
%
%
%
%
%
%
%
%
%
%
%
%
%
%
%
%
%
%
%
%
%
%
%
%
%
%
%
%
%
%
%
%
%
%
%
%
%
%
%
%
%
and therefore, both pointwise and in the $C^{1}$ topology 
\begin{equation*}
\lim_{\delta \rightarrow 0}\frac{1}{\delta }\sum_{i=1}^{d}\frac{%
w(y_{i}^{\delta })-w(y_{i}^{0})}{T_{0}^{\prime }(y_{i}^{\delta })}%
=-L_{0}\left( \frac{\dot{T}w^{\prime }}{T_{0}^{\prime }}\right) (x).
\end{equation*}%
Finally, for the third term we can write 
\begin{eqnarray*}
T_{0}^{\prime }(y_{i}^{\delta }) &=&T_{0}^{\prime }(y_{i}^{0})+T_{0}^{\prime
\prime }(y_{i}^{0})\left( \frac{dy_{i}^{\delta }}{d\delta }|_{\delta
=0}\right) \delta +o_{C^{1}}(\delta ) \\
&=&T_{0}^{\prime }(y_{i}^{0})+T_{0}^{\prime \prime }(y_{i}^{0})\left( -\frac{%
\dot{T}(y_{i}^{0})}{T_{0}^{\prime }(y_{i}^{0})}\right) \delta
+o_{C^{1}}(\delta ),
\end{eqnarray*}%
Again using the Lemma \ref{tec}, we get%
\begin{eqnarray*}
(III) &=&\frac{1}{\delta }\left( \sum_{i=1}^{d}w(y_{i}^{0})\left( \frac{1}{%
T_{0}^{\prime }(y_{i}^{\delta })}-\frac{1}{T_{0}^{\prime }(y_{i}^{0})}%
\right) \right)  \\
&=&\frac{1}{\delta }\left( \sum_{i=1}^{d}w(y_{i}^{0})\left( \frac{%
T_{0}^{\prime }(y_{i}^{0})-T_{0}^{\prime }(y_{i}^{\delta })}{T_{0}^{\prime
}(y_{i}^{\delta })T_{0}^{\prime }(y_{i}^{0})}\right) \right)  \\
&=&\frac{1}{\delta }\left( \sum_{i=1}^{d}w(y_{i}^{0})\left( \frac{-\left(
T_{0}^{\prime }(y_{i}^{0})+T_{0}^{\prime \prime }(y_{i}^{0})\left( -\frac{%
\dot{T}(y_{i}^{0})}{T_{0}^{\prime }(y_{i}^{0})}\right) \delta \right)
+T_{0}^{\prime }(y_{i}^{0})}{T_{0}^{\prime }(y_{i}^{\delta })T_{0}^{\prime
}(y_{i}^{0})}\right) \right) +o_{C^{1}}(1) \\
&=&\left( \sum_{i=1}^{d}w(y_{i}^{0})\left( \frac{\dot{T}(y_{i}^{0})T_{0}^{%
\prime \prime }(y_{i}^{0})}{T_{0}^{\prime }(y_{i}^{0})^{2}T_{0}^{\prime
}(y_{i}^{\delta })}\right) \ \right) +o_{C^{1}}(1)
\end{eqnarray*}%
%
%
%
%
%
%
%
%
%
%
%
%
%
%
%
%
%
%
%
%
%
%
%
%
%
%
%
%
%
%
%
%
%
%
%
%
%
%
%
%
%
%
%
%
%
%
%
%
%
%
%
%
%
%
%
%
%
%
%
%
%
%
%
%
%
%
%
%
%
%
%
%
%
%
%
%
%
and thus, finally, 
\begin{equation*}
\lim_{\delta \rightarrow 0}\frac{1}{\delta }\left(
\sum_{i=1}^{d}w(y_{i}^{0})\left( \frac{1}{T_{0}^{\prime }(y_{i}^{\delta })}-%
\frac{1}{T_{0}^{\prime }(y_{i}^{0})}\right) \right) =L_{0}\left( \frac{\dot{T%
}T_{0}^{\prime \prime }}{T_{0}^{\prime 2}}w\right) (x)
\end{equation*}%
in $C^{1}$.
\end{proof}

We remark that the $C^{1}$ convergence implies the convergence in the $%
W^{1,1}$ topology. Since we have obtained the existence of the operator $%
\dot{L}:C^{2}\rightarrow W^{1,1}$, we know that $T_{0}$ has spectral gap and
then exponential contraction on the space of zero average $W^{1,1}$
densities, since the assumption of \ Theorem \ref{LR} \ have been verified
in the last paragraph before Section \ref{lrex} we can apply the theorem and
get

\begin{proposition}
\label{pert} Let us assume that $T_{\delta }$ is a $C^{1}$ family of $C^{3}$
expanding maps as in $(\ref{fam})$. Let $\ f_{\delta }\in C^{2}(X,\mathbb{R}%
) $ be the invatiant probability density of $T_{\delta }$. We have the
following%
\begin{equation*}
\lim_{\delta \rightarrow 0}||\frac{f_{\delta }-f_{0}}{\delta }-(1-L_{0})^{-1}%
\dot{L}f_{0}||_{L^{1}}=0.
\end{equation*}%
Where $\dot{L}$ is the operator defined in Proposition \ref{mainprop}. Hence
we have a linear response formula 
\begin{equation*}
\lim_{\delta \rightarrow 0}\frac{f_{\delta }-f_{0}}{\delta }=(1-L_{0})^{-1}%
\dot{L}f_{0}
\end{equation*}%
where the limit converge is in $L^{1}.$
\end{proposition}

\subsection{Rigorous numerical methods for the computation of invariant
measures\label{Ul}}

We briefly mention an application of the transfer operator methods exposed
in these notes. It is possible to use the quantitative stability results on
the invariant measures here explained to design efficient numerical methods
for the approximation of the invariant measure and other important features
of the statistical behavior of the system.

The approximation can also be rigorous, in the sense that an explicit bound
on the approximation error can be provided (for example it is possible to
approximate the absolutely continuous invariant measure of a system up to a
small explicitly given error in the $L^{1}$ distance). We remark that doing
in this way, the result of the computation has a mathematical meaning and
can be used for computer aided proofs.

This can be done by approximating the transfer operator $L_{0}$ of the
system by a suitable finite rank one $L_{\delta }$ which is essentially a
matrix, of which we can compute fixed points and other properties.

There are many ways to construct a suitable approximating operator $%
L_{\delta }$ depending on the system which is considered. The most used one
(for $L^{1}$ approximations) is the so called Ulam discretization\emph{.}

In the \emph{Ulam Discretization} method, the phase space $X$ is discretized
by a partition $I_{\delta }=\{I_{i}\}$ (where $\delta $ is a resolution
parameter, for example the maximum diameter of the elements of the
partition) and the system is approximated by a (finite state) Markov Chain.
In the deterministic case, supposing the dynamics is generated by a map $%
T:X\rightarrow X$, the transition transition probabilities of the
approximating Markov chain are given by 
\begin{equation}
P_{ij}={m(T}^{-1}{(I_{j})\cap I_{i})}/{m(I_{i})}  \label{pij}
\end{equation}%
(where $m$ is the normalized Lebesgue measure on the phase space). The
approximated operator $L_{\delta }$ can be also seen in the following way:
let $F_{\delta }$ be the $\sigma -$algebra associated to the partition $%
I_{\delta }$, let us consider the projection on the step functions supported
on the elements of the partition given by%
\begin{equation*}
\pi _{\delta }(f)=\mathbf{E}(f|F_{\delta })
\end{equation*}%
($\mathbf{E}$ is the conditional expectation), we define the approximated
operator as:%
\begin{equation}
L_{\delta }=\pi _{\delta }L\pi _{\delta }.  \label{000}
\end{equation}

In a series of works it was proved that in several cases the fixed
probability measure $f_{\delta }$ of $L_{\delta }$ converges to the fixed
point of $L$ as the accuracy of the approximation gets better and better.
Explicit bounds on the error have been given, rigorous methods implemented
and experimented in several classes of cases (see e.g. \cite{BB},\cite{BM},%
\cite{DelJu02},\cite{F08},\cite{L} and \cite{GN} where several computations
on nontrivial systems are also shown). Ulam methods and similar methods have
been also used to rigorously compute (up to prescribed errors) other
important quantities related to the statistical properties of dynamics, as
Linear Response (see \cite{BGN}), dimension of attractors (see e.g. \cite%
{GN2}) or diffusion coefficients (see e.g. \cite{BGN2}). Some more details
on the Ulam method will be given in Section \ref{dopo} \ in the case of
Piecewise Expanding maps of the interval.

\section{Piecewise expanding maps\label{PW}}


\begin{figure}[!ht]
\centering
\includegraphics[width=40mm]{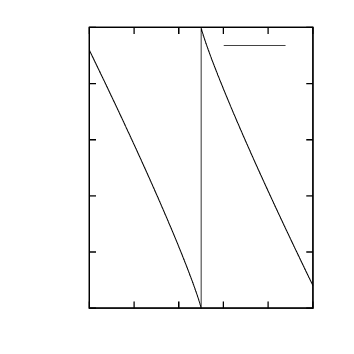}
\hspace{3mm}
\includegraphics[width=40mm]{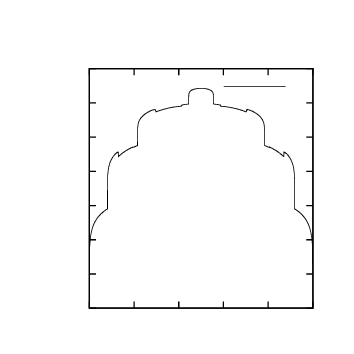}
\caption{A piecewise expanding map and a plot of its a.c.i.m. }
\end{figure}

We now consider a class of maps on the interval which are expanding, but
allow discontinuities. This class is interesting and was much studied
because it presents a quite rich behavior, while being approachable with
techniques similar to the ones introduced in the previous sections. In the
following we outline the main properties and tools which allow the study of
\ the statistical properties of these maps when seen as deterministic
dynamical systems. Since the theory of the statistical behavior of piecewise
expanding maps was exposed in several books (see \cite{gora} or \cite{Viana}%
) we will not enter in the technical details, but we will focus on
similarities and differences between the behavior of these systems and the
behavior of expanding maps treated in the previous sections.

\begin{definition}
We call a nonsingular function $T:([0,1],m)\rightarrow ([0,1],m)$ piecewise
expanding if

\begin{itemize}
\item There is a finite set of points $d_{1}=0,d_{2},...,d_{n}=1$ such that
for each $i$, $T_{i}:=T|_{(d_{i},d_{i+1})}$ is $C^{2}$ and $\sup_{[0,1]}%
\frac{|T^{\prime \prime }|}{(T^{\prime 2})}dx<\infty $.

\item $\inf_{x\in \lbrack 0,1]}|T^{\prime }(x)|>1$ on the set where it is
defined.
\end{itemize}
\end{definition}

The transfer operator associated to a map of this class has general
properties similar to the ones of the expanding maps, if we apply the
transfer operator to measures having a density we obtain the following
explicit formula (see e.g. \cite{gora} chapter 4 for details) describing the
action of the associated transfer operator (which we continue denoting with $%
L$) on measure densities

\begin{equation}
\lbrack Lf](x)=\sum_{i\leq n}\frac{f(T_{i}^{-1}x)1_{T_{i}(d_{i},d_{i+1})}}{%
|T^{\prime }(T_{i}^{-1}x)|}
\end{equation}%
where $1_{T_{i}(d_{i},d_{i+1})}$ is the characteristic function of the
interval $T_{i}(d_{i},d_{i+1}).$

In the presence of discontinuities of the map $T$ the transfer operator does
not necessarily preserve spaces of continuous densities. For this the
introduction of a suitable space of regular densities including
discontinuous functions is important.

\subsection{Bounded variation and the Lasota Yorke inequality for piecewise
expanding maps}

Let $\phi :[0,1]\rightarrow \mathbb{R}$ a real function. Let $%
\{x_{1},...,x_{k}\}\subseteq \lbrack 0,1]$ be a finite sequence of points.
Let us define the variation of $\phi $ with respect to $\{x_{1},...,x_{k}\}$
as

\begin{equation*}
Var_{\{x_{1},...,x_{k}\}}(\phi )=\sum_{i=1}^{k-1}|\phi (x_{i})-\phi
(x_{i+1})|
\end{equation*}%
we define the variation of $\phi $ as the supremum of $Var_{%
\{x_{1},...,x_{k}\}}(\phi )$ over all the finite sequences $%
\{x_{1},...,x_{k}\}$%
\begin{equation*}
Var(\phi )=\sup_{\{x_{1},...,x_{k}\}\subseteq \lbrack
0,1]}Var_{\{x_{1},...,x_{k}\}}(\phi ).
\end{equation*}

We say that $\phi $ has bounded variation if $Var(\phi )<\infty $. We call $%
BV$ or $BV[0,1]$ the set of bounded variation functions on the interval. An
important property of bounded variation functions is the following

\begin{theorem}[Helly selection principle]
\label{Helly}Let $\phi _{n}$ be a sequence of bounded variation functions on
the interval $[0,1]$ such that $Var(\phi _{n})\leq M$ and $||\phi
_{n}||_{1}\leq M$ are uniformly bounded.

Then there is $\phi \in BV$ and subsequence $\phi _{n_{k}}$ such that 
\begin{equation*}
\phi _{n_{k}}\rightarrow \phi
\end{equation*}%
in $L^{1}$ (and almost everywhere).
\end{theorem}

Bounded variation functions are preserved by the transfer operator of a
piecewise expanding map, moreover the following Lasota Yorke inequality can
be proved.

\begin{theorem}
\label{LYBV}Let $T$ be a piecewise expanding map with branches $T_{i}$ on
the intervals $I_{i}$. Let $\phi $ a bounded variation density on an
interval $I_{i}=(d_{i},d_{i+1}]$ and $T:I_{i}\rightarrow \lbrack 0,1]$ a
piecewise expanding function. Then 
\begin{equation}
Var(L_{T}\phi )\leq \frac{2}{\inf_{[0,1]}(T^{\prime })}Var(\phi
)+(\sup_{[0,1]}(|\frac{T^{\prime \prime }}{T^{\prime 2}}|)+\frac{2}{%
\inf_{i}|I_{i}|})\int |\phi |.  \label{ok}
\end{equation}
\end{theorem}

Before the proof we consider the behavior of the operator when acting on the
density on a single interval $I_{i}$.

\begin{lemma}
\label{lemly}Let $\phi $ a bounded variation density on an interval $I_{i}$
and $T:I_{i}\rightarrow \lbrack 0,1]$ an invertible expanding function. Then 
\begin{equation*}
Var(L_{T}(\phi ))\leq \frac{2}{\inf_{I_{i}}(T^{\prime })}Var(\phi
)+(\sup_{I_{i}}(|\frac{T^{\prime \prime }}{T^{\prime 2}}|)+\frac{2}{|I_{i}|}%
)\int |\phi |.
\end{equation*}
\end{lemma}

\begin{proof}
Let us consider $y_{1},...,y_{k}\in \lbrack 0,1]$ and let us suppose there
are $h_{1},h\in \mathbb{N}$ and $x_{1},...,x_{h}\in I_{i}$ such that $%
T^{-1}y_{i}=\left\{ 
\begin{array}{c}
\emptyset ~if~i<h_{1} \\ 
x_{i-h_{1}+1}~if~h_{1}\leq i\leq h_{1}+h \\ 
\emptyset ~if~i>h_{1}+h%
\end{array}%
\right. .$%
\begin{eqnarray*}
Var_{y_{1,...,}y_{k}}(L_{T}\phi ) &=&\sum_{1}^{k-1}|L_{T}\phi
(y_{i})-L_{T}\phi (y_{i+1})| \\
&\leq &|\frac{1}{T^{\prime }(x_{h_{1}})}\phi (x_{h_{1}})|+|\frac{1}{%
T^{\prime }(x_{h_{1}}+h)}\phi (x_{h_{1}+h})|+\sum_{h_{1}}^{h_{1}+h-1}|\frac{1%
}{T^{\prime }(x_{i})}\phi (x_{i})-\frac{1}{T^{\prime }(x_{i+1})}\phi
(x_{i+1})|
\end{eqnarray*}

Since there must be $\hat{x}\in I_{i}$ such that $\phi (\hat{x})\leq \frac{1%
}{|I_{i}|}\int |\phi |$, then%
\begin{eqnarray*}
&&|\frac{1}{T^{\prime }(x_{h_{1}})}\phi (x_{h_{1}})|+|\frac{1}{T^{\prime
}(x_{h_{1}}+h)}\phi (x_{h_{1}+h})| \\
&\leq &\frac{1}{\inf_{I_{i}}(T^{\prime })}(2\phi (\hat{x})+|\phi
(x_{h_{1}})-\phi (\hat{x})|+|\phi (x_{h_{1}+h})-\phi (\hat{x})|) \\
&\leq &\frac{1}{\inf_{I_{i}}(T^{\prime })}Var(\phi )+\frac{2}{|I_{i}|}\int
|\phi |.
\end{eqnarray*}

The other summand can be bounded by%
\begin{eqnarray*}
\sum_{h_{1}}^{h_{1}+h-1}|\frac{1}{T^{\prime }(x_{i})}\phi (x_{i})-\frac{1}{%
T^{\prime }(x_{i+1})}\phi (x_{i+1})| &\leq &\sum_{h_{1}}^{h_{1}+h-1}|\frac{1%
}{T^{\prime }(x_{i})}\phi (x_{i})-\frac{1}{T^{\prime }(x_{i})}\phi (x_{i+1})|
\\
&&+|\frac{1}{T^{\prime }(x_{i})}\phi (x_{i+1})-\frac{1}{T^{\prime }(x_{i+1})}%
\phi (x_{i+1})| \\
&\leq &\frac{1}{\inf_{I_{i}}(T^{\prime })}Var(\phi
)+\sum_{h_{1}}^{h_{1}+h-1}|\frac{1}{T^{\prime }(x_{i})}\phi (x_{i+1})-\frac{1%
}{T^{\prime }(x_{i+1})}\phi (x_{i+1})| \\
&\leq &\frac{1}{\inf_{I_{i}}(T^{\prime })}Var(\phi
)+\sum_{h_{1}}^{h_{1}+h-1}|\frac{T^{\prime \prime }(\xi _{i})}{T^{\prime
}(\xi _{i})}|x_{i}-x_{i+1}|\phi (x_{i+1})
\end{eqnarray*}%
by Lagrange theorem, for $\xi _{i}\in \lbrack x_{i},x_{i+1}]$. And%
\begin{eqnarray*}
&&\frac{1}{\inf_{I_{i}}(T^{\prime })}Var(\phi )+\sum_{h_{1}}^{h_{1}+h-1}|%
\frac{T^{\prime \prime }(\xi _{i})}{T^{\prime }(\xi _{i})}%
|x_{i}-x_{i+1}|\phi (x_{i+1}) \\
&\leq &\frac{1}{\inf_{I_{i}}(T^{\prime })}Var(\phi )+\sup_{I_{i}}(|\frac{%
T^{\prime \prime }}{T^{\prime 2}}|)\sum_{h_{1}}^{h_{1}+h-1}|x_{i}-x_{i+1}||%
\phi (x_{i+1})|.
\end{eqnarray*}%
Remarking that $\forall \epsilon $, if the subdivision $\{x_{i}\}$ is fine
enough then $\sum_{h_{1}}^{h_{1}+h-1}|x_{i}-x_{i+1}||\phi (x_{i+1})|\leq
\int_{I_{i}}\phi +\epsilon $ \ and collecting all summands\ we have the
statement.
\end{proof}

This allow to conclude

\begin{proof}
(of Thm\ref{LYBV}) Let $\phi _{i}=\phi |_{I_{i}}$. We have that $L_{T}\phi
=\sum_{i}L_{T}\phi _{i}$. Then%
\begin{equation*}
Var(L_{T}\phi )\leq \sum_{i}VarL\phi _{i}
\end{equation*}%
by Lemma \ref{lemly} 
\begin{eqnarray*}
\sum_{i}VarL\phi _{i} &\leq &\sum_{i}\frac{2}{\inf_{I_{i}}(T^{\prime })}%
Var(\phi _{i})+(\sup_{I_{i}}(|\frac{T^{\prime \prime }}{T^{\prime 2}}|)+%
\frac{2}{|I_{i}|})\int |\phi _{i}| \\
&\leq &\frac{2}{\inf_{[0,1]}(T^{\prime })}Var(\phi )+(\sup_{[0,1]}(|\frac{%
T^{\prime \prime }}{T^{\prime 2}}|)+\frac{2}{\inf_{i}|I_{i}|})\int |\phi |
\end{eqnarray*}
\end{proof}

One can define the Bounded Variation norm $||~||_{BV}$ as 
\begin{equation*}
||f||_{BV}=Var(f)+||f||_{1}.
\end{equation*}%
It is immediate to deduce from $(\ref{ok})$ the Lasota Yorke inequality for
the bounded variation norm: there is $B$ such that%
\begin{equation*}
||L_{T}\phi ||_{BV}\leq \frac{2}{\inf_{I_{i}}(T^{\prime })}||\phi
||_{BV}+B||\phi ||_{1}.
\end{equation*}

Like done before for expanding maps, writing $\lambda =\frac{2}{%
\inf_{I_{i}}(T^{\prime })}$ and iterating the inequality we obtain%
\begin{equation*}
||L_{T}^{n}\phi ||_{BV}\leq \lambda ^{n}||\phi ||_{BV}+\frac{B}{1-\lambda }%
||\phi ||_{1}.
\end{equation*}

This inequality holds for maps such that $\frac{2}{\inf_{I_{i}}((T^{n})^{%
\prime })}<1$. The inequality can be applied to the other piecewise
expanding maps by previously iterating the map until $\frac{2}{%
\inf_{I_{i}}((T^{n})^{\prime })}<1$.

In this case, as before, a straightforward computation lead to the general
Lasota Yorke inequality valid for any piecewise expanding map: there are $%
A,B\geq 0$ \ and $\lambda \in \lbrack 0,1)$ such that%
\begin{equation}
||L_{T}^{n}\phi ||_{BV}\leq A\lambda ^{n}||\phi ||_{BV}+B||\phi ||_{1}.
\label{lygen}
\end{equation}

The first consequence of the inequality is the existence of a bounded
variation invariant density for each piecewise expanding map, indeed, by (%
\ref{lygen}) and Theorem \ref{Helly}, repeating the arguments stated in
Section \ref{ext} we get the existence of an absolutely continuous invariant
measure with bounded variation density for this kind of maps. By arguments
very similar to the ones presented in Section \ref{spg} it is also possible
to obtain that a piecewise expanding map has spectral gap on the space of
bounded variation densities\footnote{%
Using the $BV$ and $L^{1}$ as strong and weak spaces and its Lasota Yorke
inequality, Theorem \ref{gap} and Theorem \ref{Helly} to obtain the compact
inclusion property.} if the map is mixing in some sense, then there will be
an exponential convergence to equilibrium of the system.

About the mixing assumption, it is well known that a condition which is
sufficient to imply the exponential convergence to equilibrium (and hence
all the other statistical consequences) for Piecewise expanding maps is the
topological mixing (see \cite{Viana2}, assumption E3 and following for the
details)

\begin{definition}
We say that a piecewise expanding map $T$ is \emph{topologically mixing }if
\ there is an interval $I_{\ast }\subseteq I$ such that $f(I_{\ast
})=I_{\ast },$ every orbit $T^{n}(x),x\in \lbrack 0,1]$ eventually enters $%
I_{\ast }$ \ and for every $J\subset I_{\ast }$ there is $n\geq 1$ such that 
$T^{n}(J)=I_{\ast }.$
\end{definition}

Piecewise expanding maps however, have a more complicated behavior than
expanding ones with respect to statistical stability on perturbations. We
point out that the Lasota Yorke inequality we have proved in Lemma \ref%
{lemly}, works only if the expansion rate of the map is bigger than $2$. To
prove the existence of an absolutely continuous invariant measure as
sketched above, one has to take an iterate of the map such that $\frac{2}{%
\inf_{I_{i}}((T^{n})^{\prime })}<1$.

When considering the statistical stability of a family of maps $T_{\delta }$
this is not only a technical point but is substantial, because sometime it
is not possible to find a uniform iterate which is suitable for the whole
family $T_{\delta }$. While many of the stability arguments outlined in the
previous sections applies also to piecewise expanding maps with expansion
rate greater than $2$, for the maps $T$ such that $1\leq
\inf_{I_{i}}(T^{\prime })\leq 2$ the stability questions are more
complicated. We present below some examples of results illustrating the
questions.

\subsection{The stability under deterministic perturbations\label{disc}}

We have seen that mixing piecewise expanding maps associated transfer
operators have spectral gap on bounded variation functions.

Now let us consider their statistical stability to perturbations. If we
consider $BV$ and $L^{1}$ as a weak and strong spaces, and we consider
perturbations for which UF1,...,UF4 \ are satisfied, then Proposition \ref%
{stabi}, gives us a quantitative statistical stability estimation%
\begin{equation}
||h_{\delta }-h_{0}||_{1}=O(\delta \log \delta )  \label{hso}
\end{equation}%
where $h_{\delta },h_{0}$ are the absolutely continuous invariant measures
of the perturbed map $T_{\delta }$ and of the unperturbed one $T_{0}$.
UF1,...,UF4 are satisfied by a wide variety of stochastic and deterministic
perturbations.

In particular let us consider deterministic perturbations so that $L_{\delta
}$ is the transfer operator of a family of maps $T_{\delta }$ satisfying a
uniform Lasota Yorke inequality (UF1) \ with $BV$ \ and $L^{1}$ \ as strong
and weak space (we remark that given a family or a given perturbation of
some map, the existence of a uniform Lasota Yorke inequality is something
that can me checked easily). \ In the case $T_{0}$ is topologically mixing
the assumptions UF3 \ UF4 \ are easily satisfied. We recall that in this
case the assimption UF2 would be $||(L_{\delta }-L_{0})g||_{1}\leq \delta
C||g||_{BV}$. If also UF2 is satisfied for some kind of perturbation then we
have our quantitative statistical stability estimation $(\ref{hso})$
established for these systems and perturbations.

It is not much complicated to characterize families of deterministic
perturbations for which UF2 \ hold. \ The Skorokhod metric defines the
distance between two maps $T_{1}$ and $T_{2}$ as

\begin{gather}
d_{s}(T_{1},T_{2})=\inf \{\epsilon >0:\exists A\subseteq I~and~\exists
\sigma :[0,1]\rightarrow \lbrack 0,1]~  \label{sko} \\
s.t.m(A)\geq 1-\epsilon ,~\sigma ~is~a~diffeomorphism,T_{1}|_{A}=T_{2}\circ
\sigma |_{A}~and  \notag \\
\forall x\in A,|\sigma (x)-x|\leq \epsilon ,|\frac{1}{\sigma ^{\prime }(x)}%
-1|\leq \epsilon \}  \notag
\end{gather}
Families $T_{\delta }$ where $d_{s}(T_{0},T_{\delta })\leq K\delta $ will
satisfy UF2 (see e.g. \cite{gora}, chapter 11.2) and then one can prove
quantitative statistical stability results as in $(\ref{hso}).$

In this context also uniform contraction can be proved (Proposition \ref%
{unifcont}), but Lipschitz continuity is expected to hold only in particular
cases because the difference $||(L_{\delta }-L_{0})h_{0}||_{BV}$ (see
assumptions in Proposition \ref{lip} ) of the initial and perturbed
operators applied to the initial invariant measure is not small in the 
\textbf{strong }norm (the $BV$ norm in this case) for many typical
deterministic perturbations one would like to consider (consider
perturbations moving discontinuities or values at discontinuities for
example). For examples of non Lipschitz behavior of the statistical
stability of families of piecewise expanding maps satisfying a uniform
Lasota Yorke inequality see \cite{Bsusc} or \cite{M}.

Even more complicated behavior can be found if we consider the case when the
family of maps has not a uniform Lasota Yorke inequality. The simplest case
is when the slope of the family $T_{\delta }$ is not uniformly above $2$
(see Theorem \ref{LYBV}). In this case we can have a discontinuous behavior
of the family of associated invariant measures, as shown by \cite{EM} and 
\cite{K} (see also \cite{EG} for further examples).

Consider the $3$ parameters family of maps $W_{a,b,r}$ \ defined by

\begin{equation*}
W_{a,b,r}(x)=\left\{ 
\begin{array}{c}
a(1-x/r)~~~for~0\leq x\leq r \\ 
(2b/(1-2r))(x-r)~~~for~r\leq x\leq 1/2 \\ 
W_{a,b,r}(1-x)~~~for~1/2\leq x\leq 1.%
\end{array}%
\right.
\end{equation*}%

\begin{figure}
\includegraphics[width=40mm]{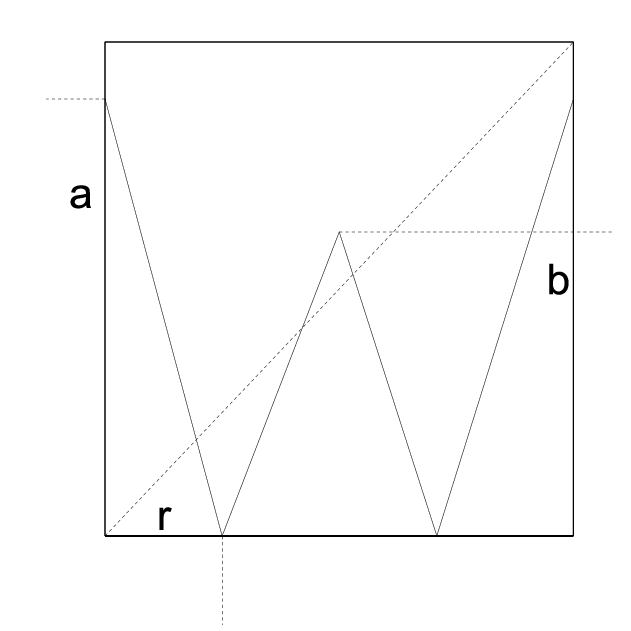}
\caption{The map $W$ is continuous, piecewise linear and expanding with slopes $a/r$ and $2b/(1-2r)$}
\end{figure}

The maps are piecewise expanding, let $h_{a,d,r}$ denote the unique
invariant density for $W_{a,b,r}$. \ Now let us consider a sequence \ $%
(a_{n},d_{n},r_{n})\rightarrow (\frac{1}{2},\frac{1}{2},\frac{1}{4})$ and
the related densities $h_{a_{n},d_{n},r_{n}}$. In \cite{K} (page 331) is
shown that $h_{1,\frac{1}{2},r}=\frac{3}{2}1_{[0,\frac{1}{2}]}+\frac{1}{2}%
1_{(\frac{1}{2},1]}$ and $h_{\frac{1}{2},\frac{1}{2},r}=21_{[0,\frac{1}{2}]}$
while if $\frac{1}{2}<b_{n}\leq 1-2r_{n}$ then $h_{a_{n},d_{n},r_{n}}%
\rightarrow \delta _{1/2}$ weakly. This is due to the fact that for $\frac{1%
}{2}<b\leq 1-2r$ \ the interval $[1-b,b]$ is sent to itself by the map, and
"attracts" all the measure while iterating the map. Hence for $a=1/2$ and $%
a=1$ the limit measure does not coincide with the invariant absolutely
continuous measure of the limit map shown above (by the way the limit map
has $\delta _{1/2}$ as a non absolutely continuous invariant measure since $%
1/2$ \ is a fixed point). We remark that for this family of maps we cannot
have a uniform Lasota Yorke inequality \ as in UF1, as the slopes tend to $2$%
. On the other hand if one takes iterates of the maps to increase the slope,
the smaller and smaller invariant interval around $1/2$ would let the second
coefficient of the Lasota Yorke inequality to converge to $\infty $.

\subsection{The approximation of the invariant measure for piecewise
expanding maps\label{dopo}}

In this section we apply Proposition \ref{stabi} to show how our stability
results can give a quantitative estimate on the error \ in the numerical
approximation of invariant densities of piecewise expanding maps with the
Ulam method We will use $BV$ and $L^{1}$ as a strong and weak space.

Let us consider a piecewise expanding map $T$ which is topologically mixing
and for which $\frac{2}{\inf_{I_{i}}((T)^{\prime })}<1$. Let us consider its
transfer operator $L$. Suppose we want to approximate the absolutely
continuous invariant probability measure $h$ by the Ulam method outlined in
Section \ref{Ul}. $L$ is then approximated by $L_{\delta }$ defined in $(\ref%
{000})$. We will approximate $h$ by the invariant density $h_{\delta }$ of $%
L_{\delta }$. We consider $L_{\delta }$ as a small perturbation of $L$ and
we will use Proposition \ref{stabi} to estimate the distance between $h$ and 
$h_{\delta }$. Suppose $L$ satisfies a Lasota Yorke inequality. Then the
approximated operator $L_{\delta }$ satisfies the same LasotaYorke inequality

\begin{lemma}
If $L$ is such that $||L\phi ||_{BV}\leq \lambda ||\phi ||_{BV}+B||\phi
||_{1}$ with $\lambda <1$, then 
\begin{equation*}
||L_{\delta }\phi ||_{BV}\leq \lambda ||\phi ||_{BV}+B||\phi ||_{1}.
\end{equation*}
\end{lemma}

\begin{proof}
We remark that $||\mathbf{E}(f|\mathcal{F}_{\delta })||_{BV}\leq ||f||_{BV}$
(see \cite{L} Lemma 4.1).

Then 
\begin{eqnarray*}
||L_{\delta }\phi ||_{BV} &=&||\mathbf{\pi }_{\delta }L\mathbf{\pi }_{\delta
}(\phi )||_{BV}\leq ||L\mathbf{\pi }_{\delta }(\phi )||_{BV} \\
&\leq &\lambda ||\mathbf{\pi }_{\delta }(\phi )||_{BV}+B||\mathbf{\pi }%
_{\delta }\phi ||_{1}\leq \lambda ||\phi ||_{BV}+B||\phi ||_{1}.
\end{eqnarray*}
\end{proof}

For each $\delta \geq 0$ we consider on the interval $[0,1]$ a partition $%
\mathcal{F}$ made of equal intervals $I_{i}$ having length $\frac{1}{%
\left\lfloor \frac{1}{\delta }\right\rfloor }$ (we consider intervals of
size $\delta $ when $\frac{1}{\delta }$ is integer). To apply Proposition %
\ref{stabi} we need an estimation on the quality of approximation by Ulam
discretization like asked in the assumption UF2. This is provided by the
following

\begin{lemma}
If $L$ is the transfer operator associated to a piecewise expanding map and $%
L_{\delta }$ is given by the Ulam discretization as explained before we have
that there is $C>0$ such that 
\begin{equation*}
||Lf-L_{\delta }f||_{1}\leq C\delta ||f||_{BV}
\end{equation*}
\end{lemma}

\begin{proof}
It is not difficult to see that for $f\in BV$, it holds 
\begin{equation}
||\pi _{\delta }f-f||_{L^{1}}\leq \delta \cdot ||f||_{BV}.  \label{bbn}
\end{equation}

Indeed from the definition of the norm we can see that $||f||_{BV}\geq
\sum_{i}|\sup_{I_{i}}(f)-\inf_{I_{i}}(f)|$, where $I_{i}$ are the intervals
composing the partition $\mathcal{F}$. Since $\sup_{I_{i}}(f)\geq \mathbf{E}%
(f|I_{i})\geq \inf_{I_{i}}(f)$, it follows $\int_{I_{i}}|\mathbf{E}(f|%
\mathcal{F}_{\delta })-f|\leq \delta |\sup_{I_{i}}(f)-\inf_{I_{i}}(f)|$
leading to \ref{bbn}.

By this it holds 
\begin{equation*}
||(L-L_{\delta })f||_{L^{1}}\leq ||\mathbf{\pi }_{\delta }L\mathbf{\pi }%
_{\delta }(f)-\pi _{\delta }L(f)||_{L^{1}}+||\pi _{\delta }(Lf)-Lf||_{L^{1}},
\end{equation*}%
and%
\begin{equation*}
||\mathbf{\pi }_{\delta }L\mathbf{\pi }_{\delta }(f)-\pi _{\delta
}L(f)||_{L^{1}}\leq ||\pi _{\delta }(Lf)-Lf||_{L^{1}}\leq \delta
||Lf||_{BV}\leq \delta (\lambda ||f||_{BV}+B||f||_{1})\leq \delta
(B+1)||f||_{BV}.
\end{equation*}
\end{proof}

Since $||\mathbf{\pi }_{\delta }||_{L^{1}\rightarrow L^{1}}\leq 1$ we
immediately get that if $L_{\delta }$ is given by the Ulam method, for each $%
f\in L^{1}$ 
\begin{equation*}
||L_{\delta }f||_{L^{1}}\leq ||f||_{L^{1}}.
\end{equation*}

Hence we proved that UF1,...,UF4 applies to the family of perturbed
operators $L_{\delta }$ and thus we can apply Proposition \ref{stabi},
concluding that for Piecewise expanding maps which are topologically mixing
and such that $\frac{2}{\inf_{I_{i}}((T)^{\prime })}<1$ the Ulam method
approximation $h_{\delta }$ converge to the real invariant density $h.$
Furthermore, we have a quantitative estimation%
\begin{equation*}
||h_{\delta }-h_{0}||_{1}=O(\delta \log \delta ).
\end{equation*}%
In \cite{BM} it is proved that this rate is the optimal one. There are
examples of map for which the approximation rate is asymptotically
proportional to $\delta \log \delta .$

\section{A look at random dynamical systems\label{rnd}}

In this section we apply the transfer operator approach to random dynamical
systems. A random dynamical system is a dynamical system where the dynamics
depend on some random parameter. Instead of iterating a single map $T$ in
this case we have a family of maps $\{T_{\omega }\}_{\omega \in \Omega }$
and the dynamics applies a sequence of maps drawn at random in this family.
We will see that in this case too, we can define an associated transfer
operator, consider the associated invariant measures, and relevant
statistical properties of the system can be studied by these concepts. In
the first part of the section we will introduce what is a random dynamical
system and what is the associated transfer operator. \ We restrict our study
to a class of dynamical systems in which the maps $\{T_{\omega }\}_{\omega
\in \Omega }$ are drawn independently. The concept of random dynamical
system can be considered from a more general point of view (see \cite{Ar}).
In the following subsection we briefly introduce the basic of the ergodic
theory of random dynamical systems. The (simple) approach we choose follows
Section 5 of \cite{Viana}. Then we will focus our attention on the class of
systems with additive noise. These are system in which at each iteration of
the dynamics one applies a deterministic map and the adds a random
perturbation (the noise). This class of systems is flexible enough to
include models of nontrivial real phenomena (see e.g. \cite{GMN}, \cite{LM})
and for this kinds of systems the transfer operator approach allow to obtain
several interesting results in a relatively simple way.

\subsection{Random dyamics, basic definitions.\label{oneside}}

First let us describe the probability space defining the randomness in our
systems Let $(\Omega ,\mathcal{X},p)$ be a probability space. Let $(\Omega ^{%
\mathbb{N}},\mathcal{A},\mu )$ the space of sequences on $\Omega $ endowed
with the product $\sigma $-algebra $\mathcal{A=X}^{\mathbb{N}}$ and the
product measure $\mu =p^{\mathbb{N}}.$ Let $s:\Omega ^{\mathbb{N}%
}\rightarrow \Omega ^{\mathbb{N}}$ be the usual shift map on $\Omega ^{%
\mathbb{N}}$ defined by $s(\omega )=(\omega _{i+1})_{i\in \mathbb{N}}$ where 
$\omega =(\omega _{i})_{i\in \mathbb{N}}$ (for example $s[(\omega
_{0},\omega _{1},\omega _{2},...)]=(\omega _{1},\omega _{2},\omega _{3},...)$
). The measure $\mu $ is invariant for $s$, and $(\Omega ^{\mathbb{N}},s,\mu
)$ is a deterministic ergodic dynamical system which is a model of
independent sorting from $\Omega $ with probability $p$. The $\sigma $%
-algebra $\mathcal{A}$ is indeed generated by the \emph{cylinders }$%
C_{i,A}=\{\omega \in \Omega ^{\mathbb{N}}|\omega _{i}\in A\}$ and $\forall
i\in \mathbb{N},A\subseteq \Omega $ we have%
\begin{equation*}
\mu (C_{i,A})=\mu (A)=\mu (s^{-1}(C_{i,A}))=\mu (C_{i+1,A}).
\end{equation*}

Now we define a random dynamical system formally as a skew product. Let $(N,%
\mathcal{B)}$ be a measurable space. Let us consider $\Omega ^{\mathbb{N}%
}\times N$ with the product $\sigma $-algebra $\mathcal{A\times B}.$

Let $\mathcal{F=\{}F_{i}:N\rightarrow N\mathcal{\}}_{i\in \Omega }$ be a set
of maps $N\rightarrow N$. A random transformation $F_{\omega }:N\rightarrow
N $ over $s$ is a measurable%
\begin{equation*}
F:\Omega ^{\mathbb{N}}\times N\rightarrow \Omega ^{\mathbb{N}}\times N
\end{equation*}%
defined by%
\begin{equation*}
F(\omega ,v)=(s(\omega ),F_{\omega _{0}}(v))
\end{equation*}

where $F_{\omega _{0}}\in \mathcal{F}$ depends only on the $0$-th coordinate 
$\omega _{0}\in \Omega $ of \thinspace $\omega \in \Omega ^{\mathbb{N}}$.

\begin{example}
Let $\Omega =\{0,1\}$ suppose $p(0)=p(1)=\frac{1}{2}$ (and the usual product
structure) Let $F_{0},F_{1}$ be maps $F_{i}:N\rightarrow N$ and 
\begin{equation*}
F(\omega ,v)=(s(\omega ),F_{\omega _{0}}(v)).
\end{equation*}%
Then $F$ represents the random dynamics in which $F_{0},F_{1}$ are applied
independently and with the same probability.
\end{example}

\begin{example}
Let $\Omega =[-\frac{1}{2},\frac{1}{2}]$ suppose $p=Leb~meas.$ (and the
usual product structure) Let $g:\mathbb{R\rightarrow R}$ be a measurable \
function and%
\begin{equation*}
F_{\omega }(v)=g(v)+\omega _{0}.
\end{equation*}%
Then $F$ represents the random dynamics in which at each step $g$ is applied
and some random unif. distributed noise in $[-\frac{1}{2},\frac{1}{2}]$ is
added.
\end{example}

\subsection{The transfer operator}

Given $F_{\omega }:N\rightarrow N$ the associated (annealed) transfer
operator $L:SM(N)\rightarrow SM(N)$ is defined by considering in some sense
the average transfer operator among all the transfer operators associated to
the maps $\{T_{\omega }\}_{\omega \in \Omega }.$ The average will be
considered according to the measure $p$ describing the randomness in our
system%
\begin{equation*}
L(\nu )=\int_{x\in \Omega }L_{F_{x}}(\nu )~dp(x)
\end{equation*}%
i.e. \ given a measurable $A\subseteq N$%
\begin{eqnarray*}
\lbrack L(\nu )](A) &=&\int_{x\in \Omega }[L_{F_{x}}(\nu )](A)~dp(x) \\
&=&\int_{x\in \Omega }\nu \lbrack F_{x}^{-1}(A)]~dp(x).
\end{eqnarray*}

As in the deterministic case: this is a positive operator, and preserves
probability measures, hence it is a Markov operator. Furthermore by $(\ref%
{weakcontr})$ if almost all the maps $\{T_{\omega }\}_{\omega \in \Omega }$
are nonsingular and $\nu \in L^{1}$ then it easily follows%
\begin{equation*}
||L(\nu )||_{1}\leq ||\nu ||_{1}.
\end{equation*}

Associated to $F$ there is another linear map acting of functions, which is
in some sense the dual of the transfer operator. This is the analog of the
composition operator and it is called as the Koopman operator associated to
our system. Let $\phi :N\rightarrow \mathbb{R}$ be measurable, let $x\in N,$
define $P(\phi ):N\rightarrow \mathbb{R}$ as%
\begin{equation*}
\lbrack P(\phi )](x)=\int_{y\in \Omega }\phi (F_{y}(x))~dp(y).
\end{equation*}

As for deterministic dynamical systems we have the following duality
relation between the two operators

\begin{lemma}
Under the above assumptions, suppose that $\phi $ is bounded and measurable,
let $\nu \in SM(N),$ then%
\begin{equation*}
\int \phi ~d[L\nu ]=\int [P(\phi )]~d\nu .
\end{equation*}
\end{lemma}

Compare this statement with Proposition \ref{duality}, from which this
proposition easily follows.

Now we consider the measures which are in a certain sense invariant for a
random dynamical system: a measure $\eta \in PM(N)$ is called \emph{%
stationary for }$F_{\omega }$ if 
\begin{equation*}
L\eta =\eta .
\end{equation*}

The following proposition shows a link between being stationary for a random
system and invariant for the associated skew product $F,$for the proof see 
\cite{Viana}, Section 5.

\begin{theorem}
\label{vi1}Let $F:\Omega ^{\mathbb{N}}\times N\rightarrow \Omega ^{\mathbb{N}%
}\times N$ be a one sided random transformation with associated transfer
operator $L$. A probability measure $\nu $ on $N$ is stationary for the
system ($L\eta =\eta $) if and only if the probability measure $\mu \times
\nu $ is invariant for the skew product map $F$.
\end{theorem}

\subsection{Ergodic stationary measures}

In this section we extend the notion of ergodicity to random dynamical
systems. The reader will notice the similarity of the definitions with the
analogous definitions in the deterministic case. Let $\eta $ be a stationary
measure. A bounded measurable function $\phi :N\rightarrow \mathbb{R}$ is
said to be stationary if it satisfies 
\begin{equation*}
P\phi =\phi .
\end{equation*}

A set $B\subset N$ is said to be stationary if $1_{B}$ is stationary. A
bounded measurable function $\phi :N\rightarrow \mathbb{R}$ is said to be $%
\eta -$stationary if it satisfies 
\begin{equation*}
P\phi =\phi
\end{equation*}%
$\eta -a.e.$ A set $B\subset N$ is said $\eta $-stationary if $1_{B}$ is $%
\eta $-stationary.

\begin{example}
The complement of a stationary set is stationary. Consider a stationary set $%
B,$ then $1_{B}=P(1_{B})$. Observe that $1_{B^{C}}=1-1_{B}$ and%
\begin{eqnarray*}
P(1_{B^{C}}) &=&\int_{x\in \Omega }(1-1_{B})\circ F_{x}~dp(x)=1-\int_{x\in
\Omega }(1_{B})\circ F_{x}~dp(x) \\
&=&1-P1_{B}=1_{B^{C}}
\end{eqnarray*}%
hence $B^{C}$ is stationary.
\end{example}

As in the deterministic case, one has the following equivalence

\begin{theorem}
Let $\eta $ be a stationary measure. The following conditions are equivalent:

\begin{enumerate}
\item every $\eta $-stationary function is constant on some set with full $%
\eta $-measure;

\item if $B\subset N$ is an $\eta $-stationary set then $\eta (B)$ is either 
$0$ or $1$.
\end{enumerate}
\end{theorem}

\begin{definition}
\label{ergor}$\eta $ is said to be \emph{ergodic} if $1$ and $2$ holds.
\end{definition}

There is a relation between the concept of ergodicity in the random system,
and the same concept for the associated skew product map (for the proof of
the statement, again see \cite{Viana}, Section 5).

\begin{theorem}
Let $F:\Omega ^{\mathbb{N}}\times N\rightarrow \Omega ^{\mathbb{N}}\times N$
be a "one sided" random transformation as defined in Section \ref{oneside}.
A probability measure $\nu $ on $N$ is ergodic for the random system in the
sense of Definition \ref{ergor} if and only if the probability measure $\mu
\times \nu $ is ergodic for the map $F$.
\end{theorem}

The set of stationary measures is trivially a convex set. As in the
deterministic case, every stationary measure is a convex combination of
ergodic measures (and the ergodic ones are extremal points, see \cite{Ki},
Appendix A.1 for details).

Next proposition shows another aspect of ergodicity in a random dynamical
system. It can be interpreted as the fact that in an ergodic system, given
some integrable observable $f:N\rightarrow \mathbb{R}$ \ for a typical
random orbit $x_{0},...,x_{i}\in N$ the time average of the observable will
coincide with the space average with respect to the stationary measure, like
in the deterministic case.

\begin{proposition}
Let $F:\Omega ^{\mathbb{N}}\times N\rightarrow \Omega ^{\mathbb{N}}\times N$
be the skew product map associated to a random dynamical system. Let us
consider an ergodic stationary measure $\nu $ on $N.$ \ Let us consider an
observable $f\in L^{1}(N,\mathbb{\nu })$ on $N$. Let us consider $\omega
_{0}\in \Omega ^{\mathbb{N}}$ and $x_{0}\in N$ \ as initial conditions. For
such conditions we have a random orbit $x_{i}\in N$ with this initial
condition, defined by $x_{i}:=\pi _{N}(F^{i}(\omega _{0},x_{0}))$ where $\pi
_{N}:\Omega ^{\mathbb{N}}\times N\rightarrow N$ is the natural projection.
With this notations, for $\mu \times \nu $ almost each initial condition $%
(\omega _{0},x_{0})$ it holds:%
\begin{equation*}
\int f~d\nu =\underset{n\rightarrow \infty }{\lim }\frac{%
f(x_{0})+f(x_{i})+...+f(x_{n})}{n+1}.
\end{equation*}
\end{proposition}

\begin{proof}
Let us consider $f\in L^{1}(N,\mathbb{\nu })$. This function can be extended
to $\tilde{f}\in L^{1}(\Omega ^{\mathbb{N}}\times N,\mathbb{\mu \times \nu }%
) $ by setting $\tilde{f}(\omega ,x)=f(x)$. By the ergodic theorem applied
to $F$ we get%
\begin{equation*}
\int \tilde{f}~d[\mu \times \nu ]=\underset{n\rightarrow \infty }{\lim }%
\frac{\tilde{f}(\omega ,x)+\tilde{f}(F(\omega ,x))+...+\tilde{f}%
(F^{n}(\omega ,x))}{n+1}
\end{equation*}%
for $\mu \times \nu $ almost each $(\omega ,x)\in \Omega ^{\mathbb{N}}\times
N$. Now $\forall \omega $ $\underset{n\rightarrow \infty }{\lim }\frac{%
\tilde{f}(\omega ,x)+\tilde{f}(F(\omega ,x))+...+\tilde{f}(F^{n}(\omega ,x))%
}{n+1}=\underset{n\rightarrow \infty }{\lim }\frac{%
f(x_{0})+f(x_{i})+...+f(x_{n})}{n+1}$ the conclusion follows, since trivially%
\begin{equation*}
\int \tilde{f}~d[\mu \times \nu ]=\int f~d\nu .
\end{equation*}
\end{proof}

\subsubsection{Additive noise}

We are going to define a class of systems (with additive noise) which allow
a simple functional analytic treatment. To start let us consider a random
dynamics on $[0,1].$ Everything we say generalizes easilly to other spaces.
We consider a dynamics in which the orbit of a point is generated by a
deterministic map, but were at each iterate some (small or not so small)
random perturbation is added (the noise). We suppose that the perurbation is
independend from the point and always distributed in the same way. For
simplicity we will suppose that the random perturbation ranges in $[-1,1]$.
\ All these assumptions can be generalized, making the noise depending on
the point, following a similar costruction and similar arguments.

More precisely, consider a Borel map $T:[0,1]\rightarrow \lbrack 0,1]$ (the
deterministic part). Consider a probability density $\rho \in BV[-1,1]$ (the
noise kernel). For simplicity let us suppose $\rho $ being symmetric ($\rho
(x)=\rho (-x)$). Consider the process{%
\begin{equation}
x_{n+1}=T(x_{n})+\Omega _{n}  \label{systm}
\end{equation}%
} where $\Omega _{n}$ is an i.i.d. process distributed according to $\rho $ {%
\ $\in BV$ \footnote{{Notice that this can be realized with the }$n$-th{\
coordinate of a shift as shown before: we can set }$\Omega =[-1,1]$, \ the
process being defined by{\ }$\Omega _{n}(\omega )=\omega _{n}$, with $\omega
\in (\Omega ^{\mathbb{N}},p^{\mathbb{N}})$ and $\frac{dp}{dx}=\rho $.}}

This is the idea behind a random dynamical system with additive noise. If we
want to define a dynamics on the interval there is a technical point to
address: $T(x_{n})\in \lbrack 0,1]$ but $T(x_{n})+\Omega _{n}$ could be
outside, i.e the noise can let the point jump outside the interval. We have
to decide where the point goes in this case. For this we can consider
"boundary reflecting conditions", "periodic boundary conditions" or other
ways to define the dynamics on $[0,1]$ in this case. We enter in the details
of the boundary reflecting conditions. Consider hence the process{%
\begin{equation}
x_{n+1}=T(x_{n})\hat{+}\Omega _{n}
\end{equation}%
}

and $\hat{+}$ is the \textquotedblleft reflecting boundaries sum" on $[0,1]$
defined as follows.

\begin{definition}
Let $\pi :{\mathbb{R}}\rightarrow \lbrack 0,1]$ be a piecewise linear
projection on the interval 
\begin{equation}
\pi (x)=\min_{i\in \mathbb{Z}}|x-2i|.  \label{ppi}
\end{equation}%
Let $a,b\in {\mathbb{R}}$ then 
\begin{equation*}
a\hat{+}b:=\pi (a+b)
\end{equation*}%
where $+$ is the usual sum operator on $\mathbb{R}$. By this $a\hat{+}b\in
\lbrack 0,1].$
\end{definition}

Now let us describe the structure of the transfer operator associated to a
system of this kind. By definition%
\begin{equation*}
L(\nu )=\int_{x\in \Omega }L_{F_{x}}(\nu )~dp(x).
\end{equation*}

Were on $\mathbb{R}$, $F_{x}(a)=T(a)+x$, hence $L_{F_{x}}=Tran(x)\circ L_{T}$
where $Tran(x):SM(\mathbb{R)\rightarrow }SM(\mathbb{R)}$ is the pushforward
of the translation by $x$.

On $[0,1]$ with reflecting boundary conditions we have $F_{x}(a)=T(a)\hat{+}%
x $ and then 
\begin{equation*}
L_{F_{x}}=\pi _{\ast }\circ Tran(x)\circ L_{T}.
\end{equation*}

On $\mathbb{R}$, since $\rho $ is symmetric and the Lebesgue measure is
invariant by translation, if we denote with $\mu _{\rho }=\rho m$ the
measure on $[-1,1]$ having density $\rho $. 
\begin{equation}
L(\nu )=\int_{x\in \Omega }Tran(x)(L_{T}(\nu ))\cdot \rho (x)~dx=\mu _{\rho
}\ast L_{T}(\nu )  \label{ll}
\end{equation}

where $\ast $ stands for the ordinary convolution operator \ between
measures.\footnote{%
We recall that given two measures $\mu ,\nu \in BS[{\mathbb{R}}]$ their
convolution is defined as 
\begin{equation*}
\mu \ast \nu (A)=\int_{{\mathbb{R}}^{2}}{1}_{A}(x+y)d\mu (x)d\nu (y).
\end{equation*}%
}

On $[0,1]$ with the boundary reflecting conditions we then have%
\begin{equation*}
L(\nu )=\pi _{\ast }\circ \lbrack \mu _{\rho }\ast L_{T}(\nu )]:=\mu _{\rho }%
\hat{\ast}L_{T}(\nu ).
\end{equation*}

The transfer operator associated to the system is then the composition of
the pushforward $\pi _{\ast }$ \ associated to $\pi $, the convolution and
the transfer operator related to $T$.

Now let us consider absolutely continuous measures. Let us suppose that $T$
is nonsingular, and hence $L_{T}$ can be seen as an operator $%
L^{1}([0,1])\rightarrow L^{1}([0,1])$. Since the convolution between
measures, when applied to measure having a density, becomes the usual
convolution operator between the associated densities and also $\pi _{\ast }$
preserves absolutely contonuous measures we have that the transfer operator
of the random system $L$ defined at \ref{ll} can be considered as an
operator $L^{1}\rightarrow L^{1}$ sending a measure density to another one.

The well known regularizing properties of the convolution will then make the
operator $L:L^{1}\rightarrow L^{1}$ being a regularizing \ one and we can
get easy quantitative estimates on this. These estimates will play the role
of the Lasota Yorke ones in the theory of deterministic systems. \ We have
indeed the following regularization inequalities:

\begin{lemma}
\label{convoocopy1} Let $f,$ $g\in L^{1}$. We have 
\begin{equation}
\Vert f\hat{\ast}g\Vert _{1}\leq \Vert f\Vert _{1}\cdot \Vert g\Vert _{1}.
\end{equation}%
Furthermore, let $f\in L^{1},$ $g\in BV$%
\begin{equation}
\Vert f\hat{\ast}g\Vert _{BV}\leq 3\Vert f\Vert _{1}\cdot \Vert g\Vert _{BV}.
\label{conv3}
\end{equation}
\end{lemma}

Under the above assumptions, if $\rho \in BV,\nu \in L^{1}$and $n\geq 0$%
\begin{equation}
||L^{n}(\nu )||_{BV}\leq 3||\rho ||_{BV}||\nu ||_{1}.  \label{llll}
\end{equation}

\begin{proof}
We know that $\pi _{\ast }$ is a weak contraction with respect to the $L^{1}$
norm. By the classical properties of the convolution we have $\Vert f\ast
g\Vert _{1}\leq \Vert f\Vert _{1}\cdot \Vert g\Vert _{1}$, then also $\Vert f%
\hat{\ast}g\Vert _{1}\leq \Vert f\Vert _{1}\cdot \Vert g\Vert _{1}.$

About Equation $($\ref{conv3}$)$ on the real line we have the well known
estimate $\Vert f\ast g\Vert _{BV}\leq \Vert f\Vert _{1}\cdot \Vert g\Vert
_{BV}$. For the convolution on the interval let us remark that if $\mu $ is
supported on $[-1,2]$ then $\Vert \pi ^{\ast }(\mu )\Vert _{BV}\leq 3\Vert
\mu \Vert _{BV}$ , giving $(\ref{conv3})$.

About $(\ref{llll})$, since $L$ is a weak contraction with respect to the $%
L^{1}$ norm%
\begin{equation*}
||L^{n}(\nu )||_{BV}\leq 3||\rho ||_{BV}||L^{n-1}(\nu )||_{1}\leq 3||\rho
||_{BV}||\nu ||_{1}.
\end{equation*}
\end{proof}

The previous lemma implies that the tranfer operator $L$ can also be seen as 
$L:BV[0,1]\rightarrow BV[0,1]$. Now, considering $L$ as an operator acting
on a strong and a weak space, $BV$ and $L^{1}$ and exploiting a strategy
very similar to what we have done for expanding and piecewise expending
maps, we prove that a system with additive noise has a stationary measure
with density in $BV$. In this construction the Helly selection principle
(see Theorem \ref{Helly}) will provide the compact immersion between the
strong and the weak space.

\begin{theorem}
\label{exist1}A random dynamical system with additive noise with $\rho \in
BV $ has an absolutely continuous stationary measure having density $h\in BV$
and 
\begin{equation*}
||h||_{BV}\leq 3||\rho ||_{BV}.
\end{equation*}
\end{theorem}

\begin{proof}
Consider the Cesaro averages of iterates of an uniform distribution%
\begin{equation*}
f_{n}=\frac{1}{n}\sum_{i=0}^{n-1}L^{n}1
\end{equation*}
where $1$ is the density of the normalized Lebesgue measure.\ 

By Lemma \ref{convoocopy1}, the sequence has uniformly bounded $BV$ norm and
by Theorem \ref{Helly} this has a subsequence $f_{n_{k}}$ converging in $%
L^{1}$ to a limit $h$.

Since $f_{n}=\frac{1}{n}(L^{0}1+...+L^{n}1)$ we get $%
||Lf_{n_{K}}-f_{n_{k}}||_{1}\leq \frac{2}{n}$. \ Recalling that $L$ is
continuous in the $L^{1}$ norm we get 
\begin{equation*}
Lh=L(\lim_{k\rightarrow \infty }f_{n_{k}})=\lim_{k\rightarrow \infty
}Lf_{n_{k}}=h.
\end{equation*}%
Then $h$ is a stationary probability density in $L^{1}$. Applying again
Lemma \ref{convoocopy1} \ we get 
\begin{equation*}
||h||_{BV}=||Lh||_{BV}\leq 3||\rho ||_{BV}||h||_{1}<\infty
\end{equation*}%
and we proved the statement.
\end{proof}

Now we discuss the problem of how to understand when a system is ergodic and
stronger notions like the notion of convergence to equilibrium (see
Definition \ref{conve}). In the following proposition we see that in a
system with additive noise the convergence to equilibrium implies
ergodicity, thus we can apply the results we know to estimate the
statistical behavior of observables. In next subsection we will show a
general criteria (Lemma \ref{critconv}) to establish the convergence to
equilibrium in such systems.

\begin{proposition}
If a dynamical system with additive noise as above has convergence to
equilibrium (using $BV$ and $L^{1}$ as a strong and weak space, see $(\ref%
{CE})$) then the system is ergodic.
\end{proposition}

\begin{proof}
Consider a stationary set $B$. Suppose that $\mu $ is a stationary measure
for the system. It holds for each $\phi \in BV$

with $\phi \geq 0$%
\begin{equation*}
\int 1_{B}~\phi ~d\mu =\int P^{n}(1_{B})~\phi ~d\mu =\int 1_{B}~d(L^{n}\phi
\mu ).
\end{equation*}%
Since $\phi \mu \in BV$ , if the system has convergence to equilibrium $%
L^{n}\phi \mu \rightarrow \lbrack \int \phi ~d\mu ]\mu $ in $L^{1}$. We get
then that\emph{\ }$\forall \phi \in BV~$with $\phi \geq 0$%
\begin{equation*}
\int 1_{B}~\phi ~d\mu =\int 1_{B}~d\mu \int \phi ~d\mu .
\end{equation*}%
\emph{\ }Thus $1_{B}$ is a.e constant and $\mu (B)\in \{0,1\}$. The system
is hence ergodic.
\end{proof}

We now see that if a map with additive noise \ has convergence to
equilibrium then it also has spectral gap, and then the convergence to
equilibrium is exponentially fast among many other consequences.

\begin{proposition}
Suppose the transfer operator $L:BV\rightarrow BV,$ associated to a map with
additive noise as above has convergence to equilibrium using $BV$ and $L^{1}$
as a strong and weak space (see $(\ref{CE})$) then it has spectral gap.
\end{proposition}

\begin{proof}
The statement directly follows by Theorem \ref{gap}. We consider $%
B_{s}=BV[0,1]$ and $B_{w}=L^{1}[0,1]$ as strong and weak space. Weak
boundedness in $L^{1}$is granted.\ We have compact inclusion of the strong
space into the weak one (Theorem \ref{Helly}) Furthermore 
\begin{equation}
||L^{n}(\nu )||_{BV}\leq 3||\rho ||_{BV}||\nu ||_{1}  \label{reg}
\end{equation}%
is a Lasota Yorke inequality with $\lambda =0.$
\end{proof}

\begin{remark}
Once established the spectral gap on $BV$, the spectral gap on $L^{1}$for
this kind of systems follows easily. Indeed, let us consider $\nu \in \{\mu
\in L^{1}~s.t.~\mu ([0,1])=0\}.$ By $($\ref{reg}$)$ $~||L(v)||_{BV}\leq
3||\rho ||_{BV}||\nu ||_{1}$ and $L(v)\in \{\mu \in BV~s.t.~\mu
([0,1])=0\}:=V_{BV}$. By the spectral gap on $BV$ there is $A\geq 0,$ $%
\lambda \in (0,1)$ such that 
\begin{eqnarray*}
||L^{n}(\nu )||_{1} &\leq &||L^{n}(\nu )||_{BV} \\
&\leq &||L^{n-1}(L\nu )||_{BV} \\
&\leq &A\lambda ^{n-1}3||\rho ||_{BV}||\nu ||_{1}
\end{eqnarray*}%
implying spectral gap on $L^{1}.$
\end{remark}

As an example of a first simple application of the previous proposition we
state the following

\begin{proposition}
Let $(x_{n})_{n\geq 0}$ \ be a random orbit starting from the initial
condition $x_{0}$ and a realization $\omega _{0}$ of the noise. Let
\thinspace $\phi \in L^{1}(N).$ Let us define the average behavior of the
random orbits on such realizations%
\begin{equation*}
E(\phi (x_{n})):=\int_{\Omega ^{\mathbb{N}}}\phi (x_{n})~d\omega _{0}.
\end{equation*}

In a map with additive noise having convergence to equilibrium one has that
for each $x_{0}$%
\begin{equation*}
|E(\phi (x_{n}))-\int \phi ~d\mu |=O(e^{-\lambda n}).
\end{equation*}
\end{proposition}

\begin{proof}
Let us first remark that we can apply the transfer operator to a delta
measure $\delta _{x_{0}}$ and get a measure with $BV$ density: indeed $%
L(\delta _{x_{0}})=\pi _{\ast }\circ \lbrack \rho \ast \delta _{T(x_{0})}]$,
and $\rho \ast \delta _{T(x_{0})}$ is a measure having $BV$ density. We
remark that given \thinspace $\phi \in L^{1}(N)$ and $x_{0}\in N$ we get%
\begin{eqnarray*}
E(\phi (x_{1})) &=&\int_{x\in \Omega }\phi (F_{x}(x_{0}))~dp(x) \\
&=&[P(\phi )](x_{0}) \\
&=&\int P(\phi )d\delta _{x_{0}} \\
&=&\int \phi ~dL(\delta _{x_{0}})
\end{eqnarray*}%
and continuing the iteration%
\begin{equation*}
\int_{\Omega ^{\mathbb{N}}}\phi (x_{2})=\int_{\Omega }\int_{\Omega }\phi
(F_{\omega _{1}}(F_{\omega _{0}}(x_{0})))d\omega _{1}d\omega
_{0}=\int_{\omega _{0}}[P\phi ](F_{\omega _{0}}(x_{0}))~...
\end{equation*}%
We then get%
\begin{eqnarray*}
E(\phi (x_{n})) &=&[P^{n}(\phi )](x_{0}) \\
&=&\int \phi ~dL^{n}(\delta _{x_{0}}).
\end{eqnarray*}
But by the spectral gap $||L^{n}(\delta _{x_{0}})-\mu ||_{BV}=O(e^{-\lambda
n}).$ Then for each $x_{0}$%
\begin{equation*}
|E(\phi (x_{n}))-\int \phi ~d\mu |=O(e^{-\lambda n}).
\end{equation*}
\end{proof}

Having established that the transfer operators associated to systems with
additive noise have good properties as the spectral gap, on $BV$ or $L^{1}$
we can appply the results of Section \ref{stab} to suitable families of
perturbations of such systems and get quantitative stability estimates. We
briefly discuss this direction of work, and leave the details to the reader.

We now consider the easiest way to perturb such a system with additive
noise, by perturbing the distribtion of the noise. Let us consider the
following proposition:

\begin{proposition}
\label{nearn} Let $f\in L^{1}$, suppose $\Vert \rho _{1}-\rho _{0}\Vert
_{1}\leq C\xi $ \ and $L_{1,T},L_{0,T}$ are the transfer operators
associated to the systems given by the map $T$ with additive noise
distributed as $\rho _{1}$ and $\rho _{0}.$ Then 
\begin{equation}
\Vert L_{1,T}f-L_{0,T}f\Vert _{1}\leq C\xi \Vert f\Vert _{1}.
\end{equation}
\end{proposition}

\begin{proof}
The proof is a direct application of the well known fact $\Vert \rho \ast
g\Vert _{1}\leq \Vert \rho \Vert _{1}\Vert g\Vert _{1}$. By this it also
holds 
\begin{equation}
\Vert \rho \hat{\ast}g\Vert _{1}=\pi _{\ast }(\hat{\rho}\ast \hat{g})\leq
\Vert \rho \Vert _{1}\Vert g\Vert _{1}  \label{convl1}
\end{equation}

and $\Vert \rho _{0}\hat{\ast}g-\rho _{\xi }\hat{\ast}g\Vert _{1}\leq \Vert
\rho _{\xi }-\rho _{0}\Vert _{1}\Vert g\Vert _{1}$.
\end{proof}

This last proposition can be used as to show that the property UF2 \ is
verified and estblish $\delta \log \delta $ \ modulus of continuity for the
statistical stability of a map with additive noise having convergence to
equilibrium, but also Lipshitz statistical stability (using the results of
Section \ref{lipsec} and the spectral gap on $L^{1}$). We suggest the
interested reader to fill the details.

Similar results can be proved for other kinds of perturbations of the
transfer operator associated to maps with additive noise like the Ulam
discretization (see Section \ref{dopo} and \cite{GMN}) or perturbations of
the map $T$ (see \cite{GG,GM} \ for pertubations on the map or linear
response results for these systems up to perturbation of the map or noise,
including zero-noise limits).

\subsubsection{A way to establishing convergence to equilibrium in systems
with additive noise.}

Let us consider a family of nonsingular maps with additive noise. We will
suppose that these maps are nonsingular and piecewise $C^{2}$ \footnote{%
The space $[0,1]$ can be decomposed into a union of intervals $I_{i}$ such
that in every set $\overline{I_{i}}$ the map $T_{0}$ can be extended to a $%
C^{2}$ function $\overline{I_{i}}\rightarrow \lbrack 0,1].$} having good
distortion properties. More precisely let us suppose:

\begin{itemize}
\item[A1] $T_{0}$ is a nonsingular piecewise $C^{2}$ map whose associated
pushforward operator $L_{T_{0}}:BV[0,1]\rightarrow BV[0,1]$ is continuous:
there is $C\geq 0$ such that%
\begin{equation*}
||L_{T_{0}}f||_{BV}\leq C||f||_{BV}.
\end{equation*}

\item[A2] $T_{0}$ is eventually onto: for each open interval $I\subseteq
\lbrack 0,1]$ there is $n$ such that $T_{0}^{n}(I)=[0,1]$.
\end{itemize}

Suppose now that at each iteration of $T_{0}$ a random perturbation is added
with noise kernel $\rho _{0}\in BV$, and we suppose that the support of $%
\rho _{0}$ contains a neighborhhod of the origin. We have a system with
additive noise. Let us see that under these conditions this system has
convergence to equilibrium.

\begin{remark}
Assumption $A2$ \ {implies the} topological mixing of the map. In the case
of piecewise expanding maps this assumption is often taken to get a
topologically mixing system (see \cite[Section 3.1, Property E3]{Viana}).
This assumption is not the most general one possible to get convergence to
equilibrium, but it will keep the exposition simple.
\end{remark}

\begin{lemma}
\label{critconv}Under the assumptions above, the transfer operator $L$
associated to the system has convergence to equilibrium: for each $f\in BV$
with $\int f=0$ we have $||L_{0}^{n}f||_{1}\rightarrow 0.$
\end{lemma}

\begin{proof}
By Lemma \ref{convoocopy1} \ and Theorem \ref{Helly} \ $L$ is a compact
operator $L^{1}\rightarrow L^{1}.$ The spectral radius of $L$ as an operator
on $L^{1}$ is bounded by $1$, and the spectrum is discrete. This also hold
for $L$ considered on $BV$.

The system has convergence to equilibrium if and only if there are no other
eigenvalues on the unit circle than the eigenvalue $1$ with multiplicity $1$.

Let us consider one positive stationary probability measure $\mu _{0}\in BV$
for $L$ (as was proved to exist in Theorem \ref{exist1}). \ Suppose that the
system has no convergence to equilibrium, then there is a complex measure $%
\hat{\mu}\in BV$, $\hat{\mu}\neq \mu _{0}$ and {\ }$\lambda \in \mathbb{C}$
with $|\lambda |=1$ such that {$L^{i}\hat{\mu}=\lambda ^{i}\hat{\mu}{\ }$}
for each $i\geq 0$. Let $\mu $ be the real part of $\hat{\mu}.$ This is a
signed measure. Since the transfer operator preserves real valued measures $%
L^{i}\mu $ is a real valued measure with bounded variation density and, for
each $\epsilon \geq 0$ there are infinitely many $i$ such that {\ $\Vert
L^{i}\mu -\mu \Vert _{1}\leq \epsilon $ \ (}for each $\epsilon $, $|\lambda
^{i}-1|\leq \epsilon $ for infinitely many $i$). We also have that there is $%
c\in \mathbb{R}$ such that $\mu _{1}=\mu +c\mu _{0}$ is a zero average
measure with density in $BV$ and 
\begin{equation*}
\underset{i\rightarrow \infty }{\lim \sup }{\Vert L_{0}^{i}\mu _{1}\Vert
_{1}=\Vert \mu _{1}\Vert _{1}}
\end{equation*}%
(indeed ${\Vert L_{0}^{i}\mu _{1}\Vert _{1}\leq \Vert \mu _{1}\Vert _{1}}$ $%
\ $and ${\Vert L_{0}^{i}\mu _{1}\Vert _{1}=||L_{0}^{i}(}\mu +c\mu
_{0})||_{1}=||{L_{0}^{i}(}\mu )-\mu +\mu +{L_{0}^{i}(}c\mu _{0})||_{1}\leq ||%
{L_{0}^{i}(}\mu )-\mu ||_{1}+{\Vert \mu _{1}\Vert _{1}}$ ).

Now let $I$ \ be an interval for which $\mu _{1}|_{I}$ has a strictly
positive density. By assumption $A2$ there is $n$ such that $%
T_{0}^{n}(I)=[0,1]$. Let us consider the measure $\nu =\mu _{1}1_{I}$.
Suppose $s(\nu )$ is the support of $\nu $ (the set on which $\nu $ has
strictly positive density). Since $L_{T_{0}}$ is a positive operator and $%
T_{0}$ is piecewise $C^{2}$, then $L_{T_{0}}\nu $ has also a strictly
positive density almost everywhere on $T_{0}(I).$ Since the convolution can
only increase the support of a positive measure, we get that $s(L_{0}\nu
)\supseteq \overline{s(L_{T_{0}}\nu )}$ $\supseteq T_{0}(I)$, $%
s(L_{0}^{2}\nu )\supseteq T_{0}^{2}(I)$ and $s(L_{0}^{i}\nu )\supseteq
T_{0}^{i}(I)$ for $i\geq 1$. Then by assumption $A2$ we get inductively $%
s(L_{0}^{n}\nu )=[0,1]$. \ This contraddicts the fact that $\underset{%
i\rightarrow \infty }{\limsup }{\Vert L_{0}^{i}\mu _{1}\Vert _{1}}${$=\Vert
\mu _{1}\Vert _{1}$}. Indeed \ {recall that any measure }$\mu _{1}${\ of
zero average can be decomposed in $\mu _{1}^{+}+\mu _{1}^{-}$, the positive
and negative component of }$\mu _{1}${. }We have that $L_{0}^{n}\mu _{1}^{-}$
is a negative measure having a bounded variation density and the support of $%
L_{0}^{n}\nu $ being the whole space overlaps the support of $L_{0}^{n}\mu
_{1}^{-}$ in this way 
\begin{eqnarray*}
{\Vert L_{0}^{n}\mu }_{1}{\Vert _{1}} &=&\Vert L_{0}^{n}(\mu _{1}^{+}+\mu
_{1}^{-})\Vert _{1} \\
&\leq &\Vert L_{0}^{n}(\mu _{1}^{+}-\nu +\nu +\mu _{1}^{-})\Vert _{1} \\
&\leq &\Vert L_{0}^{n}(\mu _{1}^{+}-\nu )||_{1}+||L_{0}^{n}\nu +L_{0}^{n}\mu
_{1}^{-}\Vert _{1} \\
&<&\Vert L_{0}^{n}(\mu _{1}^{+}-\nu )||_{1}+||\mu _{1}^{-}||_{1}+||\nu ||_{1}
\\
&=&||\mu _{1}||_{1}.
\end{eqnarray*}

Then for each $k\geq 0$ $\ {\Vert L_{0}^{n+k}\mu }_{1}{\Vert _{1}\leq \Vert
L_{0}^{n}\mu }_{1}{\Vert _{1}<}||\mu _{1}||_{1}$, contraddicting $\underset{%
i\rightarrow \infty }{\lim \sup }{\Vert L_{0}^{i}\mu _{1}\Vert _{1}}${$%
=\Vert \mu _{1}\Vert _{1}$}.
\end{proof}

\section{Uniformly contracting maps\label{sec:contr}}

We have seen how expansion helps the transfer operator to have
regularization properties when considered on suitable spaces of regular
measures (expressed by the Lasota Yorke inequality). If the dynamics is
contracting instead we cannot expect such a regularization on the same
spaces. Iterates of absolutely continuous initial measures instead will
converge to a Dirac delta measure which is not regular at all according to
the norms considered for expanding maps. \ We will see in this section that
considering different spaces, which are essentially the dual of the spaces
considered in the expanding case we recover a regularization property, even
for the transfer operator associated to a contracting map. This will be
source of inspiration to define the suitable space to consider when
considering maps having both contracting and expanding directions
(hyperbolic systema) as it happens in many models of interesting phisical
phenomena.

Let $(X,d)$ be a compact metric space, $g:X\longrightarrow \mathbb{R}$ be a
Lipschitz function and let $Lip(g)$ be its best Lipschitz constant, i.e. 
\begin{equation*}
\displaystyle{Lip(g)=\sup_{x,y\in X}\left\{ \dfrac{|g(x)-g(y)|}{d(x,y)}%
\right\} }.
\end{equation*}

\begin{definition}
Given two Borel signed measures $\mu $ and $\nu $ on $X,$ we define a 
\textbf{Wasserstein-Kantorovich Like} distance between $\mu $ and $\nu $ by%
\begin{equation}
W_{1}^{0}(\mu ,\nu )=\sup_{g~s.t.~Lip(g)\leq 1,||g||_{\infty }\leq
1}\left\vert \int {g}d\mu -\int {g}d\nu \right\vert .
\end{equation}%
\label{wasserstein}
\end{definition}

From now on we denote%
\begin{equation*}
||\mu ||_{W}:=W_{1}^{0}(0,\mu ).
\end{equation*}%
As a matter of fact, $||\cdot ||_{W}$ defines a norm on the vector space of
signed measures defined on a compact metric space.

\begin{remark}
\label{posi} The space of signed measures is not complete with respect to
the $W_{1}^{0}$ distance, and its completion would be a distribution space.
However for our purposes it is sometime sufficient to consider sequences of
positive measures. The set of positive Borel measures on $X$ however is
complete with respect to the distance $W_{1}^{0}$ (see \cite{Ba}, \cite{BK}).
\end{remark}

We show that a contracting map $F$, such that $Lip(F)<1$ is in some sense
regularizing for the above norm

\begin{lemma}
\label{unamis}Let \ $F$ $:X\rightarrow X$,\ be a Lipschitz function and $X$
is a metric space. For every Borel measure with sign $\mu $ it holds%
\begin{equation}
||L_{F}\mu ||_{W}\leq \alpha ||\mu ||_{W}+\mu (X).  \label{LYdual}
\end{equation}%
(where $\alpha =Lip(F)$). In particular, if $\mu (X)=0$ then%
\begin{equation*}
||L_{F}\mu ||_{W}\leq \alpha ||\mu ||_{W}.
\end{equation*}%
\label{quasicontract}
\end{lemma}

\begin{proof}
If $Lip(g)\leq 1$ and $||g||_{\infty }\leq 1$, then $g\circ F$ is $\alpha $%
-Lipschitz. Moreover since $||g||_{\infty }\leq 1$ then $||g\circ F_{\gamma
}-\theta ||_{\infty }\leq \alpha $ for some $\theta \leq 1$. This implies%
\begin{align*}
\left\vert \int {g}dL_{F}\mu \right\vert & =\left\vert \int {g\circ F}d\mu
\right\vert \\
& =\left\vert \int {g\circ F-\theta }d\mu \right\vert +\left\vert \int {%
\theta }d\mu \right\vert \\
& =\alpha \left\vert \int {\frac{g\circ F-\theta }{\alpha }}d\mu \right\vert
+\theta \mu (X) \\
& =\alpha \left\vert \left\vert \mu \right\vert \right\vert _{W}+\mu (X).
\end{align*}%
And we have $||L_{F}\mu ||_{W}\leq \alpha ||\mu ||_{W}+\mu (X)$. In
particular, if $\mu (X)=0$ we get the second part.
\end{proof}

\begin{itemize}
\item From Lemma \ref{unamis} it easily follows that $L_{F}$ has spectral
gap on the normed vector space of signed measured endowed with the norm $%
||~||_{W}$. Since the space of signed Borel measures is not complete, what
we are going to prove precisely is that (Compare with Definition \ref{defgap}%
)%
\begin{equation}
L_{F}=\func{P}+\func{N}  \label{ab1}
\end{equation}%
where

\item $\func{P}$ is a projection and $\dim (Im(\func{P}))=1$;

\item there is $\alpha ,C\geq 0$ such that $\alpha <1$ and for each $n,$ and
each Borel signed measure $\mu ,$ $||\func{N}^{n}\mu ||_{w}<\alpha
^{n}C||\mu ||_{w}$ ;

\item $\func{P}\func{N}=\func{N}\func{P}=0$.
\end{itemize}

\begin{proposition}
The transfer operator $L_{F}$ associated to a \ contracting map $F$ has the
decomposition $(\ref{ab1})$ as above.
\end{proposition}

\begin{proof}
Being a contraction the map $F$ has a unique fixed point $x_{0}$. Let us
consider the Dirac measure $\delta _{x_{0}}$ placed on $x_{0}$. This is a
fixed point for $L_{F}$, generating a one dimensional fixed space $\mathbb{R}%
\delta _{x_{0}}$ and one can define a projection $\func{P}$ to $\mathbb{R}%
\delta _{x_{0}}$ as $\mu _{0}\rightarrow \mu _{0}(X)\delta _{x_{0}}$. Now
let us consider the $L_{F}$ invariant space $V$ of zero average measures,
there is also a projection to this space, defined as $\mu _{0}\rightarrow
\mu _{0}-\mu _{0}(X)\delta _{x_{0}}$, and define $\func{N}$ as \ $\func{N}%
(\mu _{0})=L_{F}\mu _{0}-\mu _{0}(X)\delta _{x_{0}}$. By Lemma \ref{unamis}
we get that for any measure $\mu $, $\func{N}^{n}(\mu )$ converges
exponentially to $0$.
\end{proof}

From this proposition one could recover many consequences for the
statistical properties of the dynamics of contracting maps, but since this
dynamics is quite trivial (every initial condition is attracted by the fixed
point of $F$) there is no need to do this. This example and the spaces
considered to get a spectral gap in this case are on the other hand,
illuminating to find a transfer operator approach for systems having both
contracting and expanding directions. In fact one idea can be to consider in
some sense some measure space endowed with some norm which behaves as the
Lipshitz norm\ or some Sobolev norm in the directions for which the dynamics
is expanding and as its dual (i.e. the $||~||_{W}$ norm) in the contracting
directions. This is what we will do in the next section on some simple
example of Hyperbolic dynamical systems and in many papers dealing with this
kind of dynamical systems (see \cite{D} and \cite{DKL} \ for a survey and a
detailed treatise on this subject).

\begin{remark}
If $\mu ^{+}$ and $\mu ^{-}$ are positive measures such that $\mu =\mu
^{+}-\mu ^{-}$ (the Jordan decomposition of $\mu $) then one can define the
total variation of $\mu $ as%
\begin{equation*}
||\mu ||_{TV}:=\mu ^{+}(X)+\mu ^{-}(X)
\end{equation*}%
then one has from $(\ref{LYdual})$ 
\begin{equation*}
||L_{F}\mu ||_{W}\leq \alpha ||\mu ||_{W}+||\mu ||_{TV}
\end{equation*}%
which looks like a Lasota Yorke inequality for the operator $L_{F}$ but it
is not, since $||.||_{TV}$ is not weaker than $||.||_{W}$. For an approach
based on a real Lasota Yorke inequality and a statement like Theorem \ref%
{gap} to get spectral gap on a weak space like $||.||_{W}$ we refer to \cite%
{D}.
\end{remark}

\section{A look at hyperbolic systems\label{yp}}

The transfer operator approach we described in the previous \ sections also
works for a large class of systems with uniform expansion and contraction
rate when appropriate functional spaces are considered. In this section, we
give an example of this for a class of uniformly hyperbolic solenoidal maps.
Following an approach of \cite{GLu}, based on the disintegration along
stable manifolds, we show how to define spaces of measures with sign adapted
to this system. We show some properties of the transfer operator restricted
to these spaces, giving the existence of a physical measure for these
systems and a Lasota Yorke inequality allowing to estimate the regularity of
iterates of measures. Quantitative statistical stability and spectral gap
can be obtained with these kind of construction, but this is outside the
scope of these elementary lectures, for more information about this see \cite%
{GLu}, \cite{G}.

A solenoidal map is a $C^{2}$ map $F:X\rightarrow X$ where $X=S^{1}\times
D^{2}$ the filled torus, such that $F$ is a skew product 
\begin{equation}
F(x,y)=(T(x),G(x,y)),  \label{1eq}
\end{equation}%
where $T:S^{1}\longrightarrow S^{1}$ and $G:X\longrightarrow D^{2}$ are
differentiable maps. We suppose the map $T:S^{1}\rightarrow S^{1}$ to be $%
C^{2}$, expanding of degree $q$, giving rise to a map $[0,1]\rightarrow
\lbrack 0,1]$, which by a small abuse of notation we denote by $T$ \ and
whose branches will be denoted by $T_{i}$, $i\in \lbrack 1,..,q]$ and we
make the following assumptions on $G:$

\begin{itemize}
\item Consider the $F$-invariant foliation $\mathcal{F}^{s}:=\{\{x\}\times
D^{2}\}_{x\in S^{1}}$. We suppose that $\mathcal{F}^{s}$ is contracted:
there exists $0<\alpha <1$ such that for all $x\in S^{1}$ holds%
\begin{equation}
|G(x,y_{1})-G(x,y_{2})|\leq \alpha |y_{1}-y_{2}|\ \ \mathnormal{for\ all}\ \
y_{1},y_{2}\in D^{2}.  \label{contracting1}
\end{equation}

\item $||\frac{\partial G}{\partial x}||_{\infty }<\infty .$
\end{itemize}

We construct now some function spaces which are suitable for the systems we
consider. The idea is to consider spaces of measures with sign, with
suitable norms constructed by disintegrating measures along the stable
foliation. Thus a measure will be seen as a collection (a path, see Remark %
\ref{path}) of measures on each leaf. In the stable direction (and on the
leaves) we will consider a norm which is the dual of the Lipschitz norm. In
the expanding direction, since we have an expanding map, we will consider
the $L^{1}$ norm or a suitable Sobolev norm.

Let $SM(X)$ be the set of Borel signed measures on $\Sigma $. Given $\mu \in
SM(X)$ denote by $\mu ^{+}$ and $\mu ^{-}$ the positive and the negative
parts of it ($\mu =\mu ^{+}-\mu ^{-}$).

Denote by $\mathcal{AB}$ the set of signed measures $\mu \in SM(X)$ such
that its associated marginal signed measures, $\mu _{x}^{\pm }=\pi
_{x}^{\ast }\mu ^{\pm }$ are absolutely continuous with respect to the
Lebesgue measure $m$, on $S^{1}$ i.e. \ 
\begin{equation}
\mathcal{AB}=\{\mu \in SM(X):\pi _{x}^{\ast }\mu ^{+}<<m\ \ \mathnormal{and}%
\ \ \pi _{x}^{\ast }\mu ^{-}<<m\}  \label{thespace1}
\end{equation}%
where $\pi _{x}:X\longrightarrow S^{1}$ is the projection defined by $\pi
(x,y)=x$ and $\pi _{x}^{\ast }$ is the associated pushforward map.

Let us consider a finite positive measure $\mu \in \mathcal{AB}$ on the
space $X$ foliated by the contracting leaves $\mathcal{F}^{s}=\{\gamma
_{l}\}_{l\in S^{1}}$ such that $\gamma _{l}={\pi _{x}}^{-1}(l)$. The Rokhlin
Disintegration Theorem describes a disintegration $\left( \{\mu _{\gamma
}\}_{\gamma },\mu _{x}=\phi _{x}m\right) $ by a family $\{\mu _{\gamma
}\}_{\gamma }$ of probability measures on the stable leaves\footnote{%
In the following to simplify notations, when no confusion is possible we
will indicate the generic leaf or its coordinate with $\gamma $.} and a non
negative marginal density $\phi _{x}:S^{1}\longrightarrow \mathbb{R}$ with $%
||\phi _{x}||_{1}=\mu (X)$.

\begin{remark}
The disintegration of a measure $\mu $ is the $\mu _{x}$-unique measurable
family $(\{\mu _{\gamma }\}_{\gamma },\phi _{x})$ such that, for every
measurable set $E\subset X$ it holds 
\begin{equation}
\mu (E)=\int_{S^{1}}{\mu _{\gamma }(E\cap \gamma )}d\mu _{x}(\gamma ).
\end{equation}%
\label{rmkv}
\end{remark}

\begin{definition}
Let $\pi _{\gamma ,y}:\gamma \longrightarrow D^{2}$ be the restriction $\pi
_{y}|_{\gamma }$, where $\pi _{y}:X\longrightarrow D^{2}$ is the projection
defined by $\pi _{y}(x,y)=y$ and $\gamma \in \mathcal{F}^{s}$. Given a
positive measure $\mu \in \mathcal{AB}$ and its disintegration along the
stable leaves $\mathcal{F}^{s}$, $\left( \{\mu _{\gamma }\}_{\gamma },\mu
_{x}=\phi _{x}m_{1}\right) $ (where $m_{1}$ is the Lebesgue measure on $%
S^{1} $) we define the \textbf{restriction of $\mu $ on $\gamma $} as the
positive measure $\mu |_{\gamma }$ on $D^{2}$ (not on the leaf $\gamma $)
defined, for all mensurable set $A\subset D^{2}$, as 
\begin{equation*}
\mu |_{\gamma }(A)=\pi _{\gamma ,y}^{\ast }(\phi _{x}(\gamma )\mu _{\gamma
})(A).
\end{equation*}%
For a given signed measure $\mu \in \mathcal{AB}$ and its decomposition $\mu
=\mu ^{+}-\mu ^{-}$, define the \textbf{restriction of $\mu $ on $\gamma $}
by%
\begin{equation}
\mu |_{\gamma }=\mu ^{+}|_{\gamma }-\mu ^{-}|_{\gamma }.
\end{equation}%
\label{restrictionmeasure}
\end{definition}

\begin{definition}
\label{l1likespace}Let $\mathcal{L}^{1}\subseteq \mathcal{AB}$ be defined as%
\begin{equation}
\mathcal{L}^{1}=\left\{ \mu \in \mathcal{AB}:\int_{S^{1}}{W_{1}^{0}(\mu
^{+}|_{\gamma },\mu ^{-}|_{\gamma })}dm_{1}(\gamma )<\infty \right\}
\label{L1measurewithsign}
\end{equation}%
and define a norm on it, $||\cdot ||_{"1"}:\mathcal{L}^{1}\longrightarrow 
\mathbb{R}$, by%
\begin{equation}
||\mu ||_{"1"}=\int_{S^{1}}{W_{1}^{0}(\mu ^{+}|_{\gamma },\mu ^{-}|_{\gamma
})}dm_{1}(\gamma ).  \label{l1normsm}
\end{equation}%
The notation we use for this norm is similar to the usual $L^{1}$ norm.
\end{definition}

\begin{remark}
\- \label{path}Indeed this is formally the case of some $L^{1}$ norm if we
associate to $\mu ,$ by disintegration, a path $G_{\mu }:S^{1}\rightarrow 
\mathcal{SB}(D^{2})$ defined by $\ G_{\mu }($ $\gamma )=\mu |_{\gamma }$. In
this case, this will be the $L^{1}$ norm of the path. For more details about
the disintegration and the properties of the restriction, see the appendix
of \cite{GLu}.
\end{remark}

Later, similarly we will define a norm $||~||_{W^{1,1}}$ which will work as
a Sobolev norm for these paths (see Definition \ref{W111}).

\subsection{Transfer operator associated to $F$}

Let us now consider the transfer operator $L_{F}$ associated with $F$, i.e.
such that 
\begin{equation*}
\lbrack L_{F}\mu ](E)=\mu (F^{-1}(E))
\end{equation*}%
for each signed measure $\mu $ on $X$ and for each measurable set $E\subset
X $. Being a pushforward map, the same function can be also denoted by $%
F^{\ast }$ we will use this notation sometime. There is a nice
characterization of the transfer operator in our case, which makes it work
quite like a one dimensional transfer operator. For the proof see \cite{GLu}
.

\begin{proposition}
For a given leaf $\gamma \in \mathcal{F}^{s}$, define the map $F_{\gamma
}:D_{2}\longrightarrow D_{2}$ by 
\begin{equation*}
F_{\gamma }=\pi _{y}\circ F|_{\gamma }\circ \pi _{\gamma ,y}^{-1}.
\end{equation*}%
For all $\mu \in \mathcal{L}^{1}$ and for almost all $\gamma \in S^{1}$
holds 
\begin{equation}
(L_{F}\mu )|_{\gamma }=\sum_{i=1}^{q}{\dfrac{F_{T_{i}^{-1}(\gamma )}^{\ast
}\mu |_{T_{i}^{-1}(\gamma )}}{|T_{i}^{^{\prime }}\circ T_{i}^{-1}(\gamma ))|}%
}\ \ \mathnormal{for~almost~all}\ \ \gamma \in S^{1}.  \label{niceformulaa}
\end{equation}%
\label{niceformulaab}
\end{proposition}

\subsection{General properties of $L^{1}$ like norms}

\begin{remark}
\label{nice}If \ $F$ is a weak contraction $:X\rightarrow X$, where $X$ is a
metric space, for every Borel measure with sign $\mu $ it holds%
\begin{equation*}
||L_{F}\mu ||_{W}\leq ||\mu ||_{W}.
\end{equation*}%
Indeed, since $F$ is a contraction, if $|g|_{\infty }\leq 1$ and $Lip(g)\leq
1$ the same holds for $g\circ F$. Then%
\begin{eqnarray*}
\left\vert \int {g~}d(L_{F}(\mu ))\right\vert &=&\left\vert \int g\circ F{~}%
d\mu \right\vert \\
&\leq &\left\vert \left\vert \mu \right\vert \right\vert _{W}.
\end{eqnarray*}%
Taking the supremum over $|g|_{\infty }\leq 1$ and $Lip(g)\leq 1$ we finish
the proof of the inequality.
\end{remark}

\begin{proposition}[The weak norm is weakly contracted by $L_{F}$]
If $\mu \in \mathcal{L}^{1}$ then 
\begin{equation}
||L_{F}\mu ||_{"1"}\leq ||\mu ||_{"1"}.
\end{equation}%
\label{weakcontral11234}
\end{proposition}

\begin{proof}
In the following we consider, for all $i$, the change of variable $\gamma
=T_{i}(\alpha )$. Thus by Remark \ref{nice} and equation (\ref{niceformulaa}%
), we have

\begin{eqnarray*}
||L_{F}\mu ||_{"1"} &=&\int_{N_{1}}{\ ||(L_{F}\mu )|_{\gamma }||_{W}}%
dm_{1}(\gamma ) \\
&\leq &\sum_{i=1}^{q}{\int_{T(\eta _{i})}{\ \left\vert \left\vert \dfrac{%
\func{F}_{T_{i}^{-1}(\gamma )}^{\ast }\mu |_{T_{i}^{-1}(\gamma )}}{%
|T_{i}^{\prime }(T_{i}^{-1}(\gamma ))|}\right\vert \right\vert _{W}}%
dm_{1}(\gamma )} \\
&=&\sum_{i=1}^{q}{\int_{\eta _{i}}{\left\vert \left\vert \func{F}_{\alpha
}^{\ast }\mu |_{\alpha }\right\vert \right\vert _{W}}dm_{1}(\alpha )} \\
&=&\sum_{i=1}^{q}{\int_{\eta _{i}}{\left\vert \left\vert \mu |_{\alpha
}\right\vert \right\vert _{W}}}dm_{1}(\alpha ) \\
&=&||\mu ||_{"1"}.
\end{eqnarray*}
\end{proof}

\subsubsection{Convergence to equilibrium}

Now we prove that $F$ in some sense has exponential convergence to
equilibrium.

\begin{proposition}
\label{5.6} For all signed measure $\mu \in \mathcal{L}^{1}$ it holds 
\begin{equation}
||\func{F}^{\ast }\mu ||_{"1"}\leq \alpha ||\mu ||_{"1"}+(\alpha +1)||\phi
_{x}||_{1}.  \label{abovv}
\end{equation}
\end{proposition}

\begin{proof}
Consider a signed measure $\mu \in \mathcal{L}^{1}$ and its restriction on
the leaf $\gamma $, $\mu |_{\gamma }=\pi _{\gamma ,y}^{\ast }(\phi
_{x}(\gamma )\mu _{\gamma })$. Set%
\begin{equation*}
\overline{\mu }|_{\gamma }=\pi _{\gamma ,y}^{\ast }\mu _{\gamma }.
\end{equation*}%
If $\mu $ is a positive measure then $\overline{\mu }|_{\gamma }$ is a
probability on $D^{2}$. Moreover $\mu |_{\gamma }=\phi _{x}(\gamma )%
\overline{\mu }|_{\gamma }$.

By the above comments and the expression given by remark \ref{niceformulaab}
we have

\begin{eqnarray*}
||\func{F}^{\ast }\mu ||_{"1"} &\leq &\sum_{i=1}^{q}{\ \int_{I}{\ \left\vert
\left\vert \frac{\func{F}_{T_{i}^{-1}(\gamma )}^{\ast }\overline{\mu ^{+}}%
|_{T_{i}^{-1}(\gamma )}\phi _{x}^{+}(T_{i}^{-1}(\gamma ))}{|T_{i}^{\prime
}|\circ T_{i}^{-1}(\gamma )}-\frac{\func{F}_{T_{i}^{-1}(\gamma )}^{\ast }%
\overline{\mu ^{-}}|_{T_{i}^{-1}(\gamma )}\phi _{x}^{-}(T_{i}^{-1}(\gamma ))%
}{|T_{i}^{\prime }|\circ T_{i}^{-1}(\gamma )}\right\vert \right\vert _{W}}%
dm_{1}(\gamma )} \\
&\leq &\sum_{i=1}^{q}{\ \int_{I}{\ \left\vert \left\vert \frac{\func{F}%
_{T_{i}^{-1}(\gamma )}^{\ast }\overline{\mu ^{+}}|_{T_{i}^{-1}(\gamma )}\phi
_{x}^{+}(T_{i}^{-1}(\gamma ))}{|T_{i}^{\prime }|\circ T_{i}^{-1}(\gamma )}-%
\frac{\func{F}_{T_{i}^{-1}(\gamma )}^{\ast }\overline{\mu ^{+}}%
|_{T_{i}^{-1}(\gamma )}\phi _{x}^{-}(T_{i}^{-1}(\gamma ))}{|T_{i}^{\prime
}|\circ T_{i}^{-1}(\gamma )}\right\vert \right\vert _{W}}dm_{1}(\gamma )} \\
&+&\sum_{i=1}^{q}{\ \int_{I}{\ \left\vert \left\vert \frac{\func{F}%
_{T_{i}^{-1}(\gamma )}^{\ast }\overline{\mu ^{+}}|_{T_{i}^{-1}(\gamma )}\phi
_{x}^{-}(T_{i}^{-1}(\gamma ))}{|T_{i}^{\prime }|\circ T_{i}^{-1}(\gamma )}-%
\frac{\func{F}_{T_{i}^{-1}(\gamma )}^{\ast }\overline{\mu ^{-}}%
|_{T_{i}^{-1}(\gamma )}\phi _{x}^{-}(T_{i}^{-1}(\gamma ))}{|T_{i}^{\prime
}|\circ T_{i}^{-1}(\gamma )}\right\vert \right\vert _{W}}dm_{1}(\gamma )} \\
&=&\func{I}_{1}+\func{I}_{2}
\end{eqnarray*}%
where

\begin{equation*}
\func{I}_{1}=\sum_{i=1}^{q}{\ \int {\ }}_{I}{{\left\vert \left\vert \frac{%
\func{F}_{T_{i}^{-1}(\gamma )}^{\ast }\overline{\mu ^{+}}|_{T_{i}^{-1}(%
\gamma )}\phi _{x}^{+}(T_{i}^{-1}(\gamma ))}{|T_{i}^{\prime }|\circ
T_{i}^{-1}(\gamma )}-\frac{\func{F}_{T_{i}^{-1}(\gamma )}^{\ast }\overline{%
\mu ^{+}}|_{T_{i}^{-1}(\gamma )}\phi _{x}^{-}(T_{i}^{-1}(\gamma ))}{%
|T_{i}^{\prime }|\circ T_{i}^{-1}(\gamma )}\right\vert \right\vert _{W}}%
dm_{1}(\gamma )}
\end{equation*}%
and

\begin{equation*}
\func{I}_{2}=\sum_{i=1}^{q}{\ \int {\ }}_{I}{{\left\vert \left\vert \frac{%
\func{F}_{T_{i}^{-1}(\gamma )}^{\ast }\overline{\mu ^{+}}|_{T_{i}^{-1}(%
\gamma )}\phi _{x}^{-}(T_{i}^{-1}(\gamma ))}{|T_{i}^{\prime }|\circ
T_{i}^{-1}(\gamma )}-\frac{\func{F}_{T_{i}^{-1}(\gamma )}^{\ast }\overline{%
\mu ^{-}}|_{T_{i}^{-1}(\gamma )}\phi _{x}^{-}(T_{i}^{-1}(\gamma ))}{%
|T_{i}^{\prime }|\circ T_{i}^{-1}(\gamma )}\right\vert \right\vert _{W}}%
dm_{1}(\gamma )}.
\end{equation*}%
Let us estimate $\func{I}_{1}$ and $\func{I}_{2}$.

By Lemma \ref{unamis} and a change of variable we have 
\begin{eqnarray*}
\func{I}_{1} &=&\sum_{i=1}^{q}{\ \int_{I}{\ \left\vert \left\vert \func{F}%
_{T_{i}^{-1}(\gamma )}^{\ast }\overline{\mu ^{+}}|_{T_{i}^{-1}(\gamma
)}\right\vert \right\vert _{W}\frac{|\phi _{x}^{+}-\phi _{x}^{-}|}{%
|T_{i}^{\prime }|}\circ T_{i}^{-1}(\gamma )}dm_{1}(\gamma )} \\
&=&\int_{I}{\ \left\vert \left\vert \func{F}_{\beta }^{\ast }\overline{\mu
^{+}}|_{\beta }\right\vert \right\vert _{W}|\phi _{x}^{+}-\phi
_{x}^{-}|(\beta )}dm_{1}(\beta ) \\
&=&\int_{I}{\ |\phi _{x}^{+}-\phi _{x}^{-}|(\beta )}dm_{1}(\beta ) \\
&=&||\phi _{x}||_{1}
\end{eqnarray*}%
and by Lemma \ref{quasicontract} we have

\begin{eqnarray*}
\func{I}_{2} &=&\sum_{i=1}^{q}{\ \int_{I}{\ \left\vert \left\vert \func{F}%
_{T_{i}^{-1}(\gamma )}^{\ast }\left( \overline{\mu ^{+}}|_{T_{i}^{-1}(\gamma
)}-\overline{\mu ^{-}}|_{T_{i}^{-1}(\gamma )}\right) \right\vert \right\vert
_{W}\frac{\phi _{x}^{-}}{|T_{i}^{\prime }|}\circ T_{i}^{-1}(\gamma )}%
dm_{1}(\gamma )} \\
&\leq &\sum_{i=1}^{q}{\ \int_{I_{i}}{\ \left\vert \left\vert \func{F}_{\beta
}^{\ast }\left( \overline{\mu ^{+}}|_{\beta }-\overline{\mu ^{-}}|_{\beta
}\right) \right\vert \right\vert _{W}\phi _{x}^{-}(\beta )}dm_{1}(\beta )} \\
&\leq &\alpha \int_{I}{\ \left\vert \left\vert \overline{\mu ^{+}}|_{\beta }-%
\overline{\mu ^{-}}|_{\beta }\right\vert \right\vert _{W}\phi _{x}^{-}(\beta
)}dm_{1}(\beta ) \\
&\leq &\alpha \int_{I}{\ \left\vert \left\vert \overline{\mu ^{+}}|_{\beta
}\phi _{x}^{-}(\beta )-\overline{\mu ^{+}}|_{\beta }\phi _{x}^{+}(\beta
)\right\vert \right\vert _{W}}dm_{1}(\beta ) \\
&\leq &\alpha \int_{I}{\ \left\vert \left\vert \overline{\mu ^{+}}|_{\beta
}\phi _{x}^{-}(\beta )-\overline{\mu ^{+}}|_{\beta }\phi _{x}^{+}(\beta
)\right\vert \right\vert _{W}}dm_{1}(\beta )+\alpha \int_{I}{\ \left\vert
\left\vert \overline{\mu ^{+}}|_{\beta }\phi _{x}^{+}(\beta )-\overline{\mu
^{-}}|_{\beta }\phi _{x}^{-}(\beta )\right\vert \right\vert _{W}}%
dm_{1}(\beta ) \\
&=&\alpha ||\phi _{x}||_{1}+\alpha ||\mu ||_{"1"}.
\end{eqnarray*}%
Summing the above estimates we finish the proof.
\end{proof}

Iterating (\ref{abovv}) we get the following corollary.

\begin{corollary}
\label{nicecoro} 
\begin{equation*}
||L_{F}^{n}\mu ||_{"1"}\leq \alpha ^{n}||\mu ||_{"1"}+\overline{\alpha }%
||\phi _{x}||_{1},
\end{equation*}%
where $\overline{\alpha }=\frac{1+\alpha }{1-\alpha }$.
\end{corollary}

Let us consider the set of zero average measures%
\begin{equation}
\mathcal{V}=\{\mu \in \mathcal{L}^{1}:\mu (X)=0\}.  \label{mathV}
\end{equation}

Since $\pi _{x}^{\ast }\mu =\phi _{x}m_{1}$ ($\phi _{x}=\phi _{x}^{+}-\phi
_{x}^{-}$) we have $\int {\phi _{x}}dm_{1}=0$. From the last corollary and
the convergence to equilibrium for expanding maps with respect to $L^{1}$
and $W^{1,1}$ norms (see Remarks \ref{conv1} and \ref{cnv2}) it directly
follows:

\begin{proposition}[Exponential convergence to equilibrium]
\label{5.8} There exist $D\in \mathbb{R}$ and $0<\beta _{1}<1$ such that,
for every signed measure $\mu \in \mathcal{V}$, it holds 
\begin{equation*}
||L_{F}^{n}\mu ||_{"1"}\leq D_{2}\beta _{1}^{n}(||\mu ||_{1}+||\phi
_{x}||_{_{W^{1,1}}})
\end{equation*}%
for all $n\geq 1$. \ 
\end{proposition}

We now prove the existence of an invariant measure for the solenoidal system
in the set $\mathcal{L}^{1}$ it is not difficult to deduce that this should
be a physical measure (points in the same stable leaves must have the same
long time average for a continuous observable).

\begin{proposition}
There is a unique $\mu \in \mathcal{L}^{1}$ such that $L_{F}\mu =\mu $.
\end{proposition}

\begin{proof}
The base map $T$ is expanding and has an absolutely continuous invariant
measure. Let us call it $\varphi _{x}$. Consider the measure $\nu =$ $%
\varphi _{x}\times m$ (the measure having marginal $\varphi _{x}$ and
Lebesgue measure on the sable leaves) and the sequence $\nu
_{n}=L_{F}^{n}\nu .$ By Proposition \ref{5.8} $||\nu _{n}-\nu
_{m}||_{"1"}\leq D_{2}\beta _{1}^{n}$ and $\nu _{n}$ is a Cauchy sequence.
By passing to a subsequence we can find $\nu _{n_{k}}$ such that for almost
each leaf $\gamma $, $\nu _{n_{k}}|_{\gamma }$ is a Cauchy sequence for the $%
||~||_{W}$ norm. By Remark \ref{posi} this sequence must have a limit which
is a positive measure. This defines a limit measure $\mu $ which is
invariant. The integrability of $||\nu _{\gamma }||_{W}$ \ follows by the
fact that $||\nu _{n_{k}}|_{\gamma }||_{W}\leq \sup \varphi _{x}$and thus it
is a bounded sequence and \ $||\nu _{n_{k}}|_{\gamma }||_{W}$ converges
pointwise to $||\nu _{\gamma }||_{W}.$\newline
The uniqueness follow trivially by Proposition \ref{5.8}.
\end{proof}

\subsection{Strong norm and Lasota Yorke inequality}

We give here an example of a strong space satisfying a kind of Lasota Yorke
inequality which holds for positive measures.

Given $\mu \in \mathcal{L}^{1}$let us denote by ${{\phi _{\mu }}}$ its
marginal density. Let us consider the following space of measures

\begin{equation*}
"W^{1,1}"=\left\{ 
\begin{array}{c}
\mu \in \mathcal{L}^{1}:{{\phi _{\mu }\in W}}^{1,1}~\forall \gamma
_{1}~\lim_{\gamma _{2}\rightarrow \gamma _{1}}||\mu |_{\gamma _{2}}-\mu
|_{\gamma _{1}}||_{W}=0~and~ \\ 
for~almost~all~\gamma _{1},~D(\mu ,\gamma _{1}):=\limsup_{\gamma
_{2}\rightarrow \gamma _{1}}||\frac{\mu |_{\gamma _{2}}-\mu |_{\gamma _{1}}}{%
\gamma _{2}-\gamma _{1}}||_{W}~<\infty%
\end{array}%
\right\}
\end{equation*}

\begin{definition}
\label{W111}Let us consider the norm%
\begin{equation*}
||\mu ||_{"W^{1,1}"}:=||\mu ||_{"1"}+\int |D(\mu ,\gamma )|d\gamma .
\end{equation*}
\end{definition}

This norm will play the role of the strong norm in the solenoid case. Indeed
the following Lasota-Yorke-like inequality can be proved

\begin{proposition}
\label{LYYY}Let $F$ be a solenoidal map, then.$L_{F}"W^{1,1}"\subseteq
"W^{1,1}"$and there are $\lambda <1,B>0$ s.t $\forall \mu \in "W^{1,1}"$
such that $\mu \geq 0$.%
\begin{equation*}
||L_{F}\mu ||_{"W^{1,1}"}\leq \lambda ({{\alpha ||\mu ||_{"W^{1,1}"}+||\phi }%
^{\prime }{_{\mu }||}}_{1}{)+}B||\mu ||_{"1"}.
\end{equation*}
\end{proposition}

\begin{proof}
Since the map is $C^{2}$ it is obvious that $\lim_{\gamma _{2}\rightarrow
\gamma _{1}}||\mu |_{\gamma _{2}}-\mu |_{\gamma _{1}}||_{W}=0$ and ${{\phi
_{L_{T}(\mu )}\in W}}^{1,1}$. Let us estimate%
\begin{equation*}
||D(L_{F}\mu ,\gamma _{1})||_{1}=\int \limsup_{\gamma _{2}\rightarrow \gamma
_{1}}||\frac{(L_{F}\mu )|_{\gamma _{1}}-(L_{F}\mu )|_{\gamma _{2}}}{\gamma
_{2}-\gamma _{1}}||_{W}~dm.
\end{equation*}

By Equation \ref{niceformulaa} we have%
\begin{equation}
(L_{F}\mu )|_{\gamma }=\sum_{i=1}^{q}{\dfrac{\func{F}_{T_{i}^{-1}(\gamma
)}^{\ast }\mu |_{T_{i}^{-1}(\gamma )}}{|T_{i}^{^{\prime }}\circ
T_{i}^{-1}(\gamma ))|}}\ \ \mathnormal{for~almost~all}\ \gamma \in N_{1}.
\end{equation}

Then%
\begin{eqnarray*}
||D(L_{F}\mu ,\gamma _{1})||_{1} &\leq &\sum_{i=1}^{q}\int \limsup_{\gamma
_{2}\rightarrow \gamma _{1}}||\frac{1}{\gamma _{2}-\gamma _{1}}({\dfrac{%
\func{F}_{T_{i}^{-1}(\gamma _{1})}^{\ast }\mu |_{T_{i}^{-1}(\gamma _{1})}}{%
|T_{i}^{^{\prime }}\circ T_{i}^{-1}(\gamma _{1}))|}-\dfrac{\func{F}%
_{T_{i}^{-1}(\gamma _{2})}^{\ast }\mu |_{T_{i}^{-1}(\gamma _{2})}}{%
|T_{i}^{^{\prime }}\circ T_{i}^{-1}(\gamma _{2}))|})||}_{W}{~dm(}\gamma _{1})
\\
&{\leq }&\sum_{i=1}^{q}\int \limsup_{\gamma _{2}\rightarrow \gamma _{1}}%
\frac{1}{\gamma _{2}-\gamma _{1}}{||\dfrac{\func{F}_{T_{i}^{-1}(\gamma
_{1})}^{\ast }\mu |_{T_{i}^{-1}(\gamma _{1})}-\func{F}_{T_{i}^{-1}(\gamma
_{2})}^{\ast }\mu |_{T_{i}^{-1}(\gamma _{2})}}{|T_{i}^{^{\prime }}\circ
T_{i}^{-1}(\gamma _{1}))|}{||}_{W}~dm} \\
&&+|\int \limsup_{\gamma _{2}\rightarrow \gamma _{1}}\frac{1}{\gamma
_{2}-\gamma _{1}}||\func{F}_{T_{i}^{-1}(\gamma _{2})}^{\ast }\mu
|_{T_{i}^{-1}(\gamma _{2})}(\frac{1}{|T_{i}^{^{\prime }}\circ
T_{i}^{-1}(\gamma _{1}))|}-\frac{1}{|T_{i}^{^{\prime }}\circ
T_{i}^{-1}(\gamma _{2}))|}){||}_{W}dm
\end{eqnarray*}%
\begin{eqnarray*}
&{\leq }&\sum_{i=1}^{q}\int \frac{1}{|T_{i}^{^{\prime }}\circ
T_{i}^{-1}(\gamma _{1}))|}\limsup_{\gamma _{2}\rightarrow \gamma _{1}}{||}%
\frac{\func{F}_{T_{i}^{-1}(\gamma _{1})}^{\ast }\mu |_{T_{i}^{-1}(\gamma
_{1})}-\func{F}_{T_{i}^{-1}(\gamma _{2})}^{\ast }\mu |_{T_{i}^{-1}(\gamma
_{2})}}{\gamma _{2}-\gamma _{1}}{||}_{W}{dm} \\
&&+\int {||}\func{F}_{T_{i}^{-1}(\gamma _{1})}^{\ast }\mu
|_{T_{i}^{-1}(\gamma _{1})}{||}_{W}\limsup_{\gamma _{2}\rightarrow \gamma
_{1}}\frac{1}{\gamma _{2}-\gamma _{1}}(\frac{1}{|T_{i}^{^{\prime }}\circ
T_{i}^{-1}(\gamma _{2}))|}-\frac{1}{|T_{i}^{^{\prime }}\circ
T_{i}^{-1}(\gamma _{1}))|})|dm
\end{eqnarray*}

Hence%
\begin{eqnarray*}
||D(L_{F}\mu ,\gamma _{1})||_{1} &{\leq }&\sum_{i=1}^{q}\int \frac{1}{%
|T_{i}^{^{\prime }}\circ T_{i}^{-1}(\gamma _{1}))|}\limsup_{\gamma
_{2}\rightarrow \gamma _{1}}{||}\frac{\func{F}_{T_{i}^{-1}(\gamma
_{1})}^{\ast }\mu |_{T_{i}^{-1}(\gamma _{1})}-\func{F}_{T_{i}^{-1}(\gamma
_{1})}^{\ast }\mu |_{T_{i}^{-1}(\gamma _{2})}}{\gamma _{2}-\gamma _{1}}{||}%
_{W} \\
&&{+\frac{1}{|T_{i}^{^{\prime }}\circ T_{i}^{-1}(\gamma _{1}))|}%
\limsup_{\gamma _{2}\rightarrow \gamma _{1}}||}\frac{\func{F}%
_{T_{i}^{-1}(\gamma _{1})}^{\ast }\mu |_{T_{i}^{-1}(\gamma _{2})}-\func{F}%
_{T_{i}^{-1}(\gamma _{2})}^{\ast }\mu |_{T_{i}^{-1}(\gamma _{2})}}{\gamma
_{2}-\gamma _{1}}{||}~dm \\
&&+\int {||}\func{F}_{T_{i}^{-1}(\gamma _{1})}^{\ast }\mu
|_{T_{i}^{-1}(\gamma _{1})}{||}_{W}\limsup_{\gamma _{2}\rightarrow \gamma
_{1}}\frac{1}{\gamma _{2}-\gamma _{1}}(\frac{1}{|T_{i}^{^{\prime }}\circ
T_{i}^{-1}(\gamma _{2}))|}-\frac{1}{|T_{i}^{^{\prime }}\circ
T_{i}^{-1}(\gamma _{1}))|})|dm \\
&=&I+II+III
\end{eqnarray*}%
Let us estimate the first summand%
\begin{equation*}
{I=}\sum_{i=1}^{q}\int \frac{1}{|T_{i}^{^{\prime }}\circ T_{i}^{-1}(\gamma
_{1}))|}\limsup_{\gamma _{2}\rightarrow \gamma _{1}}{||}\frac{\func{F}%
_{T_{i}^{-1}(\gamma _{1})}^{\ast }\left( \mu |_{T_{i}^{-1}(\gamma _{1})}-\mu
|_{T_{i}^{-1}(\gamma _{2})}\right) }{\gamma _{2}-\gamma _{1}}{||}_{W}
\end{equation*}%
We recall that by Lemma \ref{unamis} $\ ||\func{F}_{\gamma }^{\ast }\mu
||_{W}\leq \alpha ||\mu ||_{W}+\mu (D^{2})$ then%
\begin{equation*}
I{\leq }\sum_{i=1}^{q}\int \frac{1}{|T_{i}^{^{\prime }}\circ
T_{i}^{-1}(\gamma _{1}))|}\limsup_{\gamma _{2}\rightarrow \gamma
_{1}}\left\{ \alpha {||}\frac{\left( \mu |_{T_{i}^{-1}(\gamma _{1})}-\mu
|_{T_{i}^{-1}(\gamma _{2})}\right) }{\gamma _{2}-\gamma _{1}}{||}_{W}{+|%
\frac{\left( \mu |_{T_{i}^{-1}(\gamma _{1})}(D^{2})-\mu |_{T_{i}^{-1}(\gamma
_{2})}(D^{2})\right) }{\gamma _{2}-\gamma _{1}}{|}}\right\} {{d\gamma }_{1}}
\end{equation*}

and%
\begin{eqnarray*}
I &{\leq }&\sum_{i=1}^{q}\int \frac{1}{|T_{i}^{^{\prime }}\circ
T_{i}^{-1}(\gamma _{1}))|}[\limsup_{\gamma _{2}\rightarrow \gamma _{1}}\frac{%
T_{i}^{-1}(\gamma _{2})-T_{i}^{-1}(\gamma _{1})}{\gamma _{2}-\gamma _{1}}%
\limsup_{\gamma _{2}\rightarrow \gamma _{1}}\alpha {||}\frac{\left( \mu
|_{T_{i}^{-1}(\gamma _{1})}-\mu |_{T_{i}^{-1}(\gamma _{2})}\right) }{%
T_{i}^{-1}(\gamma _{2})-T_{i}^{-1}(\gamma _{1})}{||}_{W} \\
&&{+\limsup_{\gamma _{2}\rightarrow \gamma _{1}}\frac{T_{i}^{-1}(\gamma
_{2})-T_{i}^{-1}(\gamma _{1})}{\gamma _{2}-\gamma _{1}}\limsup_{\gamma
_{2}\rightarrow \gamma _{1}}|\frac{\left( \mu |_{T_{i}^{-1}(\gamma
_{1})}(D^{2})-\mu |_{T_{i}^{-1}(\gamma _{2})}(D^{2})\right) }{%
T_{i}^{-1}(\gamma _{2})-T_{i}^{-1}(\gamma _{1})}}|]{{d\gamma }_{1}} \\
&{\leq }&\sum_{i=1}^{q}{\sup_{\gamma _{1}}}\limsup_{\gamma _{2}\rightarrow
\gamma _{1}}\frac{T_{i}^{-1}(\gamma _{2})-T_{i}^{-1}(\gamma _{1})}{\gamma
_{2}-\gamma _{1}}\int_{I_{i}}\alpha D(\mu ,T_{i}^{-1}(\gamma _{1}))+|\phi
_{x}^{\prime }(T_{i}^{-1}(\gamma _{1}))|dT_{i}^{-1}(\gamma _{1}).
\end{eqnarray*}

Then summing the contributions from all branches $T_{i}$ and intervals $%
I_{i} $%
\begin{eqnarray*}
I &{\leq }&{{\sup_{i,\gamma _{1}}}\limsup_{\gamma _{2}\rightarrow \gamma
_{1}}\frac{T_{i}^{-1}(\gamma _{2})-T_{i}^{-1}(\gamma _{1})}{\gamma
_{2}-\gamma _{1}}{[\alpha ||D(\mu ,\gamma _{1})||_{1}+||\phi }^{\prime }{%
_{\mu }||_{1}}]} \\
&\leq &\frac{1}{\inf |T^{\prime }|}{{[\alpha ||D(\mu ,\gamma
_{1})||_{1}+||\phi }^{\prime }{_{\mu }||_{1}}].}
\end{eqnarray*}

Now the other summands,%
\begin{gather*}
II\leq \sum_{i}\int_{I}\frac{1}{|T_{i}^{^{\prime }}\circ T_{i}^{-1}(\gamma
_{1}))|}\limsup_{\gamma _{2}\rightarrow \gamma _{1}}{||}\frac{\func{F}%
_{T_{i}^{-1}(\gamma _{1})}^{\ast }\mu |_{T_{i}^{-1}(\gamma _{2})}-\func{F}%
_{T_{i}^{-1}(\gamma _{2})}^{\ast }\mu |_{T_{i}^{-1}(\gamma _{2})}}{\gamma
_{2}-\gamma _{1}}{||}_{W}{d\gamma }_{1} \\
\leq {{\sup_{i,\gamma _{1}}}\limsup_{\gamma _{2}\rightarrow \gamma _{1}}%
\frac{T_{i}^{-1}(\gamma _{2})-T_{i}^{-1}(\gamma _{1})}{\gamma _{2}-\gamma
_{1}}} \\
{\times }\sum_{i}\int_{I}\frac{1}{|T_{i}^{^{\prime }}\circ T_{i}^{-1}(\gamma
_{1}))|}\limsup_{\gamma _{2}\rightarrow \gamma _{1}}{||}\frac{\func{F}%
_{T_{i}^{-1}(\gamma _{1})}^{\ast }\mu |_{T_{i}^{-1}(\gamma _{2})}-\func{F}%
_{T_{i}^{-1}(\gamma _{2})}^{\ast }\mu |_{T_{i}^{-1}(\gamma _{2})}}{%
T_{i}^{-1}(\gamma _{2})-T_{i}^{-1}(\gamma _{1})}{||}_{W}{d\gamma }_{1} \\
\leq {{\sup_{i,\gamma _{1}}}\limsup_{\gamma _{2}\rightarrow \gamma _{1}}%
\frac{T_{i}^{-1}(\gamma _{2})-T_{i}^{-1}(\gamma _{1})}{\gamma _{2}-\gamma
_{1}}}\sum_{i}\int_{I_{i}}\limsup_{\gamma _{2}\rightarrow \gamma _{1}}{||}%
\frac{\func{F}_{T_{i}^{-1}(\gamma _{1})}^{\ast }\mu |_{T_{i}^{-1}(\gamma
_{2})}-\func{F}_{T_{i}^{-1}(\gamma _{2})}^{\ast }\mu |_{T_{i}^{-1}(\gamma
_{2})}}{T_{i}^{-1}(\gamma _{2})-T_{i}^{-1}(\gamma _{1})}{||}_{W}{d}%
T_{i}^{-1}(\gamma _{1}) \\
\leq \frac{1}{\inf |T^{\prime }|}||\frac{\partial G}{\partial x}||_{\infty
}||\mu ||_{1}.
\end{gather*}%
where in the last step we used that $\mu \geq 0$. Finally 
\begin{eqnarray*}
III &\leq &\sum_{i}\int {||}\func{F}_{T_{i}^{-1}(\gamma _{1})}^{\ast }\mu
|_{T_{i}^{-1}(\gamma _{1})}{||}_{W}\limsup_{\gamma _{2}\rightarrow \gamma
_{1}}\frac{1}{\gamma _{2}-\gamma _{1}}(\frac{1}{|T_{i}^{^{\prime }}\circ
T_{i}^{-1}(\gamma _{2}))|}-\frac{1}{|T_{i}^{^{\prime }}\circ
T_{i}^{-1}(\gamma _{1}))|})|dm(\gamma _{1}) \\
&\leq &||T^{\prime }||_{\infty }||\frac{T^{\prime \prime }}{(T^{\prime })^{2}%
}||_{\infty }||\mu ||_{1}.
\end{eqnarray*}

Summarizing%
\begin{equation}
||L_{F}\mu ||_{W^{1,1}}\leq \frac{1}{\inf |T^{\prime }|}({{\alpha ||\mu
||_{"W^{1,1}"}+||\phi }^{\prime }{_{\mu }||}_{1}+}||\frac{\partial G}{%
\partial x}||_{\infty }||\mu ||_{1})+(1+||T^{\prime }||_{\infty }||\frac{%
T^{\prime \prime }}{(T^{\prime })^{2}}||_{\infty })||\mu ||_{1}.
\label{final}
\end{equation}
\end{proof}

As done before, iterating the inequality, gives

\begin{corollary}
\label{LYYYYYY}There are $B>0,\lambda <1$ such that%
\begin{equation*}
||L_{F}^{n}\mu ||_{"W^{1,1}"}\leq \lambda ^{n}({{||\mu ||_{"W^{1,1}"}+||\phi 
}^{\prime }{_{\mu }||}_{1})+}B||\mu ||_{1}.
\end{equation*}
\end{corollary}


This inequality shows that the iteration of a positive measure keeps a
bounded regularity in this strong norm.

\section{Bibliographic remarks and acknowledgements}

The books \cite{B}, \cite{gora}, \cite{Viana}, \cite{DKL} contains detailed
introductions to subjects similar to the one considered in these notes (from
different points of view). Our point of view is closer to the one given in
the notes \cite{L2} and \cite{S}, from which some definition and statement
is taken. We recommend to consult all these texts for many related topics,
generalizations, applications and complements we cannot include in these
short introductory notes.

The author whish to thank C. Liverani, J. Sedro, R. Possobon and I. Nisoli
for useful conversations and warmly thank the students of the Ergodic Theory
course of "Laurea magistrale in matematica", University of Pisa, year 2015,
and in particular A. Pigati for useful comments and corrections to the
earlier versions of this text.

\end{document}